\documentclass[11pt, leqno]{article}

\usepackage[margin=25mm]{geometry}
\usepackage{amsmath, amssymb, amsthm, amscd, mathrsfs}
\usepackage[all]{xy}
\usepackage{enumerate}

\makeatletter
\@addtoreset{equation}{section}
\makeatother

\theoremstyle{theorem}
\newtheorem{theorem}{Theorem}[subsection]
\newtheorem{definition}[theorem]{Definition}
\newtheorem{proposition}[theorem]{Proposition}
\newtheorem{lemma}[theorem]{Lemma}
\newtheorem{corollary}[theorem]{Corollary}

\newtheorem{hypothesis}[theorem]{Hypothesis}
\newtheorem{apptheorem}{Theorem}[section]
\newtheorem{appdefinition}[apptheorem]{Definition}
\newtheorem{appproposition}[apptheorem]{Proposition}
\newtheorem{applemma}[apptheorem]{Lemma}

\theoremstyle{definition}
\newtheorem{remark}[theorem]{Remark}{\rm}
\newtheorem{example}[theorem]{Example}{\rm}
{\rm}
{\rm}
\newtheorem{_empty}[theorem]{}{\rm}
{\rm}

\newcommand{\ad}{{\rm ad}}
\newcommand{\an}{{\rm an}}
\newcommand{\Coh}{{\rm Coh}}
\newcommand{\Bf}{\mathfrak{B}}
\newcommand{\D}{\mathcal{D}}
\newcommand{\Df}{\mathfrak{D}}
\newcommand{\dlog}{{\rm dlog} \,}
\newcommand{\E}{\mathcal{E}}
\newcommand{\Ed}{\mathcal{E}^\dag}

\newcommand{\ex}{\mathrm{ex}}
\newcommand{\Ext}{{\rm Ext}}
\newcommand{\F}{\mathcal{F}}
\newcommand{\Frac}{{\rm Frac} \,}
\newcommand{\Homs}{\mathcal{H}\kern -.5pt om}
\newcommand{\id}{{\rm id}}
\newcommand{\Isoc}{{\rm Isoc}^\dag}

\newcommand{\LFS}{\mathrm{LFS}}
\newcommand{\LS}{\mathrm{LS}}
\newcommand{\LNM}{{\rm LNM}}
\newcommand{\M}{\mathcal{M}}

\newcommand{\Modfg}{{\rm Mod}^{\rm fg}}
\newcommand{\MIC}{{\rm MIC}}

\renewcommand{\O}{\mathcal{O}}
\newcommand{\oc}{\mathrm{oc}}
\newcommand{\Pf}{\mathfrak{P}}
\newcommand{\pLFS}{p\mathrm{LFS}}

\newcommand{\Qf}{\mathfrak{Q}}
\newcommand{\R}{\mathcal{R}}
\newcommand{\rig}{\mathrm{rig}}
\newcommand{\red}{{\rm red}}
\newcommand{\Sf}{\mathfrak{S}}
\newcommand{\Sh}{{\rm Sh}}
\newcommand{\Strat}{{\rm Strat}}
\newcommand{\Tf}{\mathfrak{T}}
\newcommand{\Uc}{\mathcal{U}}
\newcommand{\Us}{\mathscr{U}}
\newcommand{\Uf}{\mathfrak{U}}
\newcommand{\ULNM}{{\rm ULNM}}
\newcommand{\V}{\mathcal{V}}
\newcommand{\Vf}{\mathfrak{V}}
\newcommand{\Vs}{\mathscr{V}}
\newcommand{\Ws}{\mathscr{W}}
\newcommand{\Wf}{\mathfrak{W}}
\newcommand{\Xs}{\mathscr{X}}
\newcommand{\Xf}{\mathfrak{X}}
\newcommand{\Ys}{\mathscr{Y}}
\newcommand{\Yf}{\mathfrak{Y}}

\newcommand{\Zf}{\mathfrak{Z}}
\renewcommand{\sp}{{\rm sp}}
\newcommand{\Spa}{{\rm Spa}\,}
\newcommand{\Spd}{{\rm Sp}^\dag\,}
\newcommand{\Spec}{{\rm Spec}\,}
\newcommand{\Spf}{{\rm Spf}\,}

\newcommand{\coker}{{\rm coker}\,}

\newcommand{\ul}{\underline}
\newcommand{\ol}{\overline}
\newcommand{\wh}{\widehat}
\newcommand{\wt}{\widetilde}
\newcommand{\unu}{{\underline{\nu}}}

\title{Semistable Reduction Theorem for Overconvergent $F$-isocrystals over Laurent Series Fields}
\author{Yuanmin Liu}
\date{}

\begin{document}
	\maketitle
	\begin{abstract}
		We prove the semistable reduction theorem for $\mathcal{E}^{\dag}_K$-valued and $K$-valued overconvergent $F$-isocrystals over $k((t))$-varieties which were introduced by Lazda and P\'{a}l. As an application, we prove the  finite dimensionality of $\mathcal{E}^{\dag}_K$-valued rigid cohomology with compact support.
	\end{abstract}
	\tableofcontents
	
	\section{Introduction}
	Fix a field $k$ of characteristic $p > 0$ and a complete discrete valuation field $K$ of mixed characteristic whose residue field is $k$. Suppose for simplicity that $k$ is perfect. Let us recall the ``classical'' semistable reduction theorem proved by Kedlaya in the series of papers \cite{kedlaya2007semistable}\cite{kedlaya2008semistable}\cite{kedlaya2009semistable}\cite{kedlaya2011semistable}: for a closed immersion $Z \hookrightarrow X$ of algebraic varieties over $k$ with $X-Z$ smooth and an $F$-isocrystal $E$ over $X-Z$ overconvergent along $Z$, there is an alteration $f : X' \to X$ such that for $Z' := f^{-1}(Z)_\red$, $(X', Z')$ is a smooth pair and $f^*E$ extends to a locally free convergent log $F$-isocrystal over $(X', Z')$. This theorem has many important applications, especially in rigid cohomology and arithmetic $\D$-modules. Shiho \cite{shiho2002crystallineII} has shown that it implies finite dimensionality of rigid cohomology with coefficients, although it was proved in a more straightforward way in \cite{kedlaya2006finiteness}. Caro and Tsuzuki \cite{caro2012overholonomicity} proved that overconvergent $F$-isocrystals are overholonomic, which critically depends on the theorem. Furthermore Lazda \cite{lazda2023rigidification} proved that there is a canonical equivalence between the derived category of constructible isocrystals of Frobenius type and that of overholonomic complexes of arithmetic $\D$-modules of Frobenius type, which also uses the semistable reduction theorem in an essential way.
	
	Let us move to the case of varieties over $k((t))$. In the ``classical'' setting, one has to find a complete discrete valuation field whose residue field is $k((t))$, for which the Amice ring $\E_K$ stands. The main feature in this setting is that one can consider derivations by $t$ and hence ``absolute'' overconvergent isocrystals. For a $k((t))$-variety $X$ and an absolute overconvergent isocrystal $E$ on $X$, $H^i_\rig(X/\E_K; \, E)$ has the structure of $\nabla$-module over $\E_K$ induced by Gauss-Manin connection. Also, when $E$ has a Frobenius structure, these rigid cohomology groups are $(\sigma, \nabla)$-modules over $\E_K$. However, from the viewpoint of $p$-adic representations, those objects are not really ``geometric'' and they should be defined over the Robba ring $\R_K$ instead of $\E_K$. Indeed, $(\sigma, \nabla)$-modules over $\R_K$ are quasi-unipotent by $p$-adic local monodromy theorem.
	
	Lazda and P\'{a}l \cite{lazda2016rigid} defined rigid cohomology and overconvergent isocrystals valued over the bounded Robba ring $\Ed_K = \E_K \cap \R_K$ to resolve this problem. The idea is to embed a $k((t))$-variety into proper $k[[t]]$-schemes, not into proper $k((t))$-schemes. This slight modification seems subtle, but technically it causes some problems in proving the same results as the ``classical'' case; for example, it seems impossible to show finiteness of compactly supported $\Ed_K$-valued rigid cohomology with coefficients as in \cite{kedlaya2006finiteness}. Of course, one of the problems in $\Ed_K$-valued case is the semistable reduction theorem.
	
	The main idea of proving the semistable reduction theorem for $\Ed_K$-valued overconvergent $F$-isocrystals is to reduce the problem by invoking ``descent of unipotence'' to the classical case proved by Kedlaya. This fact is well-known in the context; it roughly states that for some sort of a smooth weakly complete algebra $A$ and for a $\nabla$-module $M$ over the Robba ring $\R_{A}$, $M$ is unipotent if and only if the base change $M \otimes_{\R_A} \R_{\wh{A}}$ to the Robba ring over the completion $\wh{A}$ is unipotent. This type of statement appears in \cite[Proposition 5.4.1]{kedlaya2006finiteness} and \cite[Proposition 3.42]{lazda2016rigid} for example. We generalize this statement in a more geometric way in \S\ref{subsec:descent}.
	
	In this paper, following \cite{kedlaya2007semistable} we prove that for a strictly semistable pair $(Y, Z)$ and an isocrystal $E$ on $Y-Z$ overconvergent along $Z$, $E$ extends to a log isocrystal along a component $Z_i$ of $Z$ if and only if $E$ has unipotent monodromy along $Z_i$.
	Locally choosing a lift $\Yf$ of $Y$, $]Z_i[_{\Yf}$ is isomorphic to the relative annulus $]Z_i[_{\Zf_i} \times A^1[0, 1)$ and $E$ induces some $\nabla$-module $\E_i$ over $]Z_i[_{\Zf_i} \times A^1(\eta, 1)$. If $E$ has unipotent monodromy along $Z_i$, i.e., $\E_i$ is unipotent, then it extends uniquely to a log $\nabla$-module over $]Z_i[_{\Zf_i} \times A^1[0, 1)$. We now apply descent of unipotence here to conclude that an $\Ed_K$-valued overconvergent isocrystal is log-extendable if and only if so is its completion $\wh{E}$, which is an $\E_K$-valued overconvergent isocrystal. The latter is already known to hold after alteration by Kedlaya, and the proof is almost done.
	
	The $K$-valued case needs additional consideration on components of the closed fiber. As explained in \cite[\S7]{kedlaya2007semistable}, it does not work to take an alteration so that after pullback the isocrystal has unipotent monodromy by using $p$-adic local monodromy theorem because the alteration may produce exceptional divisors on the source which do not dominate any components of the boundary on the target. In order to control the monodromy, we use some recent results of Kedlaya \cite{kedlaya2022monodromy}, in which he gives a new simpler proof of the semistable reduction theorem.
	
	\subsection{Structure of this paper}
	As we are going to prove propositions similar to those of \cite{kedlaya2007semistable}, numbering and naming of sections and subsections largely follow those of {\it loc. cit.} so that readers can easily find skipped proofs and arguments.
	
	In Section \ref{sec:rigid} we prepare some notations of rigid analytic geometry and define overconvergent isocrystals.
	
	In Section \ref{sec:log-nabla-mod} we develop the theory of log $\nabla$-modules over partial dagger spaces relative to $\Ed_K$ and check the properties of constant and unipotent objects. Also descent of unipotence is proved in this section.
	
	In Section \ref{sec:monodromy} we define the monodromy along the boundary, following \cite{kedlaya2007semistable}. The main consequence of descent of unipotence is that the constance/unipotence of monodromy in this setting can be checked after completing the isocrystal. 
	
	In Section \ref{sec:monodromy-and-extension} we prove the log extension theorem for $\Ed_K$-valued overconvergent isocrystals. The extension results for $\Ed_K$-valued overconvergent isocrystals are not necessary for main theorems, but the results are of interest of its own and we decided to include them following {\it loc. cit.}
	
	In Section \ref{sec:log-ext}, we define overconvergent log isocrystals and prove that they are equivalent to log $\nabla$-modules in good situations. Since for $\Ed_K$-valued isocrystals, it does not have derivations along components of the closed fiber, the extension terminates at the generic fiber $Y_\eta$ and hence gives overconvergent log isocrystals over the pair $(Y_\eta^\sharp, Y^\sharp).$ Then we finish the proof of semistable reduction theorem for $\Ed_K$-valued overconvergent $F$-isocrystals.
	
	In Section \ref{sec:ssr-rel-K}, we prove the semistable reduction theorem for $K$-valued overconvergent $F$-isocrystals, following the previous arguments to reduce the problem to monodromy along components of closed fibers. We invoke some results in \cite{kedlaya2022monodromy} to finish the proof.
	
	In Section \ref{sec:app}, as an application of the semistable reduction theorem proved in this paper, we prove finiteness of compactly supported rigid cohomology with coefficients in $\Ed_K$-valued overconvergent $F$-isocrystals, complementing \cite{lazda2016rigid}. 
	
	Appendices are dedicated to proving some auxiliary results.
	
	\subsection{Notations and conventions}
	\begin{itemize}
		\item For $\underline{\nu} = (\nu_1, \dots, \nu_n),$ we often use the multi-index notations
		$\underline{x}^{\underline{\nu}} = x_1^{\nu_1} \dots x_n^{\nu_n}$ and $\underline{\nu}! = {\nu_1}! \cdots {\nu_n}!. $ 
		\item $K$ is a complete discrete valuation field of characteristic $0$ whose residue field is $k$, $\V$ is the ring of integers of $K$ and $\varpi$ denotes its uniformizer. The norm $|\cdot|$ on $K$ or on its finite extensions are normalized so that $|p| = p^{-1}.$ We suppose that there is a isometric homomorphism $\sigma_K : K \to K$ which induces a $q = p^a$-th power Frobenius on $k$.
		\item Let $\Gamma^* = \{ p^r : r \in \mathbb{Q}\}$ and $\Gamma = \Gamma^* \cup \{0\}.$
		\item For a $\V$-module $M$, we will write $M_K = M \otimes_\V K.$
		\item For a Banach $K$-algebra $A$ and a product $I = I_1 \times \cdots \times I_n$ of $n$ subintervals of $[0, \infty),$ we write $A\langle x_1, \dots, x_n\rangle_I = \{ \sum_{\nu} a_\nu x_1^{\nu_1}\cdots x_n^{\nu_n} \in A[[x_1^{\pm 1}, \dots, x_n^{\pm 1}]] : \forall \eta \in I(\lim_{\nu} |a_\nu| \eta_1^{\nu_1} \cdots \eta_n^{\nu_n} = 0) \}.$ Here $\lim_{\nu} b_\nu = 0$ means that for any $\epsilon > 0,$ $-\epsilon < b_\nu < \epsilon$ except finitely many $\nu$. Instead of using the set $I$, we often write the conditions of $x_i$ directly in the subscript; for example, if $I = [0, 1] \times (p^{-1}, p], $ we write $A\langle x, y \rangle_{|x| \leq 1, \, p^{-1} < |y| \leq p}.$ When $I = [0, 1]^n$, we often omit the subscript $I$.
		\item $\displaystyle \E_K = \V[[t]]\langle t^{-1} \rangle_K = \{\sum_{n \in \mathbb{Z}} a_n t^n \in K[[t, t^{-1}]] : \lim_{n \to -\infty} |a_n| = 0, \, \sup_{n>0} |a_n| < \infty\}$ is the Amice ring.
		\item $\displaystyle \R_K = \{ \sum_{n \in \mathbb{Z}} a_n t^n \in K[[t, t^{-1}]] : \exists \eta \in (0, 1) (\lim_{n \to -\infty} |a_n|\eta^n = 0), \, \forall \rho \in (0, 1) (\lim_{n \to \infty} |a_n|\rho^n = 0) \}$ is the Robba ring.
		\item $\displaystyle \Ed_K = \E_K \cap \R_K = \displaystyle \{\sum_{n \in \mathbb{Z}} a_n t^n \in K[[t, t^{-1}]] :  \exists \eta \in (0, 1) (\lim_{n \to -\infty} |a_n|\eta^n = 0), \, \sup_{n>0} |a_n| < \infty\}$ is the bounded Robba ring.
		\item We extend $\sigma_K$ compatibly and continuously to $\E_K, \, \R_K$ and $\Ed_K$ so that $\sigma(t) = u t^q$ for some unit $u \in \V[[t]]_K$ congruent to $1$ modulo $\varpi$.
		\item For a scheme $S$, an $S$-variety means an integral, separated $S$-scheme of finite type.
		\item Formal schemes are locally noethrian throughout this paper. Until \S\ref{sec:log-ext}, they are moreover supposed to be $p$-adic and topologically of finite type over $\Spf \V[[t]]$ unless otherwise stated.
		\item Rigid analytic spaces or rigid analytic varieties are adic spaces adic over $\Spa K$ and locally strongly noetherian. They are moreover supposed to be taut\footnote{quasi-separated and closure of every quasi-compact open subset is again quasi-compact.} and locally of finite type over $\Spa \V[[t]]_K$ unless otherwise stated.
		\item For a rigid analytic space $\Xs$, the valuation at a point $x$ is denoted multiplicatively by $|\cdot|_x$. If the height of $x$ is $1$, i.e., the valuation group $|\cdot|_x$ can be embedded into $\mathbb{R}_{>0}$, we normalize it so that $|p|_x = p^{-1}.$ For $\eta = p^r \in \Gamma^*$ with $r = n/m \in \mathbb{Q}, \ n,m \in \mathbb{Z}, \ m > 0$ and $f \in \O_{\Xs, x}$, we often write inequalities such as $|f|_x < \eta$ instead of writing $|f|^m_x < |p|^{-n}$. When $x$ is of height $1$, these two conventions are compatible. Furthermore for a subinterval $I \subseteq [0, \infty)$, we often write $|f|_x \in I$ when it makes sense.
		\item $s$ and $\eta$ will respectively denote the closed point and generic point of $\Spec k[[t]]$. When we have finite extensions $k'[[t']]$ or $k_1[[t_1]]$ of $k[[t]]$, we often write without mentioning $s', \, \eta'$ or $s_1, \, \eta_1$ for their closed points and generic points.
	\end{itemize}
	
    \subsection*{Acknowledgement}
	The author is thankful for his supervisor Atsushi Shiho for helpful discussions, who proposed the idea of using descent of unipotence.
	This work was supported by the World-leading INnovative Graduate Study Program for Frontiers of Mathematical Sciences and Physics, The University of Tokyo.
	
	\section{Rigid analytic set up} \label{sec:rigid}
	\subsection{Germs, frames and tubes}
	Overconvergent isocrystals are locally coherent modules over dagger algebras equipped with an integrable connection satisfying overconvergent condition. Geometrically, for a frame $(X, Y, \Pf)$, it is a coherent $j^\dag_X \O_{]Y[_{\Pf}}$ with an integrable connections plus extra conditions. As in \cite{abe2022comparison} and \cite{abe2025proper}, it is convenient to introduce the concept of germs.
	
	\begin{definition}[cf. {\cite[Definition 2.2]{abe2025proper}}]
		A pregerm is a pair $(X, \Xs)$ of an adic space $\Xs$ and a closed subset $X \subseteq \Xs$. A morphism $f : (X, \Xs) \to (Y, \Ys)$ of pregerms is a morphism of adic spaces $f : \Xs \to \Ys$ such that $f(X) \subseteq Y.$ 
		
		A morphism of pregerms $j : (X, \Xs) \to (Y, \Ys)$ is called a strict neighborhood if $j$ is an open immersion and $j(X) = Y$. We localize the category of pregerms at strict neighborhoods; the resulting category ${\rm Germ}$ is called the category of germs.
		
		For a germ $(X, \Xs)$, we define the sheaf of rings $\O_X$ on $X$ to be the pullback of $\O_{\Xs}$. We often drop $\Xs$ and simply write $X$ for the germ $(X, \Xs)$.
	\end{definition}
	
	In order to calculate the ring of sections of germs, we state the following proposition:
	
	\begin{proposition} \label{prop:germ-section}
		Let $(X, \Xs)$ be a germ with $\Xs$ an affinoid adic space adic over $K$. If $X$ is the intersection of its open neighborhoods in $\Xs$, then
		$$\varinjlim_{\Us} H^i(\Us, \O_\Xs) \to H^i(X, \O_X)$$
		is an isomorphism for any $i$, where $\Us$ runs over all open neighborhoods of X in $\Xs$.
	\end{proposition}
	\begin{proof}
		This is the content of \cite[Lemma 1.4.3]{abe2025proper}.
	\end{proof}
	
	\begin{definition}[cf. {\cite[Definition 2.22, Definition 5.10]{lazda2016rigid}}]
		A frame over $\V[[t]]$ (resp. $\E_K$) or a $\V[[t]]$-frame (resp. $\E_K$) is a quintuple $(X, Y, \Pf, j, i)$ where $X$ is a variety over $k((t))$, $Y$ is a separated scheme over $k[[t]]$ of finite type (resp. a variety over $k((t))$), $\Pf$ is a formal scheme over $\V[[t]]$ (resp. over $\O_{\E_K}$), $j : X \hookrightarrow Y$ is an open immersion and $i: Y \hookrightarrow \Pf$ is a closed immersion. There is an obvious notion of morphisms of frames. We often drop $j$ and $i$ from the notation and simply say that $(X, Y, \Pf)$ is a frame.
		
		A morphism $(f, g, u) : (X', Y', \Pf') \to (X, Y, \Pf)$ of $\V[[t]]$-frames or $\E_K$-frames is said to be proper (resp. smooth, etale) if $g$ is proper (resp. $u$ is smooth in an open neighborhood of $X'$, $u$ is etale in an open neighborhood of $X'$). A $\V[[t]]$-frame (resp. $\E_K$-frame) $(X, Y, \Pf)$ is said to be proper (resp. smooth, etale) if the structure morphism $(X, Y, \Pf) \to (\Spec k((t)), \Spec k[[t]], \Spec \V[[t]])$ (resp. $(X, Y, \Pf) \to (\Spec k((t)), \Spec k((t)), \Spf \O_{\E_K})$) is.
	\end{definition}
	
	\begin{remark}
		The category of $\E_K$-frames is a full subcategory of $\V[[t]]$-frames.
	\end{remark}
	
	\begin{definition}
		For a formal scheme $\Pf$ over $\V[[t]]$, we denote by $\Pf^\ad$ the associated rigid analytic space, let $\sp : \Pf^\ad \to \Pf$ be the specialization map and let $[\cdot] : \Pf^\ad \to \Pf^\ad$ be the map that takes a point to its maximal generalization. We denote by $[\sp] : \Pf^\ad \to \Pf$ to be the composite of $[\cdot]$ and $\sp$. Noting that the special fiber $P$ of $\Pf$ is homeomorphic to $\Pf$, for a subset $X \subseteq P$ define the tube $]X[_\Pf$ of $X$ to be $[\sp]^{-1}(X)$, the set theoretic inverse image of $X$ by $[\sp]$. When $\Pf$ is clear from the context, we often omit it from the notation of tubes.
	\end{definition}
	
	\begin{proposition}[{\cite[Proposition II.4.2.11]{fujiwara2018foundations}}, {\cite[Lemma 2.17, Proposition 2.18]{lazda2016rigid}}, {\cite[Proposition 2.4.7]{caro2021arithmetic}}]
		Let $\Pf = \Spf A$ be an affine formal $\V$-scheme. Then for any closed subscheme $Z = V(\bar{f_1}, \dots, \bar{f_r})$ of $X$ with $f_i \in A$, we have
		$$ ]Z[_{\Xf} = \sp^{-1}(Z)^\circ = \{ x \in \Xf^\ad : |f_1|_{[x]}, \dots, |f_r|_{[x]} < 1 \} = \bigcup_{\eta \in (0, 1) \cap \Gamma^*} \{ x \in \Xf^\ad : |f_1|_{x}, \dots, |f_r|_{x} \leq \eta \}.$$
	\end{proposition}
	
	\begin{remark}
		The functor $\Box^\ad$ is locally defined as follows: for $\Pf = \Spf A$, we have $\Pf^{\ad} = \Spa (A_K, A_K^\circ)$, where $A_K^\circ$ is the power-bounded elements of $A_K$.
	\end{remark}
	
	\begin{definition}
		For a frame $(X, Y, \Pf)$ over $\V[[t]]$, the pair $(]X[_\Pf, ]Y[_\Pf)$ is a germ. Let $i : ]X[_\Pf \hookrightarrow ]Y[_\Pf$ be the closed immersion. Define the functor $j^\dag_X : \Sh(]Y[_\Pf) \to \Sh(]Y[_\Pf)$ to be $i_*i^{-1}.$
	\end{definition}
	
	With the above definitions, the theorems that hold in the classical setting are available:
	
	\begin{proposition} \label{prop:tube-sheaf}
		Let $(X, Y, \Pf)$ be a frame over $\V[[t]]$.
		\begin{enumerate}[(i)]
			\item For any open subset $U$ of $]Y[$, we have $\displaystyle\Gamma(U, j^\dag_X\O_{]Y[}) = \varinjlim_{V} \Gamma(U \cap V, \O_{]Y[})$ where $V$ runs over open neighborhoods of $]X[$ in $]Y[$.
			\item The sheaves of rings $\O_{]X[}$ and $j^\dag_X\O_{]Y[}$ are coherent.
			\item The canonical functor $\displaystyle\varinjlim_{V} \Coh(\O_V) \to \Coh(j^\dag_X\O_{]Y[})$ is an equivalence of categories, where $V$ varies as in (i).
			\item The functors $i_* : \Coh(\O_{]X[}) \leftrightarrows \Coh(j^\dag_X\O_{]Y[}) : i^{-1}$ are quasi-inverse to each other.
			\item (Theorem A) If $\Pf$ is affine and the closed immersion $Y \to \Pf$ induces an isomorphism of $Y$ and the special fiber $\Pf_0$, then $A = \Gamma(]X[, \O_{]X[})$ is noetherian and $\Gamma(]X[, \Box) : \Coh(\O_{]X[}) \to \Modfg(A)$ is an equivalence between the category of coherent $\O_{]X[}$-modules and the category of finitely generated $A$-modules.
			\item (Theorem B) In the situation of (v), for a coherent $\O_{]X[}$-module $\E$, we have $H^i(]X[, \E) = 0$ for all $i > 0.$
		\end{enumerate}
	\end{proposition}
	\begin{proof}
		(i) by the same argument as in \cite[1.4.1]{abe2025proper}. 
		(iii) by \cite[2.55]{lazda2016rigid}.
		(iv) by abstract nonsense.
		
		(ii), (v) and (vi) are verified by Proposition \ref{prop:theoremAB} and Proposition \ref{prop:tube-pds}.
	\end{proof}
	
	\begin{theorem}[Strong fibration theorem] \label{thm:strong-fibration}
		Let
		$$
		\xymatrix{
			X \ar@{^{(}->}[r] \ar@{=}[d] & Y' \ar@{^{(}->}[r] \ar[d]^{g} & \Pf' \ar[d]^{u}\\
			X \ar@{^{(}->}[r] & Y \ar@{^{(}->}[r] & \Pf
		}
		$$
		be a proper etale morphism of frames over $\V[[t]]$.
		Then $u^\ad$ induces an isomorphism between a cofinal system of neighborhoods of $]X[_{\Pf'}$ and that of $]X[_\Pf$.
		In particular, the germs $]X[_{\Pf'}$ and $]X[_\Pf$ are isomorphic.
	\end{theorem}
	
	\subsection{Partial dagger algebra and spectra} \label{subsec:pda}
	For $\eta \in (0, 1)$, define $\displaystyle\E_\eta = \{ \sum_{n \in \mathbb{Z}} a_nt^n \in \Ed_K : \lim_{n \to -\infty} |a_n|\eta^n = 0 \}.$ It is a Banach $K$-algebra equipped with the norm
	$$ \displaystyle \| \sum a_n t^n \|_\eta = \max\left(\sup_{n \leq 0} |a_n|\eta^n, \ \sup_{n \geq 0} |a_n|\right). $$
	For a positive integer $m$, define $\E_m = \E_{p^{1/m}}$ and $\| \cdot \|_m = \| \cdot \|_{p^{1/m}}$.
	
	\begin{definition}
		For nonnegative integers $d, e$ and $\eta > 1$, define
		$$\displaystyle W^\dag_{d, e, \eta} = \E_{\eta^{-1}}\langle x_1, \dots, x_d, y_1, \dots, y_e \rangle_{|x_i| \leq 1, \, |y_j| \leq \eta}$$
		and
		$$W^\dag_{d, e} = \E^\dag_K\langle x_1, \dots, x_d \rangle \langle y_1, \dots, y_e \rangle ^\dag = \displaystyle\lim_{m \to \infty} W^\dag_{d, e, p^{1/m}}.$$
		When $d = 0$, we simply denote it by $W^\dag_e = \E^\dag_K \langle y_1, \dots, y_e \rangle ^\dag$ and call it a Lazda-P\'{a}l algebra.
	\end{definition}
	
	$W^\dag_{d, e, \eta}$ is a Banach $\E_\eta$-algebra and we again let $\| \cdot \|_\eta$ denote the canonical norm on $W^\dag_{d, e, \eta}$. The inductive limit topology induced by norms $\| \cdot \|_\eta$ is called the fringe topology. On the other hand, there is an obvious inclusion $W^\dag_{d, e} \subseteq T_{d+e} = \E_K \langle x_1, \dots, x_d, y_1, \dots, y_e\rangle$ and let $\| \cdot \|$ be the Gauss norm on $T_{d+e}$; the subspace topology on $W^\dag_{d, e}$ is called affinoid topology.
	
	This definition is a slight generalization of \cite{lazda2016rigid}, in which they only consider the case $d = 0$. The propositions regarding $W^\dag_e$ have natural generalizations.
	
	\begin{proposition}[cf. {\cite[3.2]{lazda2016rigid}}]
		Fix $f \in W^\dag_{d, e}$ and assume $\|f\| \leq c.$ Then for any $\epsilon > 1$ there exists some $\eta > 1$ such that $f \in W^\dag_{d, e, \eta}$ and $\|f\|_\eta \leq c\epsilon.$
	\end{proposition}
	
	\begin{proposition}
		$W^\dag_{d, e}$ is a noetherian ring and $W^\dag_{d, e} \to T_{d+e}$ is faithfully flat.
	\end{proposition}
	\begin{proof}
		Let $W$ be the $p$-adic weak completion of the ring $\V[[t]]\langle x_1, \dots, x_d \rangle[t^{-1}, y_1, \dots, y_e]$ and $T$ the $p$-adic completion of $W$. By \cite{fulton1969note}, $W$ is noetherian and $W \to T$ is faithfully flat. Since $W^\dag_{d, e} = W_K$ and $T_{d+e} = T_K$, we win. 
	\end{proof}
	
	\begin{corollary}
		Any ideal $I$ of $W^\dag_{d, e}$ is a closed subset for the affinoid topology.
	\end{corollary}
	\begin{proof}
		By faithful flatness, $I = IT_{d+e} \cap W^\dag_{d, e}$.
	\end{proof}
	
	\begin{definition} \label{def:pda}
		A partial dagger algebra over $\Ed_K$ is a $\Ed_K$-algebra $A$ which is isomorphic to some $W^\dag_{d, e} / I$.
		
		For a partial dagger algebra $A$ and a surjection $W^\dag_{d, e} \to A$, the quotient topology on $A$ induced by the affinoid topology of $W^\dag_{d, e}$ is again called an affinoid topology. Given an affinoid topology on $A$, we denote by $\wh{A}$ the completion with respect to the topology.
	\end{definition}
	
	\begin{remark}
		Differing from the case of Tate algebras, Monsky-Washnitzer algebras and Lazda-P\'{a}l algebras, the affinoid topology is not unique for partial dagger algebras. We will often consider completions of partial dagger algebras and readers should be careful that we implicitly consider surjective homomorphisms from $W_{d, e}^\dag$,
	\end{remark}
	
	\begin{corollary} \label{cor:pda}
		Any partial dagger algebra $A$ is a noetherian ring and $A \to \wh{A}$ is faithfully flat.
	\end{corollary}
	
	\begin{definition}
		Let $A$ be a partial algebra, $f : W^\dag_{d, e} = \E^\dag_K\langle x_1, \dots, x_d \rangle \langle y_1, \dots, y_e \rangle ^\dag \to A$ a surjective homomorphism and $I = \ker f$. Let $I_0$ be an ideal of $W^\dag_{d, e, \eta_0}$ for some $\eta_0 \in (1, \infty] \cap \Gamma^*$ such that $I_0 W^\dag_{d, e} = I$. We call $A_\eta = W^\dag_{d, e, \eta} / I_0 W^\dag_{d, e, \eta}$ a fringe algebra of $A$ for $\eta \in (1, \eta_0] \cap \Gamma^*$ and call $\{A_\eta\}_\eta$ a fringe inductive system of $A$.
		The partial dagger spectrum $\Spd A$ of $A$ is defined to be the germ $(\{ x \in \Spa A_{\eta_0} : |t|_{[x]} \geq 1, |y_1|_{[x]}, \dots, |y_e|_{[x]} \leq 1 \}, \Spa A_{\eta_0})$. Note that $\{ \Spa A_\eta \}_\eta$ is a cofinal system of open neighborhood of $\Spd A$ in $\Spa A_{\eta_0}.$
	\end{definition}
	
	\begin{remark}
		This definition is dependent on the choice of a surjection $W^\dag_{d, e} \to A$ and $\eta_0, \, I_0$. When $d = 0$, it is independent of the choice since any fringe inductive systems are equivalent by the proof of \cite[3.12]{lazda2016rigid}. Although the notation of partial dagger spectrum $\Spa A$ should contain the information of the presentation, we abbreviate it in order to avoid the notation being intricate. 
	\end{remark}
	
	\begin{definition}
		An affinoid partial dagger space is a germ isomorphic to $\Spd A$ for some partial dagger algebra $A$.
		A parital dagger space is a germ locally isomorphic to an affinoid partial dagger space.
		
		Given an affinoid partial dagger space $(X, \Xs) = \Spd A$, we define the completion $\wh{X}$ to be $\Spa \wh{A},$ which is an open subset of $X$. Given a partial dagger space $(X, \Xs)$, and an affinoid open covering $\{ U_i \}$ of $X$, we define the completion $\wh{X}$ to be the open subset $\bigcup_i \wh{U}_i$ of $X$.
	\end{definition}
	
	\begin{remark}
		The completion of partial dagger space is dependent on the choice of open coverings, but we again omit it from the notation.
	\end{remark}
	
	\begin{remark} \label{rem:open-pds}
		Given a partial dagger space $(X, \Xs)$ and an open subset $\Us$ of $\Xs$, the germ $(X \cap \Us, \Us)$ is again a partial dagger space. In order to check this, we may assume that $(X, \Xs)$ is an affinoid and $\Us$ is a rational open subset of $\Xs$; it is straightforward (see the proof of following proposition). 
	\end{remark}
	
	\begin{proposition} \label{prop:theoremAB}
		Let $A$ be a partial dagger algebra and $X = \Spd A.$ 
		\begin{enumerate}[(i)]
			\item $\O_X$ is a coherent sheaf.
			\item (Theorem A) The functor $\Gamma(X, \Box) : \Coh(\O_X) \to \Modfg(A)$ is an equivalence of categories.
			\item (Theorem B) For a coherent $\O_X$-module $\E$, we have $H^i(X, \E) = 0$ for all $i > 0$.
		\end{enumerate}
	\end{proposition}
	\begin{proof}
		(i) Let $\{ A_\eta \}_{1 < \eta \leq \eta_0}$ be a fringe inductive system of $A$ which gives rise to $X$. Let $\Xs = \Spa A_{\eta_0}$. For a rational open subset $\Us = \Xs(\frac{f_1, \dots, f_r}{g})$, we have $\Gamma(\Us \cap X, \O_X) = A \langle x_1, \dots, x_r \rangle/(gx_1 - f_1, \dots, gx_r - f_r).$ In particular, $\Gamma(\Us \cap X, \O_X)$ is noetherian by Corollary \ref{cor:pda}. The induced map
		$\Gamma(X, \O_X)^{\wh{\ }} \to \Gamma(\Us \cap X, \O_X)^{\wh{\ }}$
		is the restriction map of rigid analytic varieties over $\E_K$ and therefore it is flat. By faithful flatness of Corollary \ref{cor:pda}, the restriction map
		$\Gamma(X, \O_X) \to \Gamma(\Us \cap X, \O_X)$
		is also flat.
		
		(ii) Let $\Us_\eta = \Spa A_\eta$. These form a cofinal system of open neighborhoods of $X$ in $\Xs$, and we have $\Coh(\O_X) \simeq \varinjlim_{\eta} \Coh(\O_{\Us_\eta}) \simeq \varinjlim_{\eta} \Modfg(A_\eta) \simeq \Modfg(A).$
		
		(iii) Use \cite[1.4.3]{abe2025proper} and the vanishing of higher cohomology of coherent modules on affinoid spaces.
	\end{proof}
	
	\begin{proposition} \label{prop:tube-pds}
		Let $(X, Y, \Pf)$ be a frame over $\V[[t]]$. Assume that $Y-X$ is a divisor of $Y$. Then the tube $]X[$ is a partial dagger space.
	\end{proposition}
	It suffices to prove this for $\Pf = \Spf A$ affine. Take a presentation $A \simeq \V[[t]]\langle x_1, \cdots, x_n \rangle / (f_1, \dots, f_k)$ and choose $g_1, \dots, g_l, h \in A$ such that $Y = V(g_1, \dots, g_l)$ and $X = Y \cap D(h).$
	Then
	$$]Y[ = \{ x \in \Pf^\ad : |g_1|_{[x]}, \dots, |g_l|_{[x]} < 1\}$$
	is an open subset of $\Pf$ and
	$$]X[ = \{ x \in ]Y[ \, : |h|_{[x]} \geq 1, \, |t|_{[x]} \geq 1\}.$$
	Let $Z = \{ x \in \Pf^\ad : |h|_{[x]} \geq 1, \, |t|_{[x]} \geq 1 \}.$ Then by definition $(Z, \Pf^\ad) = \Spd W^\dag_{n, 1}/(f_1, \dots, f_k, gy - 1).$ Since $]X[ = Z \cap ]Y[_\Pf$, it follows from Remark \ref{rem:open-pds}.
	
	\subsection{Overconvergent isocrystals}
	Definitions of $\Ed$ and $K$-valued overconvergent isocrystals are given in \cite[2.54, 5.11]{lazda2016rigid}. For reader's sake, we collect the definitions and propositions in this and the following sections.
	
	\begin{definition}[cf. {\cite[Definition 5.9]{lazda2016rigid}}]
		\begin{enumerate}
			\item We say that a $k$-scheme $X$ is of pseudo finite type if there is a $k$-morphism $X \to \Spec k[[t]] \times_k \cdots \times_k \Spec k[[t]]$ of finite type.
			\item We say that a $p$-adic formal scheme $\Xf$ over $\V$ is of pseudo finite type if there is a $\V$-morphism $\Xf \to \Spf(\V[[t]]) \times_\V \cdots \times_\V \Spf(\V[[t]])$ topologically of finite type.
			\item We say that a rigid analytic space $\Xs$ over $K$ is called locally of pseudo finite type if there is a $K$-morphism $\Xs \to \Spa(\V[[t]]_K) \times_K \cdots \times_K \Spa(\V[[t]]_K)$ locally of finite type.
		\end{enumerate}
	\end{definition}
	
	\begin{definition}[cf. {\cite[Definition 5.10]{lazda2016rigid}}]
		A frame over $k((t))$ or a $k((t))$-frame is a 5-uple $(X, Y, \Pf, j, i)$ of a $k((t))$-variety $X$, a  $k[[t]]$-scheme $Y$, a formal $\V$-scheme $\Pf$ of pseudo finite type, an open immersion $j : X \to Y$ and a closed immersion $i : Y \to \Pf$. There is an obvious notion of morphisms of $k((t))$-frames.
		
		A morphism $(f, g, u) : (X', Y', \Pf') \to (X, Y, \Pf)$ of $k((t))$-frames is said to be proper (resp. smooth, etale) if $g$ is proper (resp. $u$ is smooth in an open neighborhood of $X'$, $u$ is etale in an open neighborhood of $X'$). A $k((t))$-frame $(X, Y, \Pf)$ is said to be proper (resp. smooth, etale) if the structure morphism $(X, Y, \Pf) \to (\Spec k((t)), \Spec k[[t]], \Spf \V[[t]])$ is.
	\end{definition}
	
	For a $k((t))$-frame $(X, Y, \Pf)$, we can define the tubes $]X[_\Pf$ and $]Y[_\Pf$ exactly as before, together with the functor $j^\dag_X$. Proposition \ref{prop:tube-sheaf} again holds for $k((t))$-frames.
	
	\begin{definition}
		A pair over $\V[[t]]$ (resp. $k((t))$, resp. $\E_K$) or a $\V[[t]]$-pair (resp. $k((t))$-pair, $\E_K$-pair) is a triple $(X, Y, j)$ with $X$ a $k((t))$-variety, $Y$ a separated $k[[t]]$-scheme of finite type (resp. a $k[[t]]$-scheme of pseudo finite type, a variety over $k((t))$) and $j : X \to Y$ an open immersion. We often drop $j$ from the notation and simply write $(X, Y)$ for a pair. There is an obvious notion of morphisms of pairs. A morphism $(f, g) : (X', Y') \to (X, Y)$ is said to be proper if $g$ is proper. A $\V[[t]]$-pair (resp. $k((t))$-pair, $\E_K$-pair) is said to be proper if the structure morphism $(X, Y) \to (\Spec k((t)), \Spec k[[t]])$ (resp. $(X, Y) \to (\Spec k((t)), \Spec k[[t]])$, $(X, Y) \to (\Spec k((t)), \Spec k((t)))$ is.
	\end{definition}
	
	\begin{definition}[cf. {\cite[Definition 2.54, Definition 5.10]{lazda2016rigid}}]
		\begin{enumerate}
			\item An $\Ed_K$-valued (resp. $K$-valued, $\E_K$-valued) overconvergent isocrystal $\E$ on a $\V[[t]]$-frame (resp. $k((t))$-frame, $\E_K$-frame) $(X, Y, \Pf)$ is a collection of coherent $\O_{]X'[_{\Pf'}}$-modules $\E_{\Pf'}$, one for each $\V[[t]]$-frame (resp. $k((t))$-frame, $\E_K$-frame) $(X', Y', \Pf')$ over $(X, Y, \Pf)$, together with $\O$-linear isomorphisms $\epsilon_u : u^* \E_{\Pf'} \to \E_{\Pf''}$ for every morphism of $\V[[t]]$-frames (resp. $k((t))$-frames, $\E_K$-frames) $u : (X'', Y'', \Pf'') \rightarrow (X', Y', \Pf')$ such that the diagram
			\[
			\xymatrix{
				(X'', Y'') \ar[rr] \ar[rd] & & (X', Y') \ar[ld] \\
				& (X, Y) &
			}
			\]
			commutes, which satisfies the cocycle condition: for any commutative diagram
			\[
			\xymatrix{
				(X''', Y''', \Pf''') \ar[rr]^{u''} \ar[rd]_{u} & & (X'', Y'', \Pf'') \ar[ld]^{u'} \\
				& (X', Y', \Pf') &
			}
			\]
			of frames over $(X, Y, \Pf)$, we have $\epsilon_u = \epsilon_{u''} (u''^*\epsilon_{u'})$. There is an obvious notion of morphism of $\Ed_K$-valued (resp. $K$-valued, $\E_K$-valued) overconvergent isocrystal $\E$ on $(X, Y, \Pf)$. The category of $\Ed_K$-valued (resp. $K$-valued, $\E_K$-valued) overconvergent isocrystal is denoted by $\Isoc(X, Y, \Pf/\Ed_K)$ (resp. $\Isoc(X, Y, \Pf/K), \Isoc(X, Y, \Pf/\E_K)$).
			
			\item A $\V[[t]]$-frame (resp. $k((t))$-frame, $\E_K$-frame) over a $\V[[t]]$-pair (resp. $k((t))$-pair, $\E_K$-pair) $(X, Y)$ is a frame $(X', Y', \Pf')$ together with a morphism of pairs $(X', Y') \to (X, Y)$. An $\Ed_K$-valued (resp. $K$-valued, $\E_K$-valued) overconvergent isocrystal $\E$ on $(X, Y)$ is a collection of coherent $\O_{]X'[_{\Pf'}}$-modules $\E_{\Pf'}$, one for each frame $(X', Y', \Pf')$ over $(X, Y, \Pf)$, together with $\O$-linear isomorphisms $\epsilon_u : u^* \E_{\Pf'} \to \E_{\Pf''}$ for every morphism of frames $u : (X'', Y'', \Pf'') \rightarrow (X', Y', \Pf')$, which satisfies the cocycle condition as above. There is an obvious notion of morphism of $\Ed_K$-valued (resp. $K$-valued, $\E_K$-valued) overconvergent isocrystal $\E$ on $(X, Y)$. The category of $\Ed_K$-valued overconvergent isocrystals is denoted by $\Isoc(X, Y/\Ed_K)$ (resp. $\Isoc(X, Y/K)$, $\Isoc(X, Y/\E_K)$).
			
			\item A $\V[[t]]$-frame (resp. $k((t))$-frame, $\E_K$-frame) over a $k((t))$-variety $X$ is a frame $(X', Y', \Pf')$ together with a morphism $X' \to X$. An $\Ed_K$-valued overconvergent isocrystal $\E$ on $X$ is a collection of coherent $\O_{]X'[_{\Pf'}}$-modules $\E_{\Pf'}$, one for each frame $(X', Y', \Pf')$ over $(X, Y, \Pf)$, together with $\O$-linear isomorphisms $\epsilon_u : u^* \E_{\Pf'} \to \E_{\Pf''}$ for every morphism of frames $u : (X'', Y'', \Pf'') \rightarrow (X', Y', \Pf')$ over $X$, which satisfies the cocycle condition as above. There is an obvious notion of morphism of $\Ed_K$-valued (resp. $K$-valued, $\E_K$-valued) overconvergent isocrystal $\E$ on $X$. The category of $\Ed_K$-valued overconvergent isocrystal is denoted by $\Isoc(X/\Ed_K)$ (resp. $\Isoc(X/K)$, $\Isoc(X/\E_K)$).
		\end{enumerate} 
	\end{definition}
	
	The expected independence of frames also holds in our situation.
	
	\begin{theorem}
		\begin{enumerate}[(i)]
			\item For any proper smooth morphism of $\V[[t]]$-frames (resp. $k((t))$-frames, $\E_K$-frames) $(id, g, u) : (X, Y', \Pf') \to (X, Y, \Pf)$, the canonical functor $u^* : \Isoc(X, Y, \Pf/\Ed_K) \to \Isoc(X, Y', \Pf'/\Ed_K)$ (resp. $u^* : \Isoc(X, Y, \Pf/K) \to \Isoc(X, Y', \Pf'/K)$, $u^* : \Isoc(X, Y, \Pf/\E_K) \to \Isoc(X, Y', \Pf'/\E_K)$) is an equivalence of categories.
			\item For any proper morphism of $\V[[t]]$-pairs (resp. $k((t))$-pairs, $\E_K$-pairs) $(id, g): (X, Y') \to (X, Y)$, the canonical functor $g^* : \Isoc(X, Y/\Ed_K) \to \Isoc(X, Y'/\Ed_K)$ (resp. $g^* : \Isoc(X, Y/K) \to \Isoc(X, Y'/K)$, $g^* : \Isoc(X, Y/\E_K) \to \Isoc(X, Y'/\E_K)$)  is an equivalence of categories.
			\item For any smooth $\V[[t]]$-frame (resp. $k((t))$-frame, $\E_K$-frame) $(X, Y, \Pf)$, the forgetful functor $\Isoc(X, Y/\Ed_K) \to \Isoc(X, Y, \Pf/\Ed_K)$ (resp. $\Isoc(X, Y/K) \to \Isoc(X, Y, \Pf/K)$, $\Isoc(X, Y/\E_K) \to \Isoc(X, Y, \Pf/\E_K)$) is an equivalence of categories.
			\item For any proper $\V[[t]]$-pair (resp. $k((t))$-pair, $\E_K$-pair) $(X, Y)$, the forgetful functor $\Isoc(X/\Ed_K) \to \Isoc(X, Y/\Ed_K)$ (resp. $\Isoc(X/K) \to \Isoc(X, Y/K)$, $\Isoc(X/\E_K) \to \Isoc(X, Y/\E_K)$) is an equivalence of categories.
		\end{enumerate}
	\end{theorem}
	\begin{proof}
		The proof goes as in the classical case; see \cite[Proposition 2.56, Corollary 2.57]{lazda2016rigid} for example.
	\end{proof}
	
	As in the classical case, we can define stratified modules and modules with integrable connections on a frame and give functors.
	
	Let $(X, Y, \Pf)$ be a smooth $\V[[t]]$-frame (resp. $k((t))$-frame, $\E_K$-frame), let $\Pf(j)$ be the $(j+1)$-fold fiber product of $\Pf$ over $\V[[t]]$ (resp. $\V$, $\O_{\E_K}$), let $p_0, p_1 : \Pf(1) \rightarrow \Pf, \ p_{01}, p_{02}, p_{12} : \Pf(2) \rightarrow \Pf(1)$ be the projections and $\Delta : \Pf \rightarrow \Pf(1)$ the diagonal morphism. We also regard these morphisms as frames; for example, $p_0$ is a morphism $(X, Y, \Pf(1)) \to (X, Y, \Pf)$.
	For any adic $\Xs$ space over $\V[[t]]_K$ (resp. $K$, $\E_K$), we define $\Xs^{(n)}$ to be the $n$-th infinitesimal neighborhood of $\Xs$ in $\Xs \times_{\V[[t]]_K} \Xs$ (resp. $\Xs \times_K \Xs$, $\Xs \times_{\E_K} \Xs$) and let $p_0^{(n)}, p_1^{(n)} : \Xs^{(n)} \to \Xs$ to be the projections.
	
	\begin{definition}[{cf. \cite[Definition 2.58, Definition 5.12]{lazda2016rigid}}]
		Let $(X, Y, \Pf)$ be a smooth $\V[[t]]$-frame.
		\begin{enumerate}
			\item An $\Ed_K$-stratification (resp. $K$-stratification, $\E_K$-stratification) on a $j_X^\dag\O_{]Y[_\Pf}$-module $\E$ is a collection of compatible $\O$-linear isomorphisms $\epsilon_n : p_1^{(n)*} \E \to p_0^{(n)*} \E$ with $\epsilon_0 = \id$ satisfying cocycle conditions.
			\item An overconvergent $\Ed_K$-stratification (resp. $K$-stratification, $\E_K$-stratification) on a $j_X^\dag\O_{]Y[_\Pf}$-module $\E$ is an $j^\dag_X \O_{]Y[_{\Pf(1)}}$-linear isomorphism $\epsilon_n : p_1^* \E \to p_0^* \E$ with $\Delta^* \epsilon = \id$ and  satisfying cocycle condition $p_{02}^* \epsilon = (p_{01}^* \epsilon)(p_{12}^* \epsilon)$.
			\item A $\Ed_K$-connection (resp. $K$-connection, $\E_K$-connection) on a $j_X^\dag \O_{]Y[_\Pf}$-module $\E$ is a $\Ed_K$-linear morphism $\nabla : \E \to \E \otimes_{j_X^\dag \O_{]Y[_\Pf}} j_X^\dag\Omega^1_{]Y[_\Pf / \V[[t]]_K}$ (resp. $\nabla : \E \to \E \otimes_{j_X^\dag \O_{]Y[_\Pf}} j_X^\dag\Omega^1_{]Y[_\Pf / K}$, $\nabla : \E \to \E \otimes_{j_X^\dag \O_{]Y[_\Pf}} j_X^\dag\Omega^1_{]Y[_\Pf / \E_K}$) satisfying Leibnitz rule. A connection is said to be integral if $\nabla^2 = 0$.
		\end{enumerate}
	\end{definition}
	
	We can define overconvergent stratifications and connections for coherent $\O_{]X[_\Pf}$-modules, which coincides with the definition above via the categorical equivalence $\Coh(j^\dag_X\O_{]Y[_\Pf}) \simeq \Coh(\O_{]X[_\Pf})$. The categories of coherent $j^\dag \O_{]Y[}$-modules equipped with $R$-stratifications, overconvergent $R$-stratifications and integrable $R$-connections are repectively denoted by $\Strat(X, Y, \Pf/R), \, \Strat^\dag(X, Y, \Pf/R)$ and $\MIC(X, Y, \Pf/R)$ for $R = \Ed_K, K, \E_K.$
	
	\begin{proposition}[{cf. \cite[Propositoin 2.60, Proposition 5.13]{lazda2016rigid}}]
		Let $R$ be one of $\Ed_K, K, \E_K$. There exist canonical functors $\Isoc(X, Y, \Pf/R) \to \Strat^\dag(X, Y, \Pf/R)$, $\Strat(X, Y, \Pf/R) \to \MIC(X, Y, \Pf/R)$ and $\Strat^\dag(X, Y, \Pf/R) \to \Strat(X, Y, \Pf/R)$.
		The first and second functors are equivalence of categories and the last one is fully faithful.
	\end{proposition}
	
	\begin{definition}
		Let $\MIC^\dag(X, Y, \Pf/\Ed_K)$ be the essential image of the functor $\Strat^\dag(X, Y, \Pf/\Ed_K) \to \MIC(X, Y, \Pf/\Ed_K)$. This category is equivalent to $\Isoc(X, Y, \Pf/\Ed_K).$ An object $(\E, \nabla)$ of $\MIC$ is called a $\nabla$-module and is said to be overconvergent if it belongs to $\MIC^\dag.$
	\end{definition}
	
	As in the classical case, we have following propositions:
	
	\begin{proposition}[{\cite[Lemma 2.66]{lazda2016rigid}}]
		The canonical functor $\varinjlim_V \MIC(\O_V/\V[[t]]_K) \to \MIC(X, Y, \Pf/\Ed_K)$ (resp. $\varinjlim_V \MIC(\O_V/K) \to \MIC(X, Y, \Pf/K)$, $\varinjlim_V \MIC(\O_V/\E_K) \to \MIC(X, Y, \Pf/\E_K)$) is an equivalence of categories, where $V$ runs over all open neighborhoods of $]X[$ in $]Y[$.
	\end{proposition}
	
	\begin{proposition}[{\cite[Corollary 2.67]{lazda2016rigid}}]
		Let $R$ be one of $\Ed_K, K, \E_K$. The category $\MIC(X, Y, \Pf/R)$ and the condition of being overconvergent are both local on $X$.
	\end{proposition}
	
	\begin{definition}
		Let $A$ be a Banach $K$-algebra, $M$ a finitely generated $A$-module and $\| \cdot \|$ a norm on $M$. For a multi-indexed sequence $a_{\unu} \in M$ and $\eta \in \mathbb{R}_{> 0}$, we say that $a_{\unu}$ is $\eta$-convergent if $\| a_\unu \| \eta^{|\unu|}$ converges to zero.
	\end{definition}
	
	\begin{proposition}[{\cite[Proposition 2.68]{lazda2016rigid}}] \label{prop:mic-oc}
		Suppose that $(X, Y, \Pf)$ is a smooth $\V[[t]]$-frame (resp. $k((t))$-frame, $\E_K$-frame), $\Pf = \Spf A$ is affine and $Y = V(\bar f_1, \dots, \bar f_r)$, $X = Y \cap D(\overline{tg})$ for some $f_1, \dots, f_r, g \in A$. For $\lambda \in \Gamma^* \cap (0, 1)$, define
		$$[Y]_\lambda = \{ \Pf^\ad : |f_1|_x, \dots, |f_r|_x \leq \lambda\}, \ [X]_\lambda = \{ x \in \Pf^\ad : |t|_x, |g|_x \geq \lambda \}.$$
		Let $V$ be an affinoid open neighborhood of $]X[$ in $]Y[$, and let $t_1, \dots, t_n \in \Gamma(V, \O)$ be such that $dt_1, \dots, dt_n$ freely generates $\Omega_{V/\V[[t]]_K}$ (resp. $\Omega_{V/K}$, $\Omega_{V/\E_K}$). Let $(\E, \nabla)$ be a coherent $\O_V$-module with integrable connection and let $E = \E | _{]X[}$ be the induced object of $\MIC(X, Y, \Pf/\Ed_K).$ Then $E$ is overconvergent if and only if for all $\eta \in (0, 1) \cap \Gamma^*$, there exist some $\lambda \in [\eta, 1) \cap \Gamma^*$ and $\rho \in (0, 1) \cap \Gamma^*$ such that $[Y]_{\lambda} \cap [X]_{\rho} \subseteq V$ and for any section $e \in \Gamma([Y]_{\lambda} \cap [X]_{\rho}, \E)$, the multi-indexed sequence $\displaystyle \frac{ \underline{\partial}^{\underline{\nu}} e }{\underline{\nu}!}$ is $\eta$-convergent where $\displaystyle\partial_i = \frac{\partial}{\partial t_i}$.
	\end{proposition}
	
	\subsection{Functors on overconvergent isocrystals}
	Let $(X, Y, \Pf)$ be a $\V[[t]]$-frame. There exist canonical forgetful functors
	$$ \Isoc(X, Y, \Pf/K) \to \Isoc(X, Y, \Pf/\E^\dag_K), $$
	$$ \Isoc(X, Y, \Pf/\E^\dag_K) \to \Isoc(X, Y_\eta, \Pf_\eta/\E_K), \ \E \mapsto \wh{\E}.$$
	Via the faithful functor $\Isoc \simeq \MIC^\dag$, the former functor is simply forgetting the derivation regarding $t$ and the latter is the extension of scalers, called the completion of overconvergent isocrystal. 
	
	Let $F/k((t))$ be a finite separable extension. By Cohen's structure theorem, $F \simeq l((u))$ for some finite separable field extension. Let $L$ be a field finite separable over $K$ with residue field $l$. Since $\Ed_K$ is a henselian discrete valuation field with residue field $k((t))$, the finite separable extension $l((u))/k((t))$ uniquely lifts to some finite field extension of $\Ed_K$; it is of the form $\Ed_L$ with parameter $u$. Let $X' = X \times_{k((t))} l((u)), \, Y' = Y \times_{k[[t]]} l[[u]], \, \Pf' = \Pf \times_{\V[[t]]} \mathcal{W}[[u]]$, where $\mathcal{W}$ is the ring of integers of $L$; then we have the base extension functors
	$$ \Isoc(X, Y, \Pf / \Ed_K) \to \Isoc(X', Y', \Pf' / \Ed_L), $$
	$$ \Isoc(X, Y_{\eta}, \Pf_{\eta} / \E_K) \to \Isoc(X', Y'_{\eta'}, \Pf'_{\eta'} / \E_L) $$
	with $\eta'$ the generic point of $\Spec l[[u]]$. These two are compatible with completion functor $\wh{\Box}$.
	
	For a $k((t))$-variety $X$ (resp. $k[[t]]$-scheme $Y$, formal $\V[[t]]$-scheme $\Pf$), we define
	\begin{align*}
		X^{(n)} & = (X \times_{k((t))} k((t^{1/p^n})))_{\red}, \\
		Y^{(n)} & = Y \times_{k[[t]]} k[[t^{1/p^n}]], \\
		\Pf^{(n)} & = \Pf \times_{\V[[t]]} \V[[t^{1/p^n}]].
	\end{align*}
	Let $\E^{\dag, (n)}_K$ be the bounded Robba ring with parameter $t^{1/p^n}.$
	
	\begin{proposition} \label{prop:purely-insep-equiv}
		Let $(X, Y)$ a smooth $\V[[t]]$-pair. Then the pullback functors
		$$\pi^* : \Isoc(X, Y/\Ed_K) \allowbreak \to \Isoc(X^{(n)}, Y^{(n)}/\E^{\dag, (n)}_K)$$
		$$\pi^* : \Isoc(X, Y/K) \allowbreak \to \Isoc(X^{(n)}, Y^{(n)}/K)$$
		are equivalences of categories.
	\end{proposition}
	\begin{proof}
		For simplicity, we assume that $Y-X$ is a divisor of $Y$ to invoke Proposition \ref{prop:tube-pds}. We will discuss the $\Ed_K$-valued case (proof are the same). Since the categories are local, we may assume that $Y$ is affine and there is a lift of $Y$ to a formal $\V[[t]]$-scheme $\Yf$ smooth in an open neighborhood of $X$. Recall that we have categorical equivalences $\Isoc(X, Y/\Ed_K) \simeq \MIC^\dag(X, Y, \Yf/\Ed_K)$ and $\Isoc(X^{(n)}, Y^{(n)}/\E^{\dag, (n)}_K) \simeq \MIC^\dag(X^{(n)}, Y^{(n)}, \Yf^{(n)}/\E^{\dag, (n)}_K)$ and the pullback functor is the base extension $\Box \otimes_{\O_{]X[}} \O_{]X^{(n)}[}.$ of $\nabla$-modules. Put $u = t^{1/p^n}$. As $1, u, \dots, u^{p^n-1}$ is a basis of $\E^{\dag, (n)}_K$ over $\Ed_K$, it is finite free and we have $\O_{]X^{(n)}[} = \O_{]X[} \otimes_{\Ed_K} \E^{\dag, (n)}_K$. We can define the trace map $\mathrm{Tr} : \O_{]X[^{(n)}} \to \O_{]X[}$. This extends to a morphism of complexes $\mathrm{Tr}^\cdot : \Omega^\cdot_{]X^{(n)}[ / \E^{\dag, (n)}_K} \to \Omega^\cdot_{]X[ / \Ed_K}$ satisfying $\mathrm{Tr}(\eta \wedge \xi) = \mathrm{Tr}(\eta) \wedge \xi$ for $\eta \in \Omega^\cdot_{]X^{(n)}[ / \E^{\dag, (n)}_K}$ and $\xi \in \Omega^\cdot_{]X[ / \Ed_K}.$ Now we can define the pushforward functor
		$$\pi_* : \MIC(X^{(n)}, Y^{(n)}, \Yf^{(n)} / \E^{\dag, (n)}_K) \to \MIC(X, Y, \Yf / \Ed_K)$$
		by sending a $\nabla$-module $\M$ over $]X^{(n)}[$ to $\M$ equipped with the connection $\M \to \M \otimes \Omega^1_{]X[/\Ed_K}, \ m \mapsto (\id \otimes \mathrm{Tr})\nabla(m).$
		It is easy to check that $\pi_*$ glue to give a global functor
		$$\pi_* : \Isoc(X, Y/\Ed_K) \to \Isoc(X^{(n)}, Y^{(n)} / \E^{\dag, (n)}_K).$$
		As usual we have adjunction map $\mathrm{adj} : \id \to \pi_* \pi^*$ and trace map $\mathrm{tr} : \pi^*\pi_* \to \id$ which gives the adjoint pair $(\pi^*, \pi_*)$. The composition $\id \xrightarrow{\mathrm{adj}} \pi_* \pi^* \xrightarrow{\mathrm{tr}} \id$ is multiplication by $p^n$ and thus $\pi^*$ is an equivalence of categories.
		
		The general case is proved in the similar way, which we leave to the reader. 
	\end{proof}
	
	Next we discuss the Frobenius structure on overconvergent isocrystals. Recall that we have Frobenius endomorphism $\sigma : \Ed_K \to \Ed_K$. For a $k((t))$-variety $X$, let $X^{(q)}$ be the $q$-th Frobenius twist of $X$ over $k((t))$; then we have the $\sigma$-linear base extension functor $\sigma^* : \Isoc(X/\Ed_K) \to \Isoc(X^{(q)}/\Ed_K)$. Composing with the pullback functor $\Isoc(X^{(q)}/\Ed_K) \to \Isoc(X/\Ed_K)$ by the relative Frobenius, we get the Frobenius pullback $F^* : \Isoc(X/\Ed_K) \to \Isoc(X/\Ed_K).$ This is compatible with the completion functor.
	
	\begin{proposition}
		If $k$ is perfect and $X$ is smooth, then $F^*$ is an equivalence of categories.
	\end{proposition}
	\begin{proof}
		The $q$-th absolute Frobenius on $k((t))$ is the composition of two morphisms $F_1, F_2 : k((t)) \to k((t))$ given by $F_1(\sum_n a_n t^n) = \sum_n a_n^q t^n$ and $F_2(\sum_n a_n t^n) = \sum_n a_n t^{qn}.$ Similarly the Frobenius endomorphism $\sigma$ is the composition of two morphisms $\sigma_1, \sigma_2 : \Ed_K \to \Ed_K$ given by $\sigma_1(\sum_n a_n t^n) = \sum_n \sigma_K(a_n) t^n$ and $\sigma_2(\sum_n a_n t^n) = \sum_n a_n \sigma(t)^n.$ Then $\sigma^* = \sigma_1^* \sigma_2^*$. Since $k$ is perfect, the functor $\sigma_1^*$ is an isomorphism. $\sigma_2^*$ is an isomorphism by Proposition \ref{prop:purely-insep-equiv}. Thus $\sigma^*$ is also an isomorphism.
		
		It remains to show that pullback by the relative Frobenius $X \to X^{(q)}$ is an isomorphism. This is generalized to the statement that if $f : X \to Y$ is faithfully flat, then $f^* : \Isoc(Y/\Ed_K) \to \Isoc(X/\Ed_K)$ is an isomorphism. This can be proved in the same way as \cite{lazda2022note}.
	\end{proof}
	
	\section{Log $\nabla$-modules relative to $\Ed_K$} \label{sec:log-nabla-mod}
	In this section, we define log $\nabla$-modules without derivations by $t$ and prepare necessary theorems as in \cite{kedlaya2007semistable}.
	
	\subsection{Polyannuli}
	We start by giving the partial dagger version of the definition of polyannuli.
	\begin{definition}[{\cite[Definition 3.1.1]{kedlaya2007semistable}}] \label{def:aligned-subint}
		A subinterval $I$ of $[0, \infty)$ is said to be aligned if any endpoint at which $I$ is closed is contained in $\Gamma.$
		A subinterval $I$ of $[0, \infty)$ is said to be quasi-open if $I = (a, b)$ or $I = [0, b).$ 
	\end{definition}
	
	\begin{definition} \label{def:polyannuli}
		Let $I \neq [0, 0]$ be an aligned subinterval of $[0, \infty)$. We define the polyannulus over $\Ed_K$ to be
		$$A^d_{\Ed_K}(I) = \{ x \in \mathbb{A}^{d, \ad}_{\V[[t]]_K} : |t|_{[x]} \geq 1, \, |z_1|_{[x]}, \dots, |z_d|_{[x]} \in I\}$$
		where $z_1, \dots, z_d$ is the coordinate of $\mathbb{A}^d$. When $I = [\alpha, \beta]$ is a closed aligned subinterval, we define the polyannulus over $\E_K$ by
		$$A^d_{\E_K}(I) = \{ x \in \mathbb{A}^{d, \ad}_{\E_K} : \, \alpha \leq |z_1|_{x}, \dots, |z_d|_{x} \leq \beta \}$$
		and otherwise we define
		$$A^d_{\E_K}(I) = \bigcup_J A^d_{\E_K}(J)$$
		where $J$ runs over all closed aligned subinterval of $I$.
		$A^d_{\Ed_K}(I)$ is a partial dagger space and $A^d_{\E_K}(I)$ is an interpretation of \cite[Definition 3.1.2]{kedlaya2007semistable} in adic spaces and is a completion of $A^d_{\Ed_K}(I)$. We follow \cite[Convention 3.1.3]{kedlaya2007semistable} and simply write $A^d_{\Ed_K}[0, 1)$ for example.
	\end{definition}
	
	\begin{remark}
		For a closed aligned subinterval $I$ of $[0, \infty)$, it is easy to see that $A^d_{\E_K}(I)$ is a completion of $A^d_{\Ed_K}(I)$.
		When $I$ is not closed, say $I = [\alpha, \beta)$, then
		$$A^d_{\E_K}(I) = \{ x \in \mathbb{A}^{d, \ad}_{\E_K} : \alpha \leq |z_1|_x, \dots, |z_d|_x, \ |z_1|_{[x]}, \dots, |z_d|_{[x]} < \beta \}.$$
		In this case $A^d_{\Ed_K}(I)$ is the union of open subsets
		$$\{ x \in \mathbb{A}^{d, \ad}_{\V[[t]]_K}: |t|_{[x]} \geq 1, \, \alpha \leq |z_1|_{[x]}, \dots, |z_d|_{[x]}, \ |z_1|_x, \dots, |z_d|_x \leq \gamma \}$$
		for $\gamma \in (\alpha, \beta) \cap \Gamma^*$, whose completions are
		$$\{ x \in \mathbb{A}^{d, \ad}_{\E_K}: \alpha \leq |z_1|_{x}, \dots, |z_d|_{x} \leq \gamma \},$$
		whose union is $A^d_{\E_K}(I).$ Thus also in this case, $A^d_{\E_K}(I)$ is a completion of $A^d_{\Ed_K}(I)$. In fact, this hold for any type of aligned subintervals of $[0, \infty).$
	\end{remark}
	
	\begin{example} \label{ex:robba}
		For $\lambda_0 \in (0, 1) \cap \Gamma^*$, $A^d_{\Ed_K}[\lambda_0, 1)$ is covered by open subsets
		$$ U_\lambda = \{ x \in \mathbb{A}^{d, \ad}_{\V[[t]]_K} : |t|_{[x]} \geq 1, \, \lambda_0 \leq |z_i|_{[x]}, \, |z_i|_x \leq \lambda \} $$
		for $\lambda \in (\lambda_0, 1) \cap \Gamma^*.$
		$U_\lambda$ has a cofinal system of open neighborhoods in the affinoid space $A^{d, \ad}_{\V[[t]]_K}$ given by
		$$ U_{\lambda, \epsilon} = \{ x \in \mathbb{A}^{d, \ad}_{\V[[t]]_K} : |t|_{x} \geq 1-\epsilon, \, \lambda_0 - \epsilon \leq |z_i|_x \leq \lambda \} $$
		for $\epsilon \in (0, 1) \cap \Gamma^*$.
		Thus
		\begin{align*}
			\Gamma(A^d_{\Ed_K}[\lambda_0, 1), \O) & = \varprojlim_{\lambda} \Gamma(U_{\lambda}, \O) \\
			& = \varprojlim_{\lambda < 1} \varinjlim_{\epsilon} \Gamma(U_{\lambda, \epsilon}, \O)\\
			& = \bigcap_{\lambda < 1} \bigcup_{\epsilon} \E_{1-\epsilon} \langle z_1, \dots, z_d \rangle_{\lambda_0 - \epsilon \leq |z_i| \leq \lambda}.
		\end{align*}
		When $d = 1$, the ring $\bigcup_{\lambda_0 < 1} \Gamma(A^d_{\Ed_K}[\lambda_0, 1), \O)$ is the Robba ring $\R_{\Ed_K}$ over $\Ed_K$ defined in \cite[Definition 3.34]{lazda2016rigid}. Indeed, they define
		$$\R_{\Ed_K} = \bigcup_{\lambda_0 < 1} \bigcap_{\lambda_0 \leq \rho < 1} \bigcup_{\eta} \E_\eta \langle z \rangle_{|z| = \rho}.$$
		It is easy to see that two rings coincide.
		
		Note that in the calculation of $\Gamma(A^d_{\Ed_K}[\lambda_0, 1), \O)$, it is important to expand injective and projective limit in this order, which is a consequence of Proposition \ref{prop:germ-section}.
	\end{example}
	
	\subsection{Log connections}
	\begin{definition}
		Let $f: (X, \Xs) \to (Y, \Ys)$ be a morphism of partial dagger spaces. We say that $f$ is smooth if there is a representation
		\[
		\xymatrix{
			(X', \Us) \ar[d]_{j} \ar[rd]^{f'} & \\
			(X, \Xs) \ar@{..>}[r]^{f} & (Y, \Ys)
		}
		\]
		of $f$ by morphisms of pregerms with $j$ a strict neighborhood and $f' : \Us \to \Ys$ smooth as a morphism of rigid analytic spaces.
		We say that $x_1, \dots, x_n \in \Gamma(X, \O_X)$ is a coordinate of $f$ if $x_1, \dots, x_n$ comes from some $x'_1, \dots, x'_n \in \Gamma(\Us, \O_\Us)$ such that the morphism
		$$\Us \to \Ys \times_{\V[[t]]_K} \Spa \V[[t]]_K \langle z_1, \dots, z_n \rangle$$
		induced by $z_i \mapsto x'_i$ is etale.
		
		A partial dagger space $X$ is said to be smooth if the structure morphism $X \to \Spd \Ed_K$ is smooth.
	\end{definition}
	\begin{remark}
		If $X$ is a smooth partial dagger space, then so is the completion $\wh{X}$ over $\E_K.$
	\end{remark}
	\begin{example}
		Let $\Yf = \Spf \V[[t]] \langle x_1, \dots, x_m \rangle / (t - x_1 \cdots x_m)$, $\Ys = \Yf^\ad$ and $Y = \{ x \in \Ys : |t|_{[x]} \geq 1 \}.$ $(Y, \Ys)$ is a partial dagger space and $(Y, \Ys) \to \Spd \Ed_K$ is smooth with coordinate $x_2, \dots, x_m$.
	\end{example}
	
	\begin{definition}[{cf. \cite[Definition 2.3.5]{kedlaya2007semistable}}] \label{def:log-diffs}
		Let $f : V \to W$ be a morphism of partial dagger spaces and $x_1, \dots, x_n \in \Gamma(V, \O_V)$. We define the module of log differentials $\Omega^{1, \log}_{V/W}$ with respect to $x_1, \dots, x_n$ to be the quotient of $\Omega^1_{V/W} \oplus \O_V \dlog x_1 \oplus \cdots \oplus \O_V \dlog x_n$, with each $\dlog x_i$ a formal symbol, divided by the $\O_V$-submodule generated by $x_i \dlog x_i - dx_i.$
	\end{definition}
	
	\begin{definition}[{cf. \cite[Definition 2.3.7, Definition 2.3.8]{kedlaya2007semistable}}] \label{def:log-nabla-mod}
		In the situation of Definition \ref{def:log-diffs}, a log $\nabla$-module over $V$ relative to $W$ (or simply on $V/W$) is a locally free $\O_V$-module $\E$ equipped with an integrable log connection $\nabla : \E \to \E \times \Omega^{1, \log}_{V/W}.$
		A section $v \in \Gamma(V, \E)$ is said to be horizontal relative to $W$ if $\nabla v = 0.$
	\end{definition}
	
	\begin{definition}[{cf. \cite[Definition 2.3.9]{kedlaya2007semistable}}] \label{def:residue-map}
		In the situation Definition \ref{def:log-nabla-mod}, $\nabla$ induces an $\O$-linear morphism $\E \to \E \otimes \O \dlog x_i \simeq \E$ on each zero locus $V(x_i)$ as explained in \cite[Definition 2.3.9]{kedlaya2007semistable}. This map is called the residue map along $V(x_i)$.
	\end{definition}
	
	\subsection{Constant and unipotent connections}
	In this section, we will prove the statements regarding log $\nabla$-modules on partial dagger spaces similar to those in \S3.2, \S3.3 of \cite{kedlaya2007semistable}.
	
	\begin{hypothesis} \label{hypo:smc}
		Let $f : V \to W$ be a smooth morphism of partial dagger spaces and suppose that $x_1, \dots, x_n \in \Gamma(V, \O)$ is a coordinate of $V$ relative to $W$ whose zero loci are smooth and meet transversely. 
	\end{hypothesis}
	
	\begin{definition}
		Under Hypothesis \ref{hypo:smc}, for an open subset $X$ of $V \times_{\Ed_K} A^d_{\Ed_K}[0, 1)$ we define $\LNM_{X/W}$ to be the category of log $\nabla$-modules over $V$ relative to $W$ with respect to $x_1, \dots, x_n, z_1, \dots, z_d$. When $W = \Spd \Ed_K$, it is abbreviated to $\LNM_{X}$.
	\end{definition}
	
	\begin{definition}
		Under Hypothesis \ref{hypo:smc}, an object $\E$ of $\LNM_{X/W}$ with $X = V \times_{\Ed_K} A^d_{\Ed_K}[0, 1)$ is said to be constant if $\E$ is isomorphic to $\pi^* \F$ for some $\F \in \LNM_{V/W}$ where $\pi$ is the projection. $\E$ is said to be unipotent if it is an iterated extension of constant objects. The full subcategory of $\LNM$ consisting of unipotent objects is denoted by $\ULNM$.
	\end{definition}
	
	The categorical equivalence of Proposition \ref{prop:tube-sheaf}-(iii) can be upgraded to log $\nabla$-modules with nilpotent residues, since it preserves locally freeness.
	
	\begin{proposition} \label{prop:lnm-colim}
		Under Hypothesis \ref{hypo:smc}, suppose that $V \to W$ comes from a morphism of pregerms $(V, \Vs) \to (W, \Ws)$ and there exist cofinal systems of open neighborhoods $\{ \Vs_\lambda \}$ of $V$ in $\Vs$ and $\{ \Ws_\lambda \}$ of $W$ in $\Ws$ such that $\Vs_\lambda$ is mapped to $\Ws_\lambda$. Then we have canonical equivalence of categories
		$$
		\varinjlim_\lambda \LNM_{\Vs_\lambda/\Ws_\lambda} \xrightarrow{\sim} \LNM(V/W).
		$$
		Let $\Xs_\lambda$ be a cofinal system of open neighborhoods of $X = V \times A^d_{\Ed_K}(I)$ in some ambient rigid analytic space $\Xs$. Then for an object $E \in \LNM(X/W)$, it is constant/unipotent if and only if there is some constant/unipotent object $\E_\lambda \in \LNM(\Xs_\lambda/\Ws_\lambda)$ which gives rise to $E$.
	\end{proposition}
	
	\begin{definition}
		$\LNM_{V \times A^d[0, 0]}$ is defined to be the category of log $\nabla$-modules over $V/W$ equipped with $d$ nilpotent $\O_V$-linear endomorphisms $\partial_i = \frac{\partial}{\partial z_i}$ commuting with each other and $\nabla$. An object $(\E, \nabla, \partial_1, \dots, \partial_d)$ is said to be constant if $\partial_1, \dots, \partial_d$ are zero maps. An object is said to be unipotent if it is an iterated extension of constant objects.
	\end{definition}
	\begin{definition}
		Under Hypothesis \ref{hypo:smc}, define the functor
		$$ \Uc_I : \LNM_{V \times A^d[0, 0]/W} \to \LNM_{V \times A^d(I)/W} $$
		by
		$$ (\E, \nabla, \partial_1, \dots, \partial_d) \mapsto (\pi^* \E, \nabla') $$
		with
		$$\nabla' : v \mapsto (\pi^* \nabla)v + \sum_i (\pi^*\partial_i)v \otimes \dlog z_i.$$
		$\nabla'$ is integrable since $\partial_i$ and $\nabla$ commute with each other and the residues are nilpotent since $\partial_i$ are nilpotent. Note that $\Uc_I$ preserves constance and unipotence.
	\end{definition}
	
	\begin{lemma} \label{lem:lnm-abel}
		Let $X$ be an smooth partial dagger space with coordinates $x_1, \dots, x_n$ whose zero loci are smooth and meet transversely. Then for any morphism $f : \E \to \F$ of $\LNM_X$, the kernel and cokernel are locally free.
	\end{lemma}
	\begin{proof}
		We may suppose that $X = \Spd A$ is affinoid and let $\mathcal{K} = \ker f, \, K = \Gamma(X, \mathcal{K})$ and $ \wh{\mathcal{K}} = \mathcal{K}|_{\wh{X}}, \, \wh{K} = \Gamma(X, \wh{\mathcal{K}}) = K \otimes_A \wh{A}.$ By \cite[Lemma 3.2.14]{kedlaya2007semistable}, $\wh{\mathcal{K}}$ is locally free on $\wh{X}$ and by the proof of \cite[Lemma 3.2.13]{kedlaya2007semistable}, this is equivalent to saying that $\mathcal{K}$ is a locally free $\wh{A}$-module. Since $A \to \wh{A}$ is faithfully flat by Corollary \ref{cor:pda}, $K$ is a locally free $A$-module. This implies that $\mathcal{K}$ is locally free on $X$ by the categorical equivalence of Proposition \ref{prop:theoremAB}.
	\end{proof}
	
	\begin{proposition} \label{prop:lnm-abel}
		For any smooth partial dagger space $V$ and a closed aligned subinterval $I$ of $[0, \infty)$, let $X = V \times_{\Ed_K} A^d_{\Ed_K}(I)$; then $\LNM_X$ is an abelian tensor category and the full subcategory $\ULNM_X$ is an abelian tensor subcategory.
	\end{proposition}
	\begin{proof}
		$\LNM_X$ being an abelian category is a consequence of Lemma \ref{lem:lnm-abel}. In order to show the statement of $\ULNM$, we follow the proof of \cite[Proposition 3.2.20]{kedlaya2007semistable} word by word from the first to the third paragraph. This leads us to showing that the maps $\E' \to \E / (\pi^* \F)$ and $\pi^* \F' \to \E / \E'$ are zero. This is true on $\wh{X}$ by the proof of {\it loc. cit.}, and by using faithfully flatness as in Lemma \ref{lem:lnm-abel}, we conclude that the maps are zero on $X$.
	\end{proof}
	
	\begin{lemma} \label{lem:ui-cat-eq1}
		Under Hypothesis \ref{hypo:smc}, let $I$ be an aligned subinterval of $[0, \infty)$, $X = V \times A^d_{\Ed_K}$ and $\mathcal{D}_{X/W}$ the sheaf of rings of (finite order) $\O_W$-linear log differential operators on $X$. Let $\E$ and $\E'$ be left $\mathcal{D}_{X/W}$-modules with $\E$ coherent and flat over $\O_X.$ Then there exists a natural isomorphism
		$$ \Ext_{\mathcal{D}_{X/W}}(\E, \E') \simeq \mathbb{H}^i(X, \E^{\vee} \otimes \E' \otimes \Omega^{\cdot, \log}_{X/W})$$
		where $\mathbb{H}^*$ denotes hypercohomology.
	\end{lemma}
	\begin{proof}
		As mentioned in the paragraph before \cite[Lemma 3.3.1]{kedlaya2007semistable}, this is a relative analogue of \cite[Proposition 1.1.2]{chiarellotto1999pentes}, whose proof is straightforward.
	\end{proof}
	
	\begin{lemma} \label{lem:ui-cat-eq2}
		Under Hypothesis \ref{hypo:smc}, let $I$ be a closed aligned subinterval of $[0, \infty)$. Then for any $\E, \E' \in \ULNM_{V \times A^d[0, 0] / W}$, the canonical homomorphism
		$$ \Ext^i(\E, \E') \to \Ext^i(\Uc_I(\E), \Uc_I(\E')) $$
		is an isomorphism.
	\end{lemma}
	\begin{proof}
		As in the proof of \cite[Lemma 3.3.2]{kedlaya2007semistable}, we reduce the lemma to prove that for $V = W$ an affinoid, the canonical map $g : \Gamma(V, \, \Omega^{\cdot, \log}_{V \times A^d[0, 0]/V}) \to \Gamma(V, \, \Omega^{\cdot, \log}_{V \times A^d(I)/V})$ is a chain homotopy equivalence, where $\Omega^{\cdot, \log}_{V \times A^d[0, 0]}$ is freely generated by $\dlog z_1, \dots, \dlog z_d$ over $\O_V$. (Note that Theorem A and B hold for affinoid partial dagger spaces, by Proposition \ref{prop:theoremAB}).
		
		Let $h : \Gamma(V, \,\Omega^{\cdot, \log}_{V \times A^d(I)/V}) \to \Gamma(V, \,\Omega^{\cdot, \log}_{V \times A^d[0, 0]/V})$ be the map taking constant coefficient of each formal power series. $h \circ g$ is the identity map; we construct the chain homotopy $s$ between $g \circ h$ and the identity map. One such map, as in the proof of {\it loc. cit.}, is constructed in the following way:
		Given a $k$-form $z_1^{i_1} \cdots z_d^{i_d} \dlog z_{j_1} \wedge \cdots \wedge \dlog z_{j_k}$, let $h$ be the smallest positive integer with $i_h \neq 0$ and integrate it against $\dlog z_h = dz_h / z_h$ if $h \in \{ j_1, \dots, j_k \}$, i.e., $z_h^{i_h} \dlog z_h$ is replaced by $\frac{1}{i_h} tzh^{i_h}$; otherwise it is mapped to zero. In order to check that this construction is well-defined, let $\{ A_\eta \}_{1 < \eta < \eta_0}$ be a fringe inductive system of $A = \Gamma(X, \O_X)$, $\mathscr{V}_\eta = \Spa A_{\eta}$ and let $I = [\alpha, \beta]$. Then $V \times A^d_{\Ed_K}[\alpha, \beta]$ admits a cofinal system of open neighborhoods $\Vs_{\eta, \lambda, \epsilon} := \Vs_\eta \times A^d_{\E_\lambda}[\alpha - \epsilon, \beta + \epsilon]$ where $\eta \in (1, \eta_0] \cap \Gamma^*, \ \lambda \in (0, 1) \cap \Gamma^*, \ \epsilon \in (0, \infty] \cap \Gamma^*$ and $A^d_{\E_\lambda}[\alpha-\epsilon, \beta+\epsilon] = \{ x \in \mathbb{A}^{d, \ad}_{\E_\lambda} : \alpha - \epsilon \leq |z_i|_x \leq \beta + \epsilon \}. $
		For an element
		$$\displaystyle \sum_{j_1 < \cdots < j_k} \sum_{\underline{i}} a_{\underline{i}, j_1, \dots, j_k} \underline{z}^{\underline{i}} \, \dlog z_{j_1} \wedge \cdots \wedge z_{j_k} $$
		of $\Gamma(V, \, \Omega^{\cdot, \log}_{V \times A^d(I)/V})$ with $a_{\underline{i}, j_1, \dots, j_k} \in A$, there is some $\eta, \lambda, \epsilon$ in the range described above such that all $a_{\underline{i}, j_1, \dots, j_k}$ belong to $\Gamma(\Vs_{\eta, \lambda, \epsilon}, \O)$ and that for any fixed $j_1 < \cdots < j_k$, $\| a_{\underline{i}, j_1, \dots, j_k} \underline{z}^{\underline{i}}\|_{\eta, \lambda, \epsilon}$ converges to zero where $\| \cdot \|_{\eta, \lambda, \epsilon}$ is the canonical Banach norm on $\Gamma(\Vs_{\eta, \lambda, \epsilon}, \O)$. Since $s$ only multiplies $\frac{1}{i_h}$, the growth of norm after applying $s$ is bounded after changing $\epsilon$ to a smaller $\epsilon'$. Hence $s$ maps $\Gamma(V, \, \Omega^{\cdot, \log}_{V \times A^d(I)/V})$ to itself.
	\end{proof}
	
	\begin{remark}
		Differing from the classical case \cite[Lemma 3.3.2]{kedlaya2007semistable}, the aligned subinterval is closed in the above lemma. This is because we are dealing with partial dagger spaces. In the classical setting, the ring of series in $K[[z]]$ convergent for $|z| < 1$ is $\bigcap_{\lambda < 1} K \langle z \rangle_{|z| \leq \lambda}$;
		given a series $\sum_i a_i z^i \in K[[z]]$ that converges for $|z| \leq \lambda$, the integration $\sum_i \frac{a_i}{i+1} z^{i+1}$ converges for $|z| \leq \lambda'$ if $\lambda' < \lambda$; hence the integration defines a map from $\Gamma(A^1_K[0, 1), \O)$ to itself. On the other hand, in our setting, $\Gamma(A^1_{\Ed_K}[0, 1], \O) = \bigcup_{\eta > 1} \E_{\eta^{-1}} \langle z \rangle_{|z| \leq \eta}.$ The integration of a series of $\E_{\eta^{-1}} \langle z \rangle_{|z| \leq \eta}$ belongs to $\E_{\eta'^{-1}} \langle z \rangle_{|z| \leq \eta'}$ for $\eta > \eta' > 1$; hence again the integration defines a map from $\Gamma(A^1_{\Ed_K}[0, 1], \O)$ to itself.
	\end{remark}
	
	\begin{lemma} \label{lem:ui-eq}
		Under Hypothesis \ref{hypo:smc}, let $\eta \in [0, 1) \cap \Gamma$ and $\E \in \LNM_{V \times A^d_{\Ed_K}[\eta, 1) / W}.$ Then the following conditions are equivalent:
		\begin{enumerate}[(i)]
			\item $\E$ is unipotent;
			\item $\E_\lambda = \E|_{X \times A^d[\eta, \lambda]}$ is unipotent for all $\lambda \in (\eta, 1) \cap \Gamma^*$;
			\item $\E$ is in the essential image of $\ULNM_{V \times A^d[0, 0]}$ under $\Uc_{[\eta, 1)}.$
		\end{enumerate}
	\end{lemma}
	\begin{proof}
		(iii) $\Rightarrow$ (i) $\Rightarrow$ (ii) is trivial. We have to prove (ii) $\Rightarrow$ (iii). By Lemma  \ref{lem:ui-cat-eq2}, there exists some $\F_\lambda \in \ULNM_{V \times A^d[0, 0] / W}$ and an isomorphism $\phi_\lambda : \E_\lambda \to \Uc_{[\eta, \lambda]}(\F_\lambda)$. For any elements $\lambda \geq \rho$ of $(\eta, 1) \cap \Gamma^*$, define the isomorphism $\psi_{\lambda, \rho} : \F_\lambda \to \F_\rho$ to be the one makes following diagram commutative:
		\[\xymatrixcolsep{12mm}
		\xymatrix{
			\E_\lambda|_{V \times A^d[\eta, \rho]} \ar[r]^(0.42){\phi_\lambda} \ar@{=}[d] & \Uc_{[\eta, \lambda]}(\F_\lambda)|_{V \times A^d[\eta, \rho]} \ar@{=}[r] & \Uc_{[\eta, \rho]}(\F_\lambda) \ar[d]^{\Uc_{[\eta, \rho]}(\psi_{\lambda, \rho})}\\
			\E_\rho \ar[rr]^{\phi_\rho} & & \Uc_{[\eta, \rho]}(\F_\rho).
		}
		\]
		It obviously follows that $\psi_{\lambda, \rho}$ satisfies the cocycle condition. Choose any $\lambda \in (\eta, 1) \cap \Gamma^*$; then the isomorphisms $\Phi_\rho : \Uc_{[\eta, \rho]}(\psi^{-1}_{\lambda, \lambda}) \circ \phi_\rho : \E_\rho \to \Uc_{[\eta, \rho]}(\F_\lambda) = \Uc_{[\eta, 1)}(\F_\lambda)|_{V \times A^d[\eta, \rho]}$ can be glued to an isomorphism $\Phi : \E \to \Uc_{[\eta, 1)}(\F_\lambda).$
	\end{proof}
	
	\begin{remark}
		This lemma only treat aligned subintervals of the form $[\eta, 1)$, but obviously the lemma holds for any aligned subintervals of $[0, \infty)$.
	\end{remark}
	
	\begin{theorem} \label{thm:ui-cat-eq}
		Let $I \neq [0, 0]$ be an aligned subinterval of $[0, \infty).$ Under Hypothesis \ref{hypo:smc}, the functor
		$$ \Uc_I : \ULNM_{V \times A^d[0, 0]} \to \ULNM_{V \times A^d(I)} $$
		is an equivalence of categories.
	\end{theorem}
	\begin{proof}
		The case where $I$ is closed follows from Lemma \ref{lem:ui-cat-eq1}. For simplicity, we assume $I = [\eta, 1).$ Essential surjectivity follows from Lemma \ref{lem:ui-eq}. Full faithfulness can be easily proved by using the closed case.
	\end{proof}
	
	\begin{corollary} \label{cor:unip-loc}
		Let $I \neq [0, 0]$ be an aligned subinterval of $[0, \infty).$ Under Hypothesis \ref{hypo:smc} and $W = \Spd \Ed_K$, an object of $\LNM_{V \times A^d(I)}$ being unipotent can be checked locally on $V$.
	\end{corollary}
	
	We finish this section by stating the generalization of \cite[Proposition 3.3.8]{kedlaya2007semistable} with extra restrictions.
	
	\begin{lemma}[{(cf. \cite[Lemma 3.2.8]{kedlaya2007semistable})}]
		Under Hypothesis \ref{hypo:smc}, suppose that $V$ is an affinoid partial dagger space, that $X = V \times A^1_{\Ed_K}[0, a]$ and that $\E \in \LNM_{X/V}$ is such that the restriction of $\E$ to $V \times \{ 0 \}$ is free. Then there is $b \in (0, a] \cap \Gamma^*$ such that restriction of $\E$ to $V \times A^1_{\Ed_K}[0, b]$ is in the essential image of $\Uc_{[0, b]}.$ In particular if the residue of $\E$ along $V \times \{ 0 \}$ vanishes, then the restriction of $\E$ to $V \times A^1_{\Ed_K}[0, b]$ is constant.
	\end{lemma}
	\begin{proof}
		Choose a cofinal system of open neighborhoods $\{ \Vs_\eta \}_{\eta > 1}$ of $V$, some $\eta_0 \in (1, \infty) \cap \Gamma^*$, $\alpha \in (a, \infty) \cap \Gamma^*$ and a locally free log $\nabla$-module $\E_{\eta_0}$ over $\Xs_{\eta_0, \alpha} = \Vs_{\eta_0} \times A^1_{\V[[t]]_K}[0, \alpha]$ with nilpotent residues which gives rise to $\E$. We may also assume that $\E_{\eta_0}$ is free with basis $e_1, \dots, e_n \in \Gamma(\Xs_{\eta_0, \alpha}, \E_{\eta_0})$, by the similar argument as in \cite[Lemma 3.2.8]{kedlaya2007semistable}.
		
		We can follow the argument of \cite[Lemma 3.2.8]{kedlaya2007semistable} to construct an invertible matrix $M$ over $\Gamma(\Xs_{\eta_0, \beta})$ for some $\beta \in (0, \alpha] \cap \Gamma^*$ such that the basis $v_i = \sum_j M_{ij}e_j$ of $\E_{\eta_0}$ over $\Xs_{\eta_0, \beta}$ satisfies $\partial v_j = \sum_k f_{j,k} v_k$ with $f_{j, k} \in \Gamma(\Vs_{\eta_0}, \O).$ Therefore $\E$ is described by the image under $\Uc_I$ of the $\Gamma(\Vs_{\eta_0}, \O)$-span of $v_1, \dots, v_n$.  
	\end{proof}
	
	\begin{proposition} \label{prop:ext-log-sub}
		Let $X$ be an smooth partial dagger space with a coordinate $x_1, \dots, x_n$ whose zero loci are smooth and meet transversely and let $U$ be the complement of the zero loci. Suppose that each zero locus $Z_i = V(x_i)$ locally admits an open neighborhood of the form $Z_i \times A^1_{\Ed_K}[0, \eta)$ in $X$; then for any $\E \in \LNM_X$ and a subobject $\F$ of $\E|U$ in $\LNM_U$, $\F$ uniquely extends to a subobject of $\E$ in $\LNM_X$.
	\end{proposition}
	\begin{proof}
		We can apply the argument as in the proof of \cite[Proposition 3.3.8]{kedlaya2007semistable}.
	\end{proof}
	
	\subsection{Descent of unipotence} \label{subsec:descent}
	In this subsection, we prove the relative version of ``descent of unipotence'', namely a generalization of \cite[Proposition 5.4.1]{kedlaya2006finiteness} and \cite[Proposition 3.42]{lazda2016rigid}\footnote{To be precise, the statement of \cite[Proposition 3.42]{lazda2016rigid} is not true. This is because $\nabla$-modules over the Robba ring $\R_A$ relative to $A$ are defined to be free $\R_A$-module of finite rank, but the $A$-span of horizontal sections may not be free over $A$ and thus it may be necessary to localize $A$ as in \cite[Proposition 5.3.3]{kedlaya2006finiteness}. Fortunately, this error does not affect the content of \cite{lazda2016rigid} since this proposition is used to prove finiteness of rigid cohomology as in \cite{kedlaya2006finiteness}. In our case, we avoid this error by considering log $\nabla$-modules over $\R_A$ relative to $\Ed_K$.}.
	
	In this subsection, we consider the following situation.
	
	\begin{hypothesis} \label{hypo:pda1}
		Let $A$ be an smooth partial dagger algebra with coordinate $x_1, \dots, x_n$ whose zero loci are smooth and meet transversely, let $\{ A_\rho \}_{1 < \rho \leq \rho_0}$ be a fringe inductive system, $X_\rho = \Spd A_\rho$, and $\wh{X} = \Spa \wh{A}.$ 
	\end{hypothesis}
	
	\begin{definition}
		Under Hypothesis \ref{hypo:pda1}, for $\eta_0, \eta_1 \in \Gamma^*$ with $\eta_0 \leq \eta_1$, define
		$$\R_{A, \eta_0, \eta_1} = \Gamma(X \times A^d_{\Ed_K}[\eta_0, \eta_1], \O) = \bigcup_{\epsilon > 0} \bigcup_{\rho > 1} A_\rho \langle z_1, \dots, z_d \rangle_{[\eta_0 - \epsilon, \eta_1 + \epsilon]^d},$$
		$$\R_{\wh{A}, \eta_0, \eta_1} =  \Gamma(\wh{X} \times A^d_{\Ed_K}[\eta_0, \eta_1], \O) = \bigcup_{\epsilon > 0} \wh{A} \langle z_1, \dots, z_d \rangle_{[\eta_0-\epsilon, \eta_1+\epsilon]^d}.$$
		For $\eta \in \Gamma^*$, define
		$$\R^+_{A, \eta} = \Gamma(X \times A^d_{\Ed_K}[0, \eta], \O) = \bigcup_{\epsilon > 0} \bigcup_{\rho > 1} A_\rho \langle z_1, \dots, z_d \rangle_{[0, \eta_ + \epsilon]^d},$$
		$$\R^+_{\wh{A}, \eta} = \Gamma(\wh{X} \times A^d_{\Ed_K}[0, \eta], \O) = \bigcup_{\epsilon > 0}\wh{A} \langle z_1, \dots, z_d \rangle_{[0, \eta+\epsilon]^d}.$$
	\end{definition}
	
	\begin{remark} \label{rem:R-unip}
		By Theorem \ref{thm:ui-cat-eq}, any unipotent log $\nabla$-module over any of $\R_{A, \eta_0, \eta_1}, \R_{\wh{A}, \eta_0, \eta_1}, \R^+_{A, \eta}, \R^+_{\wh{A}, \eta}$ is induced by a log $\nabla$-module over $A$ or $\wh{A}$ with commuting nilpotent endomorphisms.
	\end{remark}
	
	\begin{lemma} \label{lem:unip-f}
		Under Hypothesis \ref{hypo:pda1}, let $\wh{R} = \R_{A, \eta_0, \eta_1}$ (resp. $\wh{R} = \R^+_{\wh{A}, \eta}$) for $\eta_0, \eta_1 \in \Gamma^*$ with $\eta_0 \leq \eta_1$ (resp. $\eta \in \Gamma^*$). Let $\wh{F}$ be a free log $\nabla$-module over $\wh{A}$ with nilpotent residues equipped with $d$ nilpotent commuting $\wh{A}$-linear endomorphisms $\partial_1, \dots, \partial_d$ and let $\wh{M} = \wh{F} \otimes_{\wh{A}} \wh{R}$ be the corresponding log $\nabla$-module on $\wh{R}$. Choose $i_1, \dots, i_{e-1} \in \{ 1, \dots, d \}$ such that $\partial_{i_1} \cdots \partial_{i_{e-1}} \neq 0$ and $\partial_{j_1} \cdots \partial_{j_e} = 0$ for any $j_1, \dots, j_e \in \{ 1, \dots, d \}.$ Define $\wh{A}$-linear map $f_l : \wh{M} \to \wh{M}$ inductively by
		$$ f_0(x) = \partial_{i_1} \cdots \partial_{i_{e-1}}x,$$
		$$ f_l(x) = \left(1 - \frac{\partial_1^2}{l^2} \right)^e \cdots \left(1 - \frac{\partial_d^2}{l^2} \right)^e f_{l-1}(x). $$
		If we choose any basis $v_1, \dots, v_n$ of $\wh{F}$, then for any $x \in \wh{M},$ there exist $a_1, \dots, a_n \in \wh{A}$ such that $f_l(x)$ converges to $\sum a_i v_i$.
		Moreover, if we denote the limit of $f_l(x)$ by $f(x)$, then $f : \wh{M} \to \wh{M}$ is $\wh{A}$-linear and
		$$\partial_{i_1} \cdots \partial_{i_{e-1}}\left(\sum \wh{A}v_i\right) \subseteq f(\wh{M}) \subseteq \left( \sum \wh{A}v_i \right) \cap \ker \partial_1 \cap \cdots \cap \ker \partial_d.$$
	\end{lemma}
	\begin{proof}
		This is just a rephrase of \cite[Lemma 3.4.1]{kedlaya2007semistable}, but for later use, we give a detailed proof.
		
		For $r = 1, \dots, d$, define $X_{rij} \in \wh{A}$ by $\partial_r v_i = \sum_j X_{rij}v_j$. Let $X_r = (X_{rij})$ be the $n\times n$ matrix; then we have $X_{i_1} \cdots X_{i_{e-1}} \neq 0$ and $X_{j_1} \cdots X_{j_e} = 0$ for any $j_1, \dots, j_e \in \{ 1, \dots, d \}$.
		
		For $x = \sum_i a_i v_i$ with $a_i = \sum_{\unu}a_{i\unu} \ul{z}^\unu \in \wh{R}$ and $a_{i\unu} \in \wh{A}$, write $f_l(x) = \sum_i a_i^{(l)} v_i$ with $a_i^{(l)} = \sum_{\unu}a_{i\unu}^{(l)} \ul{z}^\unu \in \wh{R}$ and $a_{i\unu}^{(l)} \in \wh{A}$; by formal series calculation we get
		
		\begin{align} 
			\begin{pmatrix}
				a_{1\unu}^{(l)} \\
				\vdots \\
				a_{n\unu}^{(l)}
			\end{pmatrix}
			=
			(\nu_{i_1} + X_{i_1}) \cdots (\nu_{i_{e-1}} + X_{i_{e-1}}) \prod_{r = 1}^{d} \prod_{j = 1}^{l} \left(1 - \frac{(\nu_r + X_r)^2}{j^2} \right)^e
			\begin{pmatrix}
				a_{1\unu} \\
				\vdots \\
				a_{n\unu}
			\end{pmatrix} \label{expr:flx}.
		\end{align}
		
		When $\unu = \ul{0} = (0, \dots, 0),$ since products of any of $e$ matrices among $X_1, \dots, X_d$ are zero, $a_{i \ul{0}}^{(l)}$ is independent of $l$. When $0 < |\nu_r| \leq l$ for some $r = 1, \dots, d$, then a factor
		$$ \left(1 - \frac{(\nu_r + X_r)^2}{\nu_r^2} \right)^e $$
		appeared in (\ref{expr:flx}) is a multiple of $X_r^e = 0$; hence $a_{i\unu}^{(l)} = 0$.
		Suppose $|\nu_r| > l$ for all $r = 1, \dots, d$. Then as in \cite[Lemma 3.4.1]{kedlaya2007semistable}, we can use \cite[Lemma 5.2.7]{kedlaya2006finiteness} to show that the matrix
		$$ (\nu_{i_1} + X_{i_1}) \cdots (\nu_{i_{e-1}} + X_{i_{e-1}}) \prod_{r = 1}^{d} \prod_{j = 1}^{l} \left(1 - \frac{(\nu_r + X_r)^2}{j^2} \right)^e$$
		appeared in (\ref{expr:flx}) is a sum of the form $k X_{j_1} \cdots X_{j_{e-1}}$ with $k$ a rational number with $\|k\| \leq p^{(e-1)e}(|\nu_1| + \cdots + |\nu_d| + l + 1)^{de}.$ Here $\| \cdot \|$ is a Banach norm on $\widehat{A}$. Thus for $c = \max(1, \| X_1 \|, \dots, \| X_d \|)$, we get
		$$ \left\| (\nu_{i_1} + X_{i_1}) \cdots (\nu_{i_{e-1}} + X_{i_{e-1}}) \prod_{r = 1}^{d} \prod_{j = 1}^{l} \left(1 - \frac{(\nu_r + X_r)^2}{j^2} \right)^e \right\| \leq c^{e-1} p^{(e-1)e}(|\nu_1| + \cdots + |\nu_d| + l + 1)^{de}.$$
		Suppose that we are in the case $\wh{R} = \R_{\wh{A}, \eta_0, \eta_1}$ (the case $\wh{R} = \R^+_{\wh{A}, \eta}$). For $0 < \alpha < \beta$, write $[\alpha, \beta]^\unu =  \max(\alpha^{\nu_1}, \beta^{\nu_1})\cdots \max(\alpha^{\nu_d}, \beta^{\nu_d})$. as $\sum_i a_{i, \unu} \ul{z}^\unu \in \wh{R},$ there is some $\epsilon > 0$ such that $\max_i \| a_{i\unu} \| \, [\eta_0-\epsilon, \eta_1+\epsilon]^\unu \to 0.$ Let $\| \cdot \|_{[\alpha, \beta]^d}$ denote the canonical Banach norm on $\widehat{A}\langle z_1, \dots, z_d \rangle_{[\alpha, \beta]^d}$. For any $0 < \epsilon' < \epsilon$, we have
		\begin{align*}
			\max_i \| a_i^{(l)} - a_{i \ul{0}}^{(l)} \|_{[\eta_0 - \epsilon', \eta_1 + \epsilon']^d} & = \max_{|\nu_1|, \dots, |\nu_d| > l} \max_i \| a_{i \unu}^{(l)}\| \, [\eta_0 - \epsilon', \eta_1 + \epsilon']^\unu \\
			& \leq \max_{|\nu_1|, \dots, |\nu_d| > l} \max_i \, C (|\unu| + l+1)^{de} \|a_{i\unu}\| \, [\eta_0 - \epsilon', \eta_1 + \epsilon']^\unu \\
			& \leq \max_{|\nu_1|, \dots, |\nu_d| > l} \max_i \, C (2|\unu| + 1)^{de} \|a_{i\unu}\| \, [\eta_0 - \epsilon', \eta_1 + \epsilon']^\unu \\
			& \leq \max_{|\nu_1|, \dots, |\nu_d| > l} \left( \left( C (2|\unu| + 1)^{de} \frac{[\eta_0 - \epsilon', \eta_1 + \epsilon']^\unu}{[\eta_0 - \epsilon, \eta_1 + \epsilon]^\unu} \right) \left( \max_i \|a_{i\unu}\|\,[\eta_0 - \epsilon, \eta_1 + \epsilon]^\unu \right) \right),
		\end{align*}
		for some constant $C$. It is easy to see that the last converges to zero as $l \to \infty$. Since $a_{i\ul{0}}^{(l)}$ is independent of $l$, we conclude that $f_l(x)$ converges to $\sum_i a_{i\ul{0}}^{(0)} v_i$.
		To summarize, if $x = \sum a_i v_i$ and $a_i = \sum_\unu a_{i\unu} \ul{z}^\unu \in \wh{\R}$ with $a_{i\unu} \in \wh{A}$, then $f(x) = \sum_i \partial_{i_1} \cdots \partial_{i_{e-1}} a_{i\unu} v_i.$ This proves the last statement.
	\end{proof}
	
	\begin{lemma} \label{lem:hadamar}
		Under Hypothesis \ref{hypo:pda1}, let $A = W^\dag_{d, e} / I$ be a presentation of a partial dagger algebra $A$. Suppose that there exist some ideal $I_0$ of $W^\dag_{d, e, \rho_0}$ which generate $I$ for $1 < \rho \leq \rho_0$, let $I_\rho = I_0 W^\dag_{d, e, \rho}$ and let $A_\rho = W^\dag_{d, e, \rho}/I_\rho.$ Let $\| \cdot \|_{\rho}$ denote the quotient norm on $A_\rho$ and let $\| \cdot \|$ be that of $T_{d+e}.$ Then for any $\delta \in (0, 1) \cap \mathbb{Q}$ and $\rho > 1$, there is some $1 < \rho' < \rho$ such that
		$$ \| f \|_{\rho'} \leq \| f \|^\delta \, \| f \|^{1 - \delta}_{\rho}$$
		for any $f \in A_{\rho}$.
	\end{lemma}
	\begin{proof}
		This is a generalization of \cite[Lemma 3.44]{lazda2016rigid} to partial dagger algebras.
		We may assume $A = W^\dag_{d, e}$ and $A_\rho = W^\dag_{d, e, \rho}.$ Let $R = \V[[t]]_K \langle x_1, \dots, x_d \rangle$. Let $f = \sum_{\unu} f_\unu \ul{y}^\unu$ with $f_{\unu} \in R \langle t^{-1} \rangle^\dag.$ By decomposing each $f_\unu$ into the sum of $f^-_\unu \in K \langle x_1, \dots, x_d \rangle \langle t^{-1} \rangle^\dag$ and $f^+_{\unu} \in R$ and $f$ into the sum of $f^\pm = \sum_\unu f^\pm_\unu \ul{y}^\unu$, we are reduced to two cases where $f \in K \langle x_1, \dots, x_d \rangle \langle y_1, \dots, y_e, t^{-1} \rangle^\dag$ and $f \in R \langle y_1, \dots, y_e \rangle ^\dag.$ Then both cases are proved in the similar way as in \cite[Lemma 2.4.2]{kedlaya2006finiteness}.
	\end{proof}
	
	\begin{lemma} \label{lem:hor-sec}
		Under Hypothesis \ref{hypo:pda1}, let $R = \R_{A, \eta_0, \eta_1}$ (resp. $R = \R^+_{A, \eta}$) and $\wh{R} = \R_{A, \eta_0, \eta_1}$ (resp. $\wh{R} = \R^+_{\wh{A}, \eta}$) for $\eta_0, \eta_1 \in \Gamma^*$ with $\eta_0 \leq \eta_1$ (resp. $\eta \in \Gamma^*$). Let $M$ be a free log $\nabla$-module over $R$, let $\wh{M} = M \otimes_R \wh{R}$ and suppose that there exists a free log $\nabla$-module $\wh{F}$ over $\wh{A}$ with nilpotent residues equipped with $d$ nilpotent commuting $\wh{A}$-linear endomorphisms $\partial_1, \dots, \partial_d$ such that $\wh{F} \otimes_{\wh{A}} \wh{R} \simeq \wh{M}$. Then the map $f : \wh{M} \to \wh{M}$ of Lemma \ref{lem:unip-f} maps $M$ into $M$; in particular, $f(M) \neq 0$ if $M \neq 0$ and hence $M$ contains nonzero sections horizontal relative to $A$.
	\end{lemma}
	\begin{proof}
		The proof is similar to that of \cite[Lemma 3.45]{lazda2016rigid}. As in Lemma \ref{lem:unip-f}, we treat the case where $R = R = \R_{A, \eta_0, \eta_1}$ and $\wh{R} = \R_{\wh{A}, \eta_0, \eta_1}$. Choose a basis $e_1, \dots, e_n$ of $M$ over $R$ and define $Y_{rij} \in R$ by $\partial_r e_i = \sum_j Y_{rij}e_j.$ For any $x = \sum_i a_i e_i \in M$ with $a_i \in R$, we have
		\begin{align}
			\partial_r x = \sum_i \frac{\partial a_i}{\partial z_r} e_i  \, + \, \sum_{i, j} a_i Y_{rij} e_j = \sum_i \left( \frac{\partial a_i}{\partial z_r} + \sum_j a_j Y_{rji} \right) e_i. \label{expr-flx2}
		\end{align}
		Suppose $a_i \in A_{\rho} \langle z_1, \cdots, z_d \rangle_{[\eta_0 - \epsilon, \eta_1 + \epsilon]}$ and write $f_l(x) = \sum_i a_i^{(l)} v_i$. Let $\| \cdot \|_\rho$ be a Banach norm on $A_\rho$ and let $\| \cdot \|_{\rho, [\alpha, \beta]^d}$ be the canonical Banach norm on $A_\rho \langle z_1, \dots, z_r \rangle_{[\alpha, \beta]^d}$.  Then by (\ref{expr-flx2}) we get
		$$\max_i(\| a_i^{(l)} - a_i^{(l-1)} \|_{\rho, [\eta_0 - \epsilon, \eta_1 + \epsilon]^d}) \leq c^{de} \| l \|^{-2de} \max_i ( \| a_i^{(l-1)} \|_{\rho, [\eta_0-\epsilon, \eta_1+\epsilon]} ),$$
		where $c = \max_{r, i, j}(1, \| Y_{rij} \|)$ Hence we get the upper bound
		$$
		\max_i(\| a_i^{(l)} - a_i^{(l-1)}\|_{\rho, [\eta_0 - \epsilon, \eta_1 + \epsilon]^d}) \leq c^{del} \| l! \|^{-2de} \max_i(\| a_i \|_{\rho, [\eta_0 - \epsilon, \eta_1 + \epsilon]^d}).
		$$
		By Legendre's lemma, we have the estimate $\| l! \|^{-1} \leq p^{\frac{l}{p-1}}.$ Hence there exist constants $C, D > 0$ such that 
		\begin{align}
			\max_i(\| a_i^{(l)} - a_i^{(l-1)} \|_{\rho, [\eta_0 - \epsilon, \eta_1 + \epsilon]^d}) \leq C D^l. \label{expr-flx3}
		\end{align}
		
		On the other hand, by the proof of Lemma \ref{lem:unip-f}, we have
		$$ \max_i(\| a_i^{(l)} - a_i^{(l-1)}\|_{[\eta_0 - \epsilon', \eta_1 + \epsilon']^d}) \to 0 $$
		for any $0 < \epsilon' < \epsilon$; hence we may choose any $0 < \epsilon' < \epsilon$ and some constant $E > 1$ such that
		\begin{align}
			\max_i(\| a_i^{(l)} - a_i^{(l-1)}\|_{[\eta_0 - \epsilon', \eta_1 + \epsilon']^d}) E^l \to 0. \label{expr-flx4}
		\end{align}
		Fix $\delta \in (0, 1) \cap \mathbb{Q}$; by Lemma \ref{lem:hadamar}, we may choose some $1 < \rho' \leq \rho$ such that
		$$ \| a_i^{(l)} - a_i^{(l-1)} \|_{\rho', [\eta_0 - \epsilon', \rho_1 + \epsilon']^d} \leq \| a_i^{(l)} - a_i^{(l-1)} \|_{[\eta_0 - \epsilon', \rho_1 + \epsilon']^d}^\delta \, \| a_i^{(l)} - a_i^{(l-1)} \|_{\rho, [\eta_0 - \epsilon', \rho_1 + \epsilon']^d}^{1-\delta}.$$
		Using (\ref{expr-flx3}), the right hand side is bounded by $C^{1-\delta} (D^{1-\delta} E^{-\delta})^l \cdot \| a_i^{(l)} - a_i^{(l-1)} \|_{[\eta_0 - \epsilon', \rho_1 + \epsilon']^d}^\delta E^{\delta l}.$ By choosing $\delta$ with $D^{1-\delta} E^{-\delta} < 1$, which is possible since $E > 1$, by (\ref{expr-flx4}) we get convergence
		$$ \| a_i^{(l)} - a_i^{(l-1)} \|_{\rho', [\eta_0 - \epsilon', \rho_1 + \epsilon']^d} \to 0. $$
		This proves $f(x) \in M.$ Since $f$ is $\wh{A}$-linear and $M$ is dense in $\wh{M}$, we get $f(M) \neq 0.$
	\end{proof}
	
	Henceforth we pose following hypothesis in addition to Hypothesis \ref{hypo:pda1}.
	
	\begin{hypothesis} \label{hypo:pda2}
		Under Hypothesis \ref{hypo:pda1}, assume the following:
		\begin{enumerate} [(i)]
			\item the spectral norm $|\cdot|_{\sup}$ on $\wh{A}$ is a valuation.
			\item for $Z_i$ the zero locus of $x_i$, $Z_i$ locally admits an open neighborhood of the form $Z_i \times A^1_{\Ed_K}[0, \eta)$ in $X$. 
		\end{enumerate}
		Note that (i) is equivalent to saying both $\wh{A}$ and the reduction $\wh{A}^{\red}$ are integral domains; we let $L$ be a field containing $\Frac \wh{A}$ complete for a valuation which restricts to the spectral norm on $\wh{A}$. (i) is satisfied for example if $\wh{A}$ is an affinoid algebra over $\E_K$ of MW-type.
	\end{hypothesis}
	
	\begin{lemma} \label{lem:des-unip-1}
		Under Hypothesis \ref{hypo:pda2}, let $R = \R_{A, \eta_0, \eta_1}$ (resp. $R = \R^+_{A, \eta}$) and $\wh{R} = \R_{A, \eta_0, \eta_1}$ (resp. $\wh{R} = \R_{\wh{A}, \eta}$) for $\eta_0, \eta_1 \in \Gamma^*$ with $\eta_0 \leq \eta_1$ (resp. $\eta \in \Gamma^*$). Let $M$ be a log $\nabla$-module over $R$ with nilpotent residues such that $\wh{M} = M \otimes_R \wh{R}$ is constant/unipotent. Then $M$ is constant/unipotent.
	\end{lemma}
	\begin{proof}
		By Remark \ref{rem:R-unip}, we can choose a log $\nabla$-module $\wh{F}$ over $\wh{A}$ with nilpotent residues equipped with $d$ nilpotent commuting $\wh{A}$-linear endomorphisms $\partial_1, \dots, \partial_d$ such that $\wh{F} \otimes_{\wh{A}} \wh{R} \simeq \wh{M}$. Since unipotence of $M$ can be checked locally on $A$ by Corollary \ref{cor:unip-loc}, we may assume that $\wh{F}$ is free over $\wh{A}$.
		
		For $X = \Spd A$, let $Y$ be the non-log locus of $X \times A^d[\eta_0, \eta_1]$ (resp. $X \times A^d[0, \eta]$), $N'$ the $\nabla$-submodule of $M' = M |_Y$ spanned by horizontal sections; then by (ii) of Hypothesis \ref{hypo:pda2} and Proposition \ref{prop:ext-log-sub},  $N'$ extends to a log $\nabla$-submodule of $M$ with nilpotent residues. Suppose that $N$ is constant. For $\wh{N} := N \otimes_R \wh{R},$ $\wh{N}$ is a subobject of $\wh{M}$ and therefore $\wh{M}/\wh{N} = (M/N) \wh{\ }$ is unipotent. If $M \neq 0$, then $N' \neq 0$ by Lemma \ref{lem:hor-sec}.  By induction on rank, we deduce that $M/N$ is unipotent and so is $M$.
		
		Thus we are reduced to show that $M$ is constant under the assumption that $M'$ is generated by horizontal sections. Let $\wt{R} = \bigcup_{\epsilon > 0} L \langle z_1, \dots, z_d \rangle_{[\eta_0-\epsilon, \eta_1+\epsilon]^d}$ (resp. $\wt{R} = \bigcup_{\epsilon > 0} L \langle z_1, \dots, z_d \rangle_{[0, \eta+\epsilon]^d}$), $\wt{M} = \wh{M} \otimes_{\wh{R}} \wt{R}$. Then this extra assumption implies that $\wt{M}$, being spanned by horizontal sections, is a constant log $\nabla$-module. Let $\wt{F} = \wh{F} \otimes_{\wh{R}} \wt{R}$ and we denote the image of $\wt{F}$ in $\wt{M}$ also by $\wt{F}$. Then \cite[Lemma 5.2.4]{kedlaya2006finiteness} implies that $\wt{F}$ is the set of horizontal sections of $\wt{M}$.
		
		We imitate the proof of \cite[Proposition 3.4.3]{kedlaya2007semistable} in what follows. Let $H$ be the set of horizontal sections of $M$ and fix a finitely generated $A$-submodule $H_0$ of $H$. We have a commutative diagram
		\[
		\xymatrix{
			H_0 \otimes_A R \ar[r] \ar[d] & M \ar[d] \\
			\wt{F} \otimes_L \wt{R} \ar[r]^(0.58){\sim} & \wt{M}.
		}
		\]
		Let us prove that the left vertical map is injective. Let $I = [\eta_0, \eta_1]$ (resp. $I = [0, \eta]$); then $R \to \wh{A}\langle z_1, \dots, z_d \rangle_I$, $\wh{R} \to \wh{A}\langle z_1, \dots, z_d \rangle_I$ and $\wt{R} \to L \langle z_1, \dots, z_d \rangle_I$ are faithfully flat by Corollary \ref{cor:pda}; hence so is $R \to \wh{R}$ faithfully flat. Let $\wh{H}_0 = H_0 \otimes_A \wh{A}.$ Note that the image of $\wh{H}_0 \to \wt{M}$ is contained in $\wt{F}.$ Since $H_0$ is finitely generated over $A$ and $\wt{F}$ is over $L$, we have inclusions $\wh{H}_0 \otimes_{\wh{A}} \wh{R} \subseteq \wh{H}_0 \otimes_{\wh{A}} \wh{A}\langle z \rangle_I$ and $\wt{F} \otimes_L \wt{R} \subseteq \wt{F} \otimes_L L \langle z \rangle_I.$ Also we have the inclusion $\wh{H}_0 \otimes_{\wh{A}} \wh{A} \langle z \rangle_I \subseteq \wt{F} \otimes L \langle z \rangle_I$ since $\wh{H}_0$ has no $\wh{A}$-torsion. By these inclusions, we get $\wh{H}_0 \otimes_{\wh A} \wh{R} \subseteq \wt{F} \otimes_{L} \wt{R}.$ Also we have the inclusion $H_0 \otimes_A R \subseteq \wh{H}_0 \otimes_{\wh A} \wh{R}$ and thus we get the assertion. Consequently, $H_0 \otimes_A R \to M$ is injective. Since $H_0$ is an arbitrary finitely generated $A$-submodule of $H$ and $M$ is a noetherian $R$-module, it follows that $H$ is finitely generated over $A$. We can again imitate the proof of \cite[Proposition 3.4.3]{kedlaya2007semistable} using previously defined map $f$ to prove that $H \otimes_A R \to M$ is also surjective. Thus $H \otimes_A R \to M$ is an isomorphism and therefore $M$ is constant.
	\end{proof}
	
	\begin{lemma} \label{lem:des-unip-2}
		Let $V$ be an smooth partial dagger space covered by affinoid open subsets satisfying Hypothesis \ref{hypo:pda2}. For any closed aligned subinterval $I$ of $[0, \infty)$ and $\E \in \LNM_{V \times A^d_{\Ed_K}(I)}$, $\E$ is constant/unipotent if and only if $\E | _{\wh{V} \times A^d_{\Ed_K}(I)}$ is constant/unipotent.
	\end{lemma}
	\begin{proof}
		This corollary is just a rephrase of Lemma \ref{lem:des-unip-1}.
	\end{proof}
	
	\begin{theorem}[Descent of unipotence] \label{thm:des-unip}
		Let $V$ be an smooth partial dagger space covered by affinoid open subsets satisfying Hypothesis \ref{hypo:pda2}. For $\eta \in [0, 1) \cap \Gamma$, let $A^d_{\E_K}[\eta, 1) = \{ x \in \mathbb{A}^{d, \ad}_{\E_K} : \eta \leq |z_i|_x, \, |z_i|_{[x]} < 1 \}$. Then
		\begin{enumerate}
			\item for any $\E \in \LNM_{V \times A^d_{\Ed_K}[0, 1)},$ $\E$ is constant/unipotent if and only if $\E|_{\wh{V} \times A^d_{\E_K}[0, 1)}$ is constant/unipotent;
			\item for any $\E \in \LNM_{V \times A^d_{\Ed_K}[\eta_0, 1)}$ with $\eta_0 \in (0, 1) \cap \Gamma^*,$  $\E | _{V \times A^d_{\Ed_K}[\eta, 1)}$ is constant/unipotent for some $\eta \in [\eta_0, 1) \cap \Gamma^*$ if and only if $\E|_{\wh{V} \times A^d_{\E_K}[\eta, 1)}$ is constant/unipotent for some $\eta \in [\eta_0, 1) \cap \Gamma^*.$
		\end{enumerate}
	\end{theorem}
	\begin{proof}
		This follows from Lemma \ref{lem:ui-eq} and Lemma \ref{lem:des-unip-2}.
	\end{proof}
	
	\begin{remark}
		A suitable modification of this descent of unipotence theorem combined with again a suitable modification of Proposition \ref{prop:lnm-colim} gives a correction of an error in the proof \cite[Proposition 3.5.3]{kedlaya2007semistable}. In {\it loc. cit.}, one passes to the tube $]X[^{\mathrm{Tate}}$ in the context of Tate's rigid analytic variety, and it requires some argument to show that unipotence is propagated to a strict neighborhood, as it could be that no strict neighborhoods have integral reductions. In the language of adic spaces, the tube $]X[^{\mathrm{Tate}}$ is interpreted to the ``naive tube'' $]X[^{\ad, \circ}$, which is contained in our tube $]X[^{\ad}$. Our tube is the intersection of all open neighborhoods, which can be regarded as the ``limit'' of strict neighborhoods. By considering unipotence on $]X[^{\ad}$, we can overcome the error. This can be regarded as an evidence that using adic space as rigid analytic spaces has nice aspects compared to others. 
	\end{remark}
	
	\subsection{Convergence and unipotence}
	By posing suitable convergent conditions on log $\nabla$-modules, we can prove the equivalence of extension property and unipotence as in \cite[\S3.6]{kedlaya2007semistable}. Although we may imitate the proof of {\it loc. cit.} to deduce the results in this subsection, for later use it suffices to prove it in the situation of Hypothesis \ref{hypo:pda2} and we can appeal to Theorem \ref{thm:des-unip} to skip the argument.
	
	\begin{definition} \label{def:conv}
		Let $X \to Y$ be an smooth morphism of partial dagger spaces, let $x_1, \dots, x_r$ be a coordinate of $X/Y$ and let $\E \in \LNM_{X/Y}$. Then for $\eta \in \mathbb{R}_{> 0}$, we say that $\E$ is $\eta$-convergent with respect to $x_1, \dots, x_n$ if $X$ can be covered by affinoid open partial dagger spaces $U = \Spd A$ with a fringe inductive system $\{ A_\rho \}_\rho$ such that $\E|_U$ comes from some log $\nabla$-module $\E_\rho$ on an open neighborhood $\Us_\rho = \Spa A_\rho$ and that for any section $s \in \Gamma(\Us_\rho, \E_\rho)$, the multi-indexed sequence $\displaystyle \frac{\ul{\partial}^\unu s}{\unu!}$ is $\eta$-convergent.
		
		For a partial dagger space $V$, $\alpha \in [0, 1) \cap \Gamma$ and $\E \in \LNM_{V \times A^d_{\Ed_K}(\alpha, 1))/V}$ or $\E \in \LNM_{V \times A^d_{\Ed_K}[\alpha, 1))/V}$, we say that $\E$ is convergent if for any $\eta < 1$, there exists $\beta \in (\alpha, 1) \cap \Gamma^*$ such that for all $\gamma \in (b, 1) \cap \Gamma^*$, $\E|_{V \times A^d_{\E_K}[\beta, \gamma]}$ is $\eta$-convergent with respect to the coordinate $z_1, \dots, z_d$ of polyannuli.
	\end{definition}
	\begin{remark} \label{rem:conv}
		In the situation of Definition \ref{def:conv}, if $\E \in \LNM_{X/Y}$ is $\eta$-convergent, then so is $\E|\wh{X}$ in the sense of \cite[Definition 2.4.2]{kedlaya2007semistable}. This follows from the fact that for a partial dagger algebra $A$ and one of its fringe algebra $A_\rho$, there exist a norm $\| \cdot \|_\rho$ of $A_\rho$ and a norm $\| \cdot \|$ of $\widehat{A}$ such that $\| \cdot \|_\rho \geq \| \cdot \|.$ As a consequence, for a partial dagger space $V$, $\alpha \in [0, 1) \cap \Gamma$ and $\E \in \LNM_{V \times A^d_{\Ed_K}(\alpha, 1))/V}$ or $\E \in \LNM_{V \times A^d_{\Ed_K}[\alpha, 1))/V}$, if $\E$ is convergent then so is $\E |_{\wh{X} \times A^d_{\E_K}(\alpha, 1)}$ or $\E|_{\wh{X} \times A^d_{\E_K}[\alpha, 1)}$ convergent in the sense of \cite[Definition 3.6.6]{kedlaya2007semistable}. 
	\end{remark}
	
	\begin{proposition} \label{prop:unip-ext}
		Let $V$ be an smooth partial dagger space covered by affinoid open subsets satisfying Hypothesis \ref{hypo:pda2}, $\alpha \in [0, 1) \cap \Gamma$ and suppose that $\E \in \LNM_{V \times A^d_{\Ed_K}[\alpha, 1)}$ is convergent. Then $\E$ is unipotent if and only if it extends to an object of $\LNM_{V \times A^d_{\Ed_K}[0, 1)}$. Moreover, this extension is unique if it exists and $\E$ is constant if and only if all residue maps for the coordinate $z_i$ of polyannuli are zero.
	\end{proposition}
	\begin{proof}
		If $\E$ is unipotent, then it uniquely extends thanks to Theorem \ref{thm:ui-cat-eq}. Conversely if $\E$ extends, then so does $\wh{\E} = \E|{\wt{V} \times A^d_{\E_K}[a, 1)}$; hence $\wh{\E}$ is unipotent by Remark \ref{rem:conv} and \cite[Proposition 3.6.9]{kedlaya2007semistable}. Applying Theorem \ref{thm:des-unip}, we obtain the statement.
	\end{proof}
	
	\section{Monodromy of isocrystals} \label{sec:monodromy}
	\subsection{Local definition}
	\begin{definition}
		A small $\V[[t]]$-frame (resp. $k((t))$-frame, $\E_K$-frame) is a smooth frame $(X, Y, \Pf, j, i)$ such that $i$ induces an isomorphism $Y \to \Pf \otimes k[[t]]$, $\Pf$ is affine and $Y-X = V(\bar{g}) \cup V(t)$ for some $g \in \Gamma(\Pf, \O).$ 
	\end{definition}
	
	\begin{hypothesis} \label{hypo:small-frame}
		Let $(X, Y, \Pf)$ be a small $\V[[t]]$-frame. Suppose that $Z = Y_\eta - X$ is smooth over $k((t))$ and that there exist $f_1, \dots, f_r \in \Gamma(\Pf, \O)$ which form a coordinate of an open neighborhood of $X$ in $\Pf$ such that $Y-X = V(\bar f_1) \cup V(t)$.
	\end{hypothesis}
	
	\begin{lemma} \label{lem:frame-rel-tube}
		Under Hypothesis \ref{hypo:small-frame}, let $\Qf = V(f_1)$. Then there is an isomorphism $]Z[_\Pf \to ]Z[_\Qf \times A^1_{\Ed_K}[0, 1)$ whose projection to the annulus $A^1$ is induced by $f_1$.
	\end{lemma}
	\begin{proof}
		Let $\overline{Z}$ denote the closure of $Z$ in $Y$ and let $\Pf \to \wh{\mathbb A}^r_{\V[[t]]}$ be the morphism induced by $f_1, \dots, f_r$.
		Since morphism of frames $(Z, \overline{Z}, \Pf \times \Qf) \to (Z, \overline{Z},  \wh{\mathbb A}^r_{\V[[t]]} \times \Qf)$ is proper etale, by strong fibration theorem \ref{thm:strong-fibration} we get $]Z[_{\Pf \times \Qf} \, \simeq \, ]Z[_{\Pf \times \wh{\mathbb A}^r} \, \simeq \, ]Z[_{\Qf} \times A^r_{\Ed_K}[0, 1)$ whose projection to the polyannulus is given by $f_1, \dots, f_r$. Taking the zero locus of $f_2, \dots, f_r$, we get the desired isomorphism $]Z[_\Pf \xrightarrow{\sim} ]Z[_\Qf \times A^1_{\Ed_K}[0, 1).$
	\end{proof}
	
	\begin{definition} \label{def:monodromy-restricted}
		Under Hypothesis \ref{hypo:small-frame}, let $\Qf = V(f_1)$, $Z = Y_\eta - X$ and let $E$ be an $\Ed_K$-valued overconvergent isocrystal on $(X, Y)$. Let $\E$ be the realization of $E$ on the frame $(X, Y, \Pf)$. There is an open neighborhood $\Vs$ of $]X[_\Pf$ in $\Pf^\ad$ over which $\E$ is defined. By Lemma \ref{lem:frame-rel-tube}, $\Vs \cap ]Z[_\Pf$ contains an open subset of the form $]Z[_\Qf \times A^1_{\Ed_K}[\lambda, 1).$ We say that $E$ has constant/unipotent monodromy along $Z$ if $\E$ is constant/unipotent over $]Z[_\Qf \times A^1_{\Ed_K}[\lambda, 1)$ for some $\lambda \in (0, 1) \cap \Gamma^*.$
	\end{definition}
	
	\begin{proposition}
		Suppose that both $(X, Y, \Pf)$ and $(X, Y, \Pf')$ are small frames satisfying Hypothesis \ref{hypo:small-frame}. Let $E$ be an $\Ed_K$-overconvergent isocrystal and $\E$ and $\E'$ be the realization of $E$ on the frames $(X, Y, \Pf)$ and $(X, Y, \Pf')$, respectively. Then $\E$ has unipotent/unipotent monodromy along $Z$ if and only if $\E'$ does.
	\end{proposition}
	\begin{proof}
		We can imitate the proof of \cite[Proposition 4.3.5]{kedlaya2007semistable}.
	\end{proof}
	
	By this proposition, it makes sense to say that an $\Ed_K$-valued overconvergent isocrystal $E$ on a pair $(X, Y)$ with a frame satisfying Hypothesis \ref{hypo:small-frame} has constant/unipotent monodromy along $Z$.
	
	As a consequence of descent of unipotence, we have the following useful proposition.
	
	\begin{proposition} \label{prop:monodromy-completion}
		Under Hypothesis \ref{hypo:small-frame}, let $E$ be an $\Ed_K$-valued overconvergent isocrystal on $(X, Y)$ and let $\widehat{E}$ be the completion of $E$, which is an $\E_K$-valued overconvergent isocrystal on $(X, Y_\eta)$. Then $E$ has constant/unipotent monodromy along $Z$ if and only if $\wh{E}$ has constant/unipotent monodromy along $Z$ in the sense of \cite[Definition 4.3.4]{kedlaya2007semistable}. 
	\end{proposition}
	\begin{proof}
		Let $\E$ be the $\nabla$-module on $]X[_\Pf$ induced by $E$; then $\wh{\E} := \E |_{]X[_{\Pf_\eta}}$ is the $\nabla$-module on $]X[_{\Pf_\eta}$ induced by $\wh{E}$. There is an open neighborhood $\Vs$ of $]X[_\Pf$ over which $\E$ is defined; then $\wh{E}$ is defined over $\Vs_\eta$. $\wh{E}$ has constant/unipotent monodromy along $Z$ means that $\wh{E}$ is constant/unipotent over $]Z[_{\Qf_\eta} \times A^1_{\E_K}[\lambda, 1)$ for some $\lambda$. Since $Z = \Qf_\eta \otimes k((t))$, $]Z[_{\Qf_\eta}$ is a completion of $]Z[_\Qf$. Thus it suffices to apply Theorem \ref{thm:des-unip} after checking that Hypothesis \ref{hypo:pda2} holds for $]Z[_\Qf$. Condition (i) of Hypothesis \ref{hypo:pda2} is satisfied since $Z$ is smooth over $k((t))$ and $Z = \Qf_\eta \otimes k((t))$. Condition (ii) is satisfied due to smoothness of the frmae $(Z, \overline{Z}, \Qf)$ and strong fibration theorem as in the proof of Lemma \ref{lem:frame-rel-tube}.
	\end{proof}
	
	\begin{remark} \label{rem:monodromy}
		As in the classical case, we make remarks similar to \cite[Remark 4.4.1]{kedlaya2007semistable}. Under Hypothesis \ref{hypo:small-frame}, let $E$ be an $\Ed_K$-valued overconvergent isocrystal on $(X, Y)$, $Z = Y_\eta - X$.
		\begin{itemize}
			\item For an open cover $Y = U_1 \cup \cdots \cup U_n$, $E$ has constant/unipotent monodromy along $Z$ if and only if $E|(X \cap U_i, U_i)$ has constant/unipotent monodromy along $Z \cap U_i$ for each $i = 1, \dots, n$.
			\item Let $K'$ be a field containing $K$ which is complete for a discrete valuation extending that of $K$, $k'$ the residue field of $K'$, $X' = X \times_{k((t))} k'((t)), \, Y' = Y \times_{k[[t]]} k'[[t]]$ and let $E'$ be the pullback of $E$ to $(X', Y')$, which is an $\Ed_{K'}$-valued overconvergent isocrystal. Then $E$ has constant/unipotent monodromy along $Z$ if and only if $E'$ has constant/unipotent monodromy along $Z' = Y'_{\eta'} - X'$.
			\item Let $U$ be a dense open subset of $Y$ such that $U \cap  Z$ is dense in $Z$. Then $E$ has constant/unipotent monodromy along $Z$ if and only if $E|_{(U \cap X, U)}$ has constant/unipotent monodromy along $U \cap Z$.
			\item If $E$ extends to an overconvergent isocrystal on $(Y_\eta, Y)$, then $E$ has constant monodromy along $Z$.
		\end{itemize}
	\end{remark}
	
	\subsection{Global definition}
	\begin{definition} \label{def:monodromy-global}
		Let $X$ be a $k((t))$-variety, $Y$ a separated $k[[t]]$-scheme of finite type such that $Y_\eta$ is smooth and let $X \hookrightarrow Y$ be an open immersion. We say that an $\Ed_K$-valued overconvergent isocrystal on $(X, Y)$ has constant/unipotent monodromy along $Z = Y_\eta - X$ if for any finite extension $k'[[t']]$ of $k[[t]]$, any field $K'$ containing $K$ which is complete for a discrete valuation extending that of $K$ and whose residue field is $k'$ and any small $\V'[[t']]$-frame $(U', V', \Pf')$ satisfying Hypothesis \ref{hypo:small-frame} with $V'$ an open subset of $Y' = Y \times_{k[[t]]} k'[[t']]$ and $U' = X \times_{k((t))} k'((t'))$, the ${\E'^{\dag}_{K'}}$-valued overconvergent isocrystal $E|_{(U', V')}$ has constant/unipotent monodromy along $Z' = V'_{\eta'} - U'.$ Here ${\E'^{\dag}_{K'}}$ is the bounded Robba ring over $K'$ with parameter $t'$. Thanks to Remark \ref{rem:monodromy}, this extended definition of monodromy coincide with that of Definition \ref{def:monodromy-restricted}.
	\end{definition}
	
	\begin{remark}
		Since $k((t))$ is imperfect, the extension of field is essentially unavoidable in the definition since $Z$ may not be generically smooth. We only have to consider purely inseparable extensions $k'((t'))/k((t))$. If $k$ is perfect, we do not have to take extensions $K'/K$.
		In \cite[Definition 4.4.2]{kedlaya2007semistable}, discreteness of $K$ or $K'$ is not required at this point; however we require it throughout this paper to ensure $\Ed_K$ to be a field.
	\end{remark}
	
	As a consequence of Proposition \ref{prop:monodromy-completion}, we get the following key theorem of this paper.
	
	\begin{theorem} \label{thm:monodromy-completion}
		Let $X$ be a $k((t))$-variety, $Y$ a separated $k[[t]]$-scheme of finite type such that $Y_\eta$ is smooth, let $X \hookrightarrow Y$ be an open immersion, $Z = Y_\eta - X$ and let $E$ be an $\Ed_K$-valued overconvergent isocrystal on $(X, Y)$. Then $E$ has constant/unipotent monodromy along $Z$ if and only if $\wh{E}$ has constant/unipotent monodromy along $Z$.
	\end{theorem}
	
	\begin{proposition} \label{prop:codim-1-nature}
		Let $X \hookrightarrow Y' \hookrightarrow Y$ be open immersions with $X$ a $k((t))$-variety, $Y'$ and $Y$ separated $k[[t]]$-schemes of finite type and suppose that $Y_\eta - Y'_\eta$ has codimension $\geq 2$ in $Y_\eta$. Let $Z = Y_\eta - X$ and $Z' = Y'_\eta - X$. Then for any $\E^\dag_K$-valued overconvergent isocrystal $E$ on $(X, Y)$, $E$ has constant/unipotent monodromy along $Z$ if and only if $E|_{(X, Y')}$ has constant/unipotent monodromy along $Z'$.
	\end{proposition}
	\begin{proof}
		Similar to \cite[Proposition 4.4.4]{kedlaya2007semistable}. We can also appeal to Theorem \ref{thm:monodromy-completion} as an another proof (which is essentially equivalent to the other one).
	\end{proof}
	
	\section{Monodromy and extension} \label{sec:monodromy-and-extension}
	As in \cite{kedlaya2007semistable}, we will prove the relation between monodromy and extendability of overconvergent isocrystals.
	\subsection{An extension lemma}
	\begin{lemma} \label{lem:ext-lemma}
		Let $V \subseteq U \subseteq X$ be open immersions smooth $k((t))$-varieties and let $X \subseteq Y$ be an open immersion into a separated $k[[t]]$-scheme of finite type such that $V$ is dense in $Y$, $X - V$ is a strict normal crossings divisor in $X$ and that $X - U$ is a single component of $X - V$. Let $\Yf$ be a formal $\V[[t]]$-scheme such that $Y = \Yf \otimes_{\V[[t]]} k[[t]]$. Suppose that there exist $f_1, \dots, f_n, g \in \Gamma(\Yf, \O_\Yf)$ satisfying follwing conditions:
		which forms a coordinate of $\Yf$ in an open neighborhood of $X$ such that $X - V$ is cut out by $f_1 \cdots f_r$ in $X$ and that
		\begin{itemize}
			\item $f_1, \dots, f_r$ forms a coordinate in an open neighborhood of $X$ in $\Yf$;
			\item $f_1 \cdots f_r$ cuts out $X - V$ in $X$ for some $1 \leq r \leq n$;
			\item $f_1$ cuts out $X - U$ in $X$;
			\item $Y - X = V(\bar g) \cup V(t).$
		\end{itemize}Let $\ol{\Xf} = D(g)$ and $\ol{X} = \ol{\Xf} \otimes_{\V[[t]]} k[[t]]$. Then the following results hold.
		\begin{enumerate}[(i)]
			\item Let $\E$ be a coherent $\nabla$-module on $]U[_\Yf$ which represents an $\Ed_K$-valued overconvergent isocrystal on $(U, Y).$ Then $\E$ has constant monodromy along $\ol{X}_\eta - U$ if and only if $\E$ extends to an $\Ed_K$-valued overconvergent isocrystal on $(X, Y)$. 
			\item Let $\E$ be a log $\nabla$-module with nilpotent residues on $]U[_\Yf$ with respect to $f_1, \dots, f_r$ whose restriction to $]V[_\Yf$ represents an $\Ed_K$-valued overconvergent isocrystal on $(V, Y).$ Then $\E$ has unipotent monodromy along $\ol{X}_\eta - U$ if and only if $\E$ extends to a log $\nabla$-module with nilpotent residues on $]X[_\Yf$ with respect to $f_1, \dots, f_r$.
			\item In both cases, the restriction functor is fully faithful.
		\end{enumerate}
	\end{lemma}
	
	\begin{proof}
		Let $\Zf = V(f_1).$ As in Lemma \ref{lem:frame-rel-tube}, there is an isomorphism $\varphi : ]X-U[_\Yf \, \xrightarrow{\sim} \, ]X-U[_\Zf \times A^1_{\Ed_K}[0, 1)$ whose projection to the annulus $A^1$ is given by $f_1$. For $\lambda \in (0, 1) \cap \Gamma^*$, let
		$$ ]U[_{\Yf, \lambda} = \{ x \in \Yf^\ad : |t|_{[x]} \geq 1, \, |g|_{[x]} \geq 1, \, |f_1|_{[x]} > \lambda \}. $$
		Then $\E$ comes from some $\nabla$-module or a log $\nabla$-module with nilpotent residues $\E_\lambda$ on $]U[_{\Yf, \lambda}$. $\varphi$ induces isomorphisms
		\begin{align*}
			]U[_{\Yf, \lambda} \, \cap \, ]X-U[_{\Yf} \ & \simeq  \ ]X-U[_{\Zf} \times A^1_{\Ed_K}(\lambda, 1)\\
			]U[_{\Yf, \lambda} \, \cap \, ]X-U[_{\ol\Xf} \ & \simeq \ ]X-U[_{\Zf \cap \ol\Xf} \times A^1_{\Ed_K}(\lambda, 1).
		\end{align*}
		By definition, $\E$ having constant/unipotent monodromy along $\ol{X}_\eta - U$ is saying that $\E_\lambda$ is constant/unipotent over $]U[_{\Yf, \lambda} \, \cap \, ]X-U[_{\ol\Xf}$ for some $\lambda \in (0, 1) \cap \Gamma^*$.
		Both $]X-U[_{\Zf}$ and $]X-U[_{\Zf \cap \ol{\Xf}}$ admits $]X-U[_{\Zf_\eta \cap \ol{\Xf}_\eta}$ as a completion and therefore we can apply Theorem \ref{thm:des-unip} as in Proposition \ref{prop:monodromy-completion} to deduce that $\E$ has constant/unipotent monodromy along $\ol{X}_\eta - U$ if and only if $\E_\lambda$ is constant/unipotent over $]U[_{\Yf, \lambda} \, \cap \, ]X-U[_{\Yf}$ for some $\lambda \in (0, 1) \cap \Gamma^*$. By Proposition \ref{prop:unip-ext}, $\E$ extends uniquely to a constant/unipotent log $\nabla$-module with nilpotent residues over $]X-U[_{\Zf} \times A^1_{\Ed_K}[0, 1) \simeq ]X - U[_\Yf$. We may glue with $\E|_{]U[_{\Yf, \lambda}}$ to get a log $\nabla$-module with nilpotent residues over $]X[_\Yf$. Proposition \ref{prop:unip-ext} also ensures conversely that if $E$ extends then $\E$ has constant/unipotent monodromy and that extension is unique up to canonical isomorphism.
		
		Finally we have to check overconvergence condition in the case (i); this is proved in the similar way as in \cite[Lemma 5.1.1]{kedlaya2007semistable}.
	\end{proof}
	
	\subsection{Extension of $\Ed_K$-valued overconvergent isocrystals}
	We suppose $k$ is perfect throughout this subsection to invoke Proposition \ref{prop:purely-insep-equiv}.
	Theorems of \cite[\S 5.2]{kedlaya2007semistable} on the extension of overconvergent isocrystals holds in our case. We can almost copy the proofs word by word, so we only give sketches of proofs.
	
	\begin{theorem} \label{thm:ext-oc-isoc}
		Let $U \subseteq X$ be an open immersion of smooth $k((t))$-varieties and $X \subseteq Y$ be an open immersion into a separated $k[[t]]$-scheme of finite type. Let $E$ be an $\Ed_K$-valued overconvergent isocrystal on $(U, Y)$. Then $E$ has constant monodromy along $Y_\eta-X$ if and only if $E$ extends to an $\Ed_K$-valued overconvergent isocrystal on $(X, Y)$. Moreover, the restriction functor $\Isoc(X, Y/\Ed_K) \to \Isoc(U, Y/\Ed_K)$ is fully faithful.
	\end{theorem}
	\begin{proof}
		When $X-U$ is a smooth divisor, we may argue as in \cite[Theorem 5.1.2]{kedlaya2007semistable} to reduce the assertion to collection of Lemma \ref{lem:ext-lemma}.
		
		In general case, we may argue similarly as in \textit{loc. cit.} to prove that after some base change along $k((t)) \to k((t^{1/p^n}))$. By Proposition \ref{prop:purely-insep-equiv}, this is equivalent to saying that $\E$ extends to $(X, Y)$.
	\end{proof}
	
	\begin{corollary}
		In the situation of Theorem \ref{thm:ext-oc-isoc}, let $\wh{E}$ be the completion of $E$, which is an $\E_K$-valued overconvergent isocrystal on $(U, Y_\eta)$. Then $E$ extends to $(X, Y)$ if and only if $\wh{E}$ extends to $(X, Y_\eta).$
	\end{corollary}
	
	As in \cite[\S5.3]{kedlaya2007semistable}, Theorem \ref{thm:ext-oc-isoc} implies following propositions on extensions of overconvergent isocrystals.
	
	\begin{proposition}
		Let $U \hookrightarrow X$ be an open immersion of smooth $k((t))$-varieties, $X\hookrightarrow Y$ an open immersion into a separated $k[[t]]$-scheme of finite type and suppose that $U$ is dense in $Y$. Let $E$ be an $\Ed_K$-valued overconvergent isocrystal on $(X, Y)$, and let $F$ be a subobject of $E|_{(U, Y)}$. Then $F$ is the restriction
		to U of a subobject of $E$.
	\end{proposition}
	
	\begin{proposition}
		Let $U \hookrightarrow X$ be an open immersion of smooth $k((t))$-varieties, $X\hookrightarrow Y$ an open immersion into a separated $k[[t]]$-scheme of finite type and suppose that $U$ is dense in $Y$ and $X-U$ has codimension $\geq 2$ in $X$. Then the restriction functor $\Isoc(X, Y) \to \Isoc(U, Y)$ is an equivalence of categories.
	\end{proposition} 
	
	\begin{proposition}
		Let $U \hookrightarrow X$ be an open immersion of smooth $k((t))$-varieties, $X\hookrightarrow Y$ an open immersion into a separated $k[[t]]$-scheme of finite type and suppose that $U$ is dense in $Y$ and $X-U$ has codimension $\geq 2$ in $X$. Then the restriction functor $\Isoc(X, Y) \to \Isoc(U, Y)$ is an equivalence of categories.
	\end{proposition} 
	
	\begin{proposition}
		Let $f : X \to X'$ be a proper birational morphism of smooth $k((t))$-varieties, $\E$ an $\Ed_K$-valued overconvergent isocrystal on $X$. Then there exists an $\Ed_K$-valued overconvergent isocrystal $\E'$ on $X'$ such that $f^* \E' \simeq \E$.
	\end{proposition}
	\begin{proof}
		Choose an open immersion $X \hookrightarrow Y$ into a proper $k[[t]]$-scheme by Nagata's compactification theorem; then similar argument as in \cite[Proposition 5.3.6]{kedlaya2007semistable} gives the proposition.
	\end{proof}
	
	\begin{proposition}
		Let $U \hookrightarrow X \hookrightarrow Y$ be open immersions with $X$ smooth over $k((t))$, $Y$ a $k[[t]]$-variety and $U$ dense in $Y$. Then the canonical functor
		$$ \Isoc(X, Y/\Ed_K) \to \Isoc(X, X/\Ed_K) \times_{\Isoc(U, X/\Ed_K)} \Isoc(U, Y/\Ed_K)$$
		is an equivalence of categories.
	\end{proposition}
	\begin{proof}
		This is a rephrase of \cite[Propositoin 5.3.7]{kedlaya2007semistable}. For fully faithfulness, use \ref{thm:ext-oc-isoc}.
	\end{proof}
	
	\section{Logarithmic extensions} \label{sec:log-ext}
	For an $\Ed_K$-valued overconvergent isocrystal on a $\V[[t]]$-pair $(X, Y)$, we have only defined monodromy along a closed subset of $Y_\eta$. Hence logarithmic extension should terminate at the pair $(Y_\eta, Y)$, which means that we have to develop the theory of overconvergent log isocrystals, differing from the classical case.
	
	\subsection{De Jong's alteration theorem over $k[[t]]$}
	Since in our setting we take a compactification of a $k((t))$-variety into a $k[[t]]$-scheme, we have to deal with alterations of $k[[t]]$-schemes. For reader's sake, we refer to de Jong's theorem in this subsection.
	
	\begin{definition}
		Let $Y$ be a $k[[t]]$-variety (i.e., integral, separated $k[[t]]$-scheme of finite type) and $Z \subseteq Y$ a reduced closed subscheme. We say that $(Y, Z)$ is a strictly semistable pair if:
		\begin{enumerate}[(i)]
			\item $Y$ is stricly semistable over $k[[t]]$;
			\item $Z$ is a  strict normal crossings divisor in $Y$;
			\item Let $Z_1, \dots, Z_n$ be irreducible components of $Z$ each of which dominates $k[[t]]$; then for any $1 \leq i_1 < \cdots < i_p \leq n$, $Z_{i_1} \cap \cdots \cap Z_{i_p}$ is a disjoint union of strictly semistable $k[[t]]$-varieties.
		\end{enumerate}
		In this situation, we denote $Z_1 \cup \cdots \cup Z_n$ by $Z_f$. The components of $Z$ contained in $Z_f$ are called vertical divisors and the those contained in the closed fiber are called horizontal divisors.
	\end{definition}
	
	We will only consider strictly semistable pairs $(Y, Z)$ such that $Z$ contains $Y_s$.
	
	\begin{remark}
		For a strictly semistable pair $(Y, Z)$ and $x \in Y_s$, let $Z_1, \dots, Z_n$ be irreducible components of $Z_f$ containing $x$ and let $Y_1, \dots, Y_m$ be irreducible components of $Y_s$ containing $x$. Then there is an open neighborhood $U$ of $x$ in $Y$ such that there is a smooth $k[[t]]$-morphism $U \to k[[t]][z_1, \dots, z_n, y_1, \dots, y_n]/(t - y_1 \cdots y_n)$ whose pullback of $V(z_i)$ is $U \cap Z_i$ and whose pullback of $V(y_j)$ is $U \cap Y_j$. This is pointed out in \cite[Remark 1.2.6]{caro2021arithmetic}.
	\end{remark}
	
	\begin{theorem}[{\cite[Theorem 6.5]{jong1996alteration}}] \label{thm:alteration}
		Let $Y$ be a $k[[t]]$-variety and $Z \subseteq Y$ be a proper closed subset containing $Y_s$. Then there is a finite extension $k[[t]] \to k'[[t']]$ of complete discrete valuation rings, a $k'[[t']]$-variety $Y'$, an alteration (i.e., proper, dominant and generically finite) $f : Y' \to Y$ over $k[[t]]$ and an open immersion $Y' \hookrightarrow \ol{Y'}$ of $k'[[t']]$-shcemes with following properties:
		\begin{enumerate}[(i)]
			\item $\ol{Y'}$ is projective over $k'[[t']]$ with geometrically irreducible generic fiber;
			\item $(\ol{Y'}, (f^{-1}(Z)_{\red}) \cup (\ol{Y'}-Y'))$ is a strictly semistable pair.
		\end{enumerate}
	\end{theorem}
	
	\subsection{Overconvergent log $\nabla$-module}
	Before starting to discuss log isocrystals, we first define overconvergent log $\nabla$-modules and point out the equivalence between extendability to overconvergent log $\nabla$-modules and unipotence of monodromy.
	
	\begin{hypothesis} \label{hypo:ss-pair}
		Let $\Pf = \Spf \V[[t]] \langle t_1, \dots, t_{n} \rangle / (t - t_1 \cdots t_r)$, $\Df = V(t_1 \cdots t_m) \subseteq \Pf$ for some $r \leq m \leq n$, let $f : \Yf \to \Pf$ be an etale morphism of affine formal $\V[[t]]$-schemes and $\Zf$ the pullback of $\Df$ along $f$. By the corresponding block characters we denote the induced $k[[t]]$-schemes. Define $X = Y-Z.$
	\end{hypothesis}
	
	\begin{definition}
		Under Hypothesis \ref{hypo:ss-pair}, let $\E$ be a log $\nabla$-module on $]Y_\eta[_{\Yf}$ with respect to the log structure defined by $t_2, \dots, t_m$. We say that $\E$ is overconvergent if its restriction to $]X[_{\Yf}$ is overconvergent.
	\end{definition}
	
	\begin{lemma} \label{lem:ext-log-nabla}
		Under Hypothesis \ref{hypo:ss-pair}, let $\E$ be an overconvergent $\nabla$-module over $]X[_{\eta}$. Then $\E$ has unipotent monodromy along $Z_\eta = Y_\eta - X$ if and only if it extends to an overconvergent log $\nabla$-module on $]Y_\eta[_{\Yf}$ with nilpotent residues. Moreover, the restriction functor from overconvergent log $\nabla$-modules with nilpotent residues on $]Y_\eta[_{\Yf}$ to overconvergent $\nabla$-modules on $]X[_{\eta}$ is fully faithful.
	\end{lemma}
	\begin{proof}
		By taking an affine open covering, the statement is reduced to a collection of cases where Theorem \ref{lem:ext-lemma}.
	\end{proof}
	
	\begin{lemma} \label{lem:conv-mul-seq}
		Under Hypothesis \ref{hypo:ss-pair}, let $\E$ be an overconvergent log $\nabla$-module over $]Y_\eta[_{\Yf}$ relative to $\Ed_K$. For $\eta \in (0, 1) \cap \Gamma^*$, let
		$$\Ys_{\lambda} = \{ x \in \Yf^\ad : |t|_x \geq \lambda \}.$$
		Choose a log $\nabla$-module $\E_{\lambda_0}$ on $\Ys_{\lambda_0}$ which gives rise to $\E$. Let $v \in \Gamma(\Ys_{\lambda}, \E_{\lambda_0})$ for some $\lambda \in [\lambda, 1) \cap \Gamma^*$. Then for any $\eta \in (0, 1)$, there exists $\lambda \in [\lambda', 1) \cap \Gamma^*$ such that the multi-sequence in $(\nu_2, \dots, \nu_n)$
		$$ \frac{1}{\nu_2! \cdots \nu_n!}
		\left( \prod_{i=2}^{m} \prod_{l=0}^{\nu_j-1} \left( \frac{\partial}{\partial t_j} - l \right) \right)
		\left( \prod_{j=m+1}^{n} \frac{\partial^{\nu_j}}{\partial t_j^{\nu_j}} \right)v $$
		is $\eta$-convergent on $\Ys_{\lambda'}.$
	\end{lemma}
	\begin{proof}
		As in \cite[Lemma 6.3.4]{kedlaya2007semistable}, the multi-sequence is equal to
		$$ \frac{t_2^{\nu_2} \cdots t_m^{\nu_m}}{\nu_2! \cdots \nu_m!}
		\left( \prod_{i=2}^{m} \prod_{l=0}^{\nu_j-1} \left( \frac{\partial}{\partial t_j} - l \right) \right)
		\frac{1}{\nu_{m+1}! \cdots \nu_{n}!}
		\left( \prod_{j=m+1}^{n} \frac{\partial^{\nu_j}}{\partial t_j^{\nu_j}} \right)v $$
		Since $\E$ is convergent, given any $\eta \in (0, 1)$, this multi-sequence is $\eta$-null on some open neighborhood of $]X[_\Yf$ by Proposition \ref{prop:mic-oc}. Since $\| y_j \| \leq 1$, we get the statement.
	\end{proof}
	
	\subsection{Log frames and log tubular neighborhoods}
	From this subsection we begin to define overconvergent log isocrystals, which is a generalization of log convergent isocrystals and overconvergent isocrystals.
	
	Until the end of this section, we let $\Bf = \Spf \V[[t]]$ and let $\M_\Bf$ be the fine log structure on $\Bf$ induced by $\mathbb{N} \to \V[[t]], \ 1 \mapsto t.$ Let $(B, \M_B) = (\Bf, \M_\Bf) \times_{\V} k.$ By $\LFS/\Bf$ we denote the category of fine log formal schemes (not necessarily $p$-adic) topologically of finite type over $\Bf$ and let $\pLFS/\Bf$ be the full subcategory of fine log $p$-adic formal schemes. Let $\LS/B$ be the category of fine log schemes of finite type over $B$.
	
	We will only consider fine log structures on schemes or formal schemes and fiber products are taken in the category of fine log schemes or fine log formal schemes.
	
	In order to simultaneously define log isocrystals with derivations by $t$, we make the following definition as in the case of $K$-valued overconvergent isocrystals.
	
	\begin{definition}
		A fine log scheme $X$ over $k$ (resp. a fine log formal scheme $\Xf$ over $\V$) is said to be of pseudo finite type if there is a morphism $X \to (B, M_B) \times \cdots \times (B, M_B)$ (resp. $\Xf \to (\Bf, M_\Bf) \times \cdots \times (\Bf, M_\Bf)$ of finite type.
	\end{definition}
	
	By $\LFS/\V$ we denote the category of fine log formal schemes (not necessarily $p$-adic) of pseudo finite type over $\V$ and let $\pLFS/\V$ be the full subcategory of fine log $p$-adic formal schemes. Let $\LS/k$ be th category of fine log formal schemes of pseudo finite type over $k$.
	
	In what follows we let $(\LFS, \pLFS, \LS) = (\LFS/\Bf, \pLFS/\Bf, \LS/B)$ or $(\LFS/\V, \pLFS/\V, \LS/k)$.
	
	\begin{definition}[cf. {\cite[Definition 2.1.9]{shiho2002crystallineII}}, {\cite[Definition 2.1]{shiho2007relativeI}}]
		Let $(\Xf, M_\Xf) \hookrightarrow (\Yf, M_\Yf)$ be a strict open immersion in $\pLFS$ and let $(\Yf, M_\Yf) \to (\Tf, M_{\Tf})$ be a morphism in $\pLFS$.
		The category of log frames over $(\Xf, M_\Xf) \hookrightarrow (\Yf, M_\Yf)$ relative to $(\Tf, M_\Tf)$ (or simply over $(\Xf, \Yf)^{\log}/\Tf$ or over $(\Xf^\sharp, \Yf^\sharp) / \Tf^\sharp$)is defined as follows: an object is a sextuple $((X', M_{X'}), \allowbreak (Y', M_{Y'}), \allowbreak (\Yf', M_{\Yf'}), \allowbreak j, i, z)$ consisting of:
		\begin{itemize}
			\item log schemes $(X', M_{X'}), (Y', M_{Y'})$,
			\item a formal log scheme $(\Yf', M_{\Yf'})$ over $(\Tf, M_\Tf)$,
			\item a strict open immersion $j: (X', M_{X'}) \hookrightarrow (Y', M_{Y'})$,
			\item a closed immersion $i: (Y', M_{Y'}) \hookrightarrow (\Yf', M_{\Yf'})$,
			\item a morhism $z: (Y', M_{Y'}) \to (\Yf, M_{\Yf})$ with $z(j(X')) \subseteq \Xf$
		\end{itemize}
		such that the diagram
		\[
		\xymatrix{
			(X', M_{X'}) \ar@{^{(}->}[r]^j \ar[d] & (Y', M_{Y'}) \ar@{^{(}->}[r]^i \ar[d]^z & (\Yf', M_{\Yf'}) \ar[d] \\
			(\Xf, M_\Xf) \ar@{^{(}->}[r] & (\Yf, M_{\Yf}) \ar[r] & (\Tf, M_\Tf) 
		}
		\]
		is commutative.
		A morphism of sextuples is defined in the usual way. For a log frame $((X', M_{X'}), \allowbreak (Y', M_{Y'}), \allowbreak (\Yf', M_{\Yf'}), \allowbreak j, i, z)$, we will often use the notation $((X', M_{X'}), (Y', M_{Y'}), (\Yf', M_{\Yf'}))$ or even $(X', Y', \Yf')^{\log}$ if the log structure is obvious from the context.
	\end{definition}
	
	\begin{definition}
		Let $\Xf$ be an affine formal $\V[[t]]$-scheme with an ideal of definition $\mathcal{I}$. Let $f_1, \dots, f_r \in \Gamma(\Xf, \mathcal{I})$ be its generators. Following \cite[0.2.6]{berthelot1996cohomologie}, we define the associated rigid analytic space
		$$\Xf_K = \bigcup_{\eta \in (0, 1) \cap \Gamma^*} \{ x \in \Xf^\ad : |f_1|_x, \dots, |f_r|_x \leq \eta \} = \{ x \in \Xf^\ad : |f_1|_{[x]}, \dots, |f_r|_{[x]} < 1\}.$$
		This definition is independent of the choice of $\mathcal{I}$ and $f_1, \dots, f_r$ and is functorial for $\Xf$. Thus we can glue them to define the associated rigid analytic space $\Xf_K$ for any formal $\V[[t]]$-scheme $\Xf$. We also have the specialization map $\sp : \Xf_K \to \Xf.$
	\end{definition}
	
	Note that if $\Xf$ is a $p$-adic formal $\V[[t]]$-scheme, then $\Xf_K = \Xf^\ad.$ Furthermore, if $Z$ is a closed subscheme of $\Xf$ cut out by $f_1, \dots, f_r \in \Gamma(\Xf, \O)$, then the rigid analytic space $\wh{\Xf}_K$ associated to the completion of $\Xf$ along $Z$ coincides with $]Z[_{\Xf}.$
	
	\begin{definition}[{cf. \cite[Definition 2.19]{shiho2007relativeI}}]
		Let $(Z, M_Z) \in \LS$ and $(Z, M_Z) \hookrightarrow (\Zf, M_{\Zf})$ a closed immersion in $\LFS$. The (log) tubular neighborhood $]Z[^{\log}_\Zf$ is defined to be the rigid analytic space $(\Zf^\ex)_K$. We define the specialization map $\sp: ]Z[^{\log}_\Zf \to \wh{\Zf}$ (completion of $Z$ along $\Zf$) to be the composition
		$$]Z[^{\log}_\Zf = (\Zf^\ex)_K \xrightarrow{\sp} \Zf^\ex \to \wh{\Zf}$$
		where the last map is induced by $\Zf^\ex \to \Zf$. Let $[\sp] : ]Z[^{\log}_\Zf \to \wh{\Zf}$ be the composition
		$$]Z[^{\log}_\Zf = (\Zf^\ex)_K \xrightarrow{[\cdot]} (\Zf^\ex)_K \xrightarrow{\sp} \wh{\Zf}.$$
		For a locally closed subset $U$ of $Z$, regard it as an open subset of $\wh{\Zf}$ via the homeomorphism $Z \xrightarrow{\sim} \wh{\Zf}$ and define $]U[^{\log}_\Zf = [\sp]^{-1}(U)$, which forms a germ together with $]Z[^{\log}_\Zf$.
	\end{definition}
	
	\begin{theorem}[log weak fibration theorem] \label{lem:tub-nei-form}
		Let $(Z, M_Z) \in \LS$, let $U \subseteq Z$ be an open subset, let $i : (Z, M_Z) \hookrightarrow (\Zf, M_{\Zf})$ be a closed immersion in $\LFS$ and let $(Z, M_Z) \xrightarrow{i'} (\Zf', M_{\Zf'}) \xrightarrow{f} (\Zf, M_{\Zf})$ be a factorization of $i$ into an exact closed immersion $i'$ and a morphism $f$ formally log \'{e}tale in an open neighborhood of $U$, defined by infinitesimal lifting property in $\LFS$. Then $(\wh{\Zf'})_K$ is canonically isomorphic to $]Z[^{\log}_\Zf = (\Zf^\ex)_K$ and $]U[_{\wh{\Zf'}}$ is canonically isomorphic to $]U[^{\log}_{\Zf'}$ as germs.
	\end{theorem}
	\begin{proof}
		By construction of exactification (see \cite[Proposition-Definition 2.10]{shiho2007relativeI}), we have a factorization $(Z, M_Z) \xrightarrow{i^\ex} (\Zf^\ex, M_{\Zf^\ex}) \xrightarrow{g} (\Zf, M_\Zf)$ with $i^\ex$ exact, $g$ formally log \'{e}tale and $\Zf^\ex \simeq \wh{\Zf^\ex}$. Hence it suffices to check that $\wh{\Zf'}_K$ and $]U[_{\wh{\Zf'}}$ are independent of the choice of factorizations.
		We follow the Let $(Z,M_Z) \xrightarrow{i''} (\Zf'',M_{Zf''}) \xrightarrow{f''} (\Zf,M_\Zf)$ be another such factorization.
		
		For $]Z[^{\log}_\Zf$, we proceed as in \cite[Lemma 2.2.2]{shiho2002crystallineII}. Put $(\Zf^0,L)=(\Zf',M_{\Zf'}) \times_{(\Zf, M_\Zf)}(\Zf'',M_{\Zf''})$, the fibered product of fine formal log schemes and let $(Z,M_Z) \xrightarrow{k} (\Zf^0,L) \xrightarrow{g} (\Zf,M_\Zf)$ for the induced morphism. $k$ is a locally closed immersion and $h$ is formally log etale in an open neighborhood of $\Zf^0$. By restricting to an open subset of $\Zf^0$, we may assume that $k$ is a closed immersion. Let us check that $k$ is exact. by definition $k^*L$ is the pushout of the diagram:
		\[
		\xymatrix{
			i^*L \ar[r] \ar[d] & i''^*M_{P''} \\
			M_{P'}. &
		}
		\]
		Since $i'$ and $i''$ are strict, we have a commutative diagram
		\[
		\xymatrix{
			i^*L \ar[r]\ar[d] & i''^*M_{\Zf''}\ar[d]^{\wr} \\
			i'^*M_{\Zf'} \ar[r]^{\sim} & M_{Z}
		}
		\]
		with left vertical map and bottom horizontal map isomorphisms. Hence $M_Z$ is the pushout and $k^*L \simeq M_Z$.
		We have the following commutative diagram
		\[
		\xymatrix{
			& & (\Zf',M_{\Zf'}) \ar[rd]^{f'} \\
			(U,M_U) \ar@{^{(}->}[r] & (Z, M_Z) \ar@{^{(}->}[r]^k \ar@{^{(}->}[ur]^{i'} \ar@{^{(}->}[dr]_{i''} & (\Zf^0,L) \ar[u]^{p_1} \ar[d]_{p_2} \ar[r]^g & (\Zf,M_\Zf) \\
			& & (\Zf'',M_{\Zf''}). \ar[ur]^{f''} \\
		}
		\]
		It suffices to prove that $p_1$ and $p_2$ induces isomorphisms $(\wh{\Zf^0})_K \simeq (\wh{\Zf'})_K$, $]U[_{\wh{\Zf^0}} \simeq ]U[_{\wh{\Zf'}}$ and $(\wh{\Zf^0})_K \simeq (\wh{\Zf''})_K$, $]U[_{\wh{\Zf^0}} \simeq ]U[_{\wh{\Zf''}}$. Therefore we may assume in the first place that we have a commutative diagram
		\[
		\xymatrix{
			& & (\Zf',M_{\Zf'}) \ar[rd]^{f'} \\
			(U,M_U) \ar@{^{(}->}[r] & (Z, M_Z) \ar@{^{(}->}[r]^{i''} \ar@{^{(}->}[ur]^{i'} & (\Zf'', M_{\Zf''}) \ar[u]^{p} \ar[r]^{f''} & (\Zf,M_\Zf). \\
		}
		\]
		Let $\Uf' \subseteq \Zf'$ and $\Uf'' \subseteq \Zf''$ be open subsets over which $f'$ and $f''$ are formally log etale, respectively. $\Vf'' = \{ x \in \Zf'' : (p^*M_{\Zf'})_{\bar{x}} \simeq (M_{\Zf''})_{\bar{x}} \text{is an isomorphism} \}$ is an open subset of $\Zf''$ containing $U$. By replacing $\Uf''$ by $\Uf'' \cap \Vf''$ we may assume that $p$ is formally log etale on $\Uf''$. Let $\wh{\Zf''}$ and $\wh{\Zf'}$ be the formal completions along $Z$ and let $\wt{\Zf''}$ be the formal completion of $\Zf''$ along $Z' = Z \times_{\Zf'} \Zf''$. Put $U' := U \times_{\Zf'} \Zf''.$ The projection $U' \to U$ is formally log etale and the canonical morphism $U \to U'$ is a section; hence $U$ is a summand of $U'$. If we denote by $\Wf', \, \Wf''$ and $\wt{\Wf''}$ the open subsets of $\wh{\Zf'}, \, \wh{\Zf''}$ and $\wt{\Zf''}$ whose underlying set are $U, \, U$ and $U'$, then $\Wf''$ is a summand of $\wt{\Wf''}$ and $\wt{\Wf''} \to \Wf'$ is formally log etale. Therefore $\Wf'' \to \Wf'$ is formally log etale and gives an identity map $X \to X$ of their schemes of definition. It follows that $\Wf'' \to \Wf'$ is an isomorphism. Hence $p$ induces an isomorphism (of rigid analytic spaces) of naive tubes $]U[^{\circ}_{\wh{\Zf''}} \xrightarrow{\sim} ]U[^{\circ}_{\wh{\Zf'}}$ , where $]U[^{\circ}_{\wh{\Zf'}} = \sp^{-1}_{\wh{\Zf'}}(U) = \Wf'_K$ and similarly for double primes.
		
		Recall that from \cite[{\S 4}]{huber1994generalization}, $(\wh{\Zf'})_K$ is the inverse limit of all admissible blowups as a topological space and similarly for double primes; this implies that $p$ induces a homeomorphism $(\wh{\Zf''})_K \to (\wh{\Zf'})_K.$ Since $]U[_{\wh{\Zf'}}$ is the closure of naive tube $]U[^{\circ}_{\wh{\Zf'}}$ in $(\wh{\Zf'})_K$ and similarly for double primes, it follows that $p$ gives a homeomorphism $]U[_{\wh{\Zf''}} \xrightarrow{\sim} ]U[_{\wh{\Zf'}}$. Now we are in the situation of the proof of  \cite[Theorem 5.21]{stum2017rigidI}. For reader's sake, we execute the argument of {\it loc. cit.} to finish the proof. For any point $v \in ]U[_{\wh{\Zf''}}$ since $]U[^{\circ}_{\wh{\Zf''}}$ is dense open, $v$ generalizes to some point $w \in ]U[^{\circ}_{\wh{\Zf''}}.$ The local homomorphism $\O_{\wh{\Zf''}, v} \to \O_{\wh{\Zf''}, w}$ induces the isomorphism $\mathcal{H}(v) \xrightarrow{\sim} \mathcal{H}(w)$ of completed residue fields; similarly we have $\mathcal{H}(p(v)) \xrightarrow{\sim} \mathcal{H}(p(w))$. Since $w$ and $p(w)$ are in the naive tubes which are isomorphic as rigid analytic spaces, we conclude that $\mathcal{H}(p(v)) \xrightarrow{\sim} \mathcal{H}(v)$ is an isomorphism. By \cite[Proposition 5.19]{stum2017rigidI} there is an open neighborhood $V$ of $]U[_{\wh{\Zf''}}$ such that $V \to (\wh{\Zf'})_K$ is etale and by \cite[Proposition 2.3.7]{huber2013etale}, etale topoi of $(]X[_{\wh{\Zf''}}, V)$ and $(]X[_{\wh{\Zf'}}, (\wh{\Zf'})_K)$ are equivalent. This implies that germs $]X[_{\wh{\Zf''}}$ and $]X[_{\wh{\Zf'}}$ are isomorphic.
	\end{proof}
	
	\subsection{Overconvergent log isocrystals}
	\begin{definition}
		Let $(\Xf, M_\Xf) \hookrightarrow (\Yf, M_\Yf)$ be a strict open immersion in $\pLFS$ and let $(\Yf, M_\Yf) \to (\Tf, M_{\Tf})$ be a morphism in $\pLFS$. An overconvergent log isocrystal $\E$ over $((\Xf, M_\Xf), (\Yf, M_\Yf))/(\Tf, M_\Tf)$ (or simply over $(\Xf, \Yf)^{\log}/\Tf$ or over $(\Xf^\sharp, \Yf^\sharp) / \Tf^\sharp$) consists of:
		\begin{enumerate}
			\item a coherent $\O_{]X'[^{\log}_{\Yf'}}$-module $\E_{\Yf'}$ for each log frame $(X', Y', \Yf')^{\log}$ over $(\Xf, \Yf)^{\log}/\Tf$;
			\item for each morphism $f: \Yf'_1 \to \Yf'_2$ of log frames over $(\Xf,\Yf)^{\log}/\Tf$, an isomorphism $\phi_f : f^\dag \mathcal{E}_{\Yf'_2} \to \mathcal{E}_{\Yf'_1}$, where $f^\dag : \Coh(\O_{]X'_2[^{\log}_{\Yf'_2}}) \to \Coh(\O_{]X'_1[^{\log}_{\Yf'_1}})$ is defined in the usual way;
		\end{enumerate}
		with $\phi$ satisfying obvious cocycle condition. The category of overconvergent log isocrystals over $(\Xf,\Yf)^{\log}/\Tf$ is denoted by $I^{\log}_{\oc}(\Xf,\Yf/\Tf)$ or by $I_{\oc}(\Xf^\sharp, \Yf^\sharp / \Tf^\sharp)$.
	\end{definition}
	
	\begin{remark}
		It is obvious from definition that the category of overconvergent log isocrystals over $(\Yf, \Yf)^{\log} / \Tf$ is equivalent to the category of convergent log isocrystals over $(\Yf, M_\Yf) / (\Tf, M_\Tf)$ in the sense of \cite[Definition 2.5]{shiho2007relativeI}.
	\end{remark}
	
	The category of overconvergent log isocrystals is local for the frame as in the case of overconvergent isocrystals:
	
	\begin{proposition}
		Let $\{ \Yf_i \}_i$ be a Zariski open covering of $\Yf$ and let $\Xf_i = \Yf_i \times_{\Yf} \Xf.$ For $i_1, \dots,  i_p$, put $\Yf_{i_1 \dots i_p} = \Yf_{i_1} \times_{\Yf} \cdots \times_{\Yf} \Yf_{i_p}$ and $\Xf_{i_1 \dots i_p} = \Xf_{i_1} \times_{\Xf} \cdots \times_{\Xf} \Xf_{i_p}$ Then there is a canonical equivalence of categories
		\[
		I^{\log}_\oc(\Xf, \Yf / \Tf) \simeq 2\mathchar`-\lim \left[ \prod_i I^{\log}_\oc(\Xf_i, \Yf_i / \Tf) \substack{\longrightarrow \\[-2pt] \longrightarrow} \prod_{i, j} I^{\log}_\oc(\Xf_{ij}, \Yf_{ij} / \Tf) \substack{\longrightarrow \\[-2pt] \longrightarrow \\[-2pt] \longrightarrow} \prod_{i, j, k} I^{\log}_\oc(\Xf_{ijk}, \Yf_{ijk} / \Tf) \right].
		\]
	\end{proposition}
	
	\subsection{Log connections over log tubular neighborhoods}
	Using tubular neighborhoods, we can give interpretations of overconvergent log isocrystals in terms of modules with log stratifications or integrable log connections as in convergent case.
	
	Let us consider the following situation. Consider a commutative diagram in $\LFS$:
	\[
	\xymatrix{
		(X, M_X) \ar@{^{(}->}[r] \ar@{=}[d] & (Y, M_Y) \ar@{=}[d] \ar@{^{(}->}[r]^i & (\Pf, M_\Pf) \ar[d]^{g}\\
		(X, M_X) \ar@{^{(}->}[r]  & (Y, M_Y) \ar[r] & (\Tf, M_\Tf)
	}
	\]
	where $(X, M_X)$, $(Y, M_Y)$ are in $\LS$, $(\Pf, M_\Pf)$, $(\Tf, M_\Tf)$ are in $\pLFS$, $j$ an open immersion, $i$ a closed immersion and $g$ formally log smooth morphism.
	Let $(\Pf(n), M_{\Pf(n)})$ be the $(n+1)$-th fold fibered product of $(\Pf, M_\Pf)$ over $(\Tf, M_\Tf)$.
	Then we have a log frame $(X, M_X) \hookrightarrow (Y, M_Y) \hookrightarrow (\Pf(n), M_{\Pf(n)})$ of $(X, M_X) \hookrightarrow (Y, M_Y) / (\Tf, M_\Tf)$. Then projection maps and diagonal map induce canonical morphisms
	\begin{itemize}
		\item $p_{i}: ]X[^{\log}_{\Pf(1)} \to ]X[^{\log}_P$ ($i=0,1$),
		\item $p_{ij}: ]X[^{\log}_{\Pf(2)} \to ]X[^{\log}_{P(1)}$ ($0 \leq i < j \leq 2$),
		\item $\Delta: ]X[^{\log}_\Pf \to ]X[^{\log}_{\Pf(1)}$.
	\end{itemize}
	Let $\Strat(X, Y, \Pf/\Tf)^{\log}$ be the category of coherent $\mathcal{O}_{]X[^{\log}_{\Pf(1)}}$-modules $\mathcal{E}$ equipped with an $\O$-linear isomorphism $\epsilon: p_i^*\mathcal{E} \rightarrow p_0^*\mathcal{E}$ such that $\Delta^*\epsilon = \id$ and $p_{01}^*\epsilon = (p_{02}^*\epsilon)(p_{12}^*\epsilon)$.
	
	Then we have the following expected proposition as in \cite[Proposition 2.2.7]{shiho2002crystallineII} under the situation where $(\LFS, \pLFS, \LS) = (\LFS/\Bf, \pLFS/\Bf, \LS/B)$.  When $(\LFS, \pLFS, \LS) = (\LFS/\V, \pLFS/\V, \LS/k)$ and $\Tf = \Spf \V$, one can handle the fibered products of $\V[[t]]$ over $\V$ similarly as in the proof of \cite[Proposition 5.12]{lazda2016rigid} to prove the following:
	
	\begin{proposition}
		We have a canonical equivalence of categories
		$$I^{\log}_{\oc}(X \hookrightarrow Y / \Tf) \simeq \Strat(X, Y, \Pf/\Tf)^{\log}.$$
	\end{proposition}
	
	Let $E$ be an overconvergent log isocrystal on $(X, Y) / \Tf$ and let $(\mathcal{E}, \epsilon) \in \Strat$ be the associated object.
	Let $(P', M_{P'})$ be the first log infinitesimal thickening of $(P, M_P)$ in $(P(1), M_{P(1)})$.
	$\mathcal{E}$ induces an $\O_{]X[^{\log}_{\Pf'}}$-linear isomorphism $\O_{]X[^{\log}_{\Pf'}} \otimes_{\O_{]X[^{\log}_{\Pf}}} \E \rightarrow \E \otimes_{\O_{]X[^{\log}_{\Pf}}} \O_{]X[^{\log}_{\Pf'}}$). Define the log connection
	$\nabla: \E \to \E \otimes_{\O_{]X[^{\log}_\Pf}} \omega_{]X[^{\log}_\Pf / \Tf_K}$
	by
	$\nabla(e) = \epsilon \cdot (1 \otimes e) - e \otimes 1$. Here $\omega_{]X[^{\log}_\Pf / \Tf_K}  = \omega_{\Pf/\Tf}^{\text{log}} |_{]X[^{\log}_{\Pf}}$.
	Similarly to \cite[Lemma 2.2.8]{shiho2002crystallineII}, we have the following lemma.
	
	\begin{lemma} \label{lem:int-log-con}
		Let $E$ be an overconvergent log isocrystal on $(X, Y)^{\log}/ \Tf$ and let $\nabla: \E \to \E \otimes_{\O_{]X[^{\log}_\Pf}} \omega_{]X[^{\log}_\Pf / \Tf_K}$ as above. Then $\nabla$ is integrable.
	\end{lemma}
	
	\subsection{Equivalence of categories} \label{subsec:equiv-cat}
	In this subsection, we consider overconvergent isocrystals over $(\Bf, M_\Bf)$. Let $\Yf, \Zf, \Pf, \Df$ and $X, Y, Z, \allowbreak P, D$ be as in Hypothesis \ref{hypo:ss-pair}. The relative strict normal crossings divisor $\Df$ of $\Pf$ defines the fine log structure on $\Pf$, which we denote by $M_{\Pf}.$ Let $M_{\Yf}$ be the pullback of $M_{\Pf}$, which is the fine log structure of $\Yf$ defined by $\Zf$. The canonical morphism $(\Yf, M_\Yf) \to (\Pf, M_\Pf)$ is strict.
	
	We follow \cite[\S6.2]{kedlaya2007semistable} to construct factorizations as in Lemma \ref{lem:tub-nei-form} of $(k+1)$-fold fiber products of $(\Pf, M_\Pf)$. In what follows we let $(\LFS, \pLFS, \LS) = (\LFS/\Bf, \pLFS/\Bf, \LS/B)$.
	
	Let $\Yf(k)$ and $\Pf(k)$ be $(k+1)$-fold fibered product over $\V[[t]]$ and let $\mathrm{pr}_i$ be the $i$-th projection ($0 \leq i \leq k$).
	Let $M_{\Pf(k)}$ be the log structure $\mathrm{pr}_0^* M_\Pf \oplus_{p^*M_\Bf} \cdots \oplus_{p^*M_\Bf} \mathrm{pr}_k^* M_\Sf$, where $\oplus$ is the pushout here and $p: \Pf(k) \to \Bf$ is the structure morphism.
	One can check by a direct calculation that this log structure $M_{\Pf(k)}$ is fine and hence $(\Pf(k), M_{\Pf(k)})$ is the $(k+1)$-fold fibered product over $(\Bf, M_\Bf)$.
	
	Let $\wh{\mathbb{A}}^{m(k+1)^2}_{\V[[t]]}$ be the affine space over $\V[[t]]$ with variables $u_i^{(l,l')}$ ($i=1, \dots, m$, $l,l'=0, \dots, k$). Let $t^{(l)}_i = \mathrm{pr}_l^* t_i$ be the section of $\Pf(k).$
	Let $\Pf'(k)$ be the schematic image of the graph $\Pf(k)^{\mathrm{triv}} \to \Pf(k) \times_{\V[[t]]} \wh{\mathbb{A}}_{\V[[t]]}^{(n+m)(k+1)^2}$ induced by $u_i^{(l, l')} \mapsto t_i^{(l)}/t_i^{(l')}$. Here ``triv'' means the open locus where the log structure is trivial. Then the rational map $P \hookrightarrow \Pf(k) \dashrightarrow \Pf'(k)$ is in fact a regular morphism, which we will denote by $i'(k): P \hookrightarrow \Pf'(k)$.
	
	Explicitly, we have
	$$\Pf'(k) = \Spf \V[[t]] \langle t_i^{(0)}, t_i^{(1)}/t_i^{(0)} \dots, t_i^{(k)}/t_i^{(k-1)} \rangle_{i = 1, \dots, m} \langle t_i^{(0)}, \dots, t_i^{(k)} \rangle_{i = m+1, \dots, n} /(t - t_1^{(l)} \dots t_r^{(l)})_{l = 0, \dots, k}.$$
	The divisor $V(t_1^{(0)} \cdots t_{m}^{(0)})$ defines a fine log structure $M_{\Pf'(k)}$ on $\Pf'(k)$. Since $V(t_i^{(0)}) \cap P = V(t_i)$, this implies as in \cite[Lemma 6.2.9]{kedlaya2007semistable} that $i'(k)$ defines an exact closed immersion $(P, M_P) \hookrightarrow (\Pf'(k), M_{\Pf'(k)})$ with $M_P$ induced by $D$. Note that both $(\Pf'(k), M_{\Pf'(k)})$ and $(\Pf(k), M_{\Pf(k)})$ are formally log smooth over $(\Sf, M_\Sf)$. Let $f'(k) : (\Pf'(k), M_{\Pf'(k)}) \to (\Pf(k), M_{\Pf(k)})$ be the canonical morphism. Since $\dlog (t_i^{(l)}/t_i^{(l')}) = \dlog t_i^{(l)} - \dlog t_i^{(l')}$, the canonical map $f'(k)^* \omega_{(\Pf(k), M_{\Pf(k)})/(\Sf, M_\Sf)} \to \omega_{(\Pf'(k), M_{\Pf'(k)})/(\Sf, M_\Sf)}$ is an isomorphism. Therefore by \cite[Proposition 3.12]{kato1988logarighmic}, we conclude that $f'(k)$ is formally log etale.
	
	Define $(\Yf'(k), M_{\Yf'(k)})$, $i'(k)$ and $f'(k)$ (by abuse of notation) so that the diagram
	\[
	\xymatrix{
		(Y, M_Y) \ar[r]^(0.40){i'(k)} \ar[d] & (\Yf'(k), M_{\Yf'(k)}) \ar[r]^{f'(k)} \ar[d] & (\Yf(k), M_{\Yf(k)}) \ar[d] \\
		(\Pf, M_\Pf) \ar[r] & (\Pf'(k), M_{\Pf'(k)}) \ar[r] & (\Pf(k), M_{\Pf(k)})
	}
	\]
	is commutative with all squares cartesian. By arguments above, we get the following lemma.
	\begin{lemma}
		The closed immersion $i'(k) : (Y, M_Y) \hookrightarrow (\Yf'(k), M_{\Yf'(k)})$ is exact and $f'(k) : (\Yf'(k), M_{\Yf'(k)}) \to (\Yf(k), M_{\Yf(k)})$ is formally log etale.
	\end{lemma}
	
	With these preparation, we can finally state the equivalence of categories of overconvergent log $\nabla$-modules and overconvergent log isocrystals.
	
	\begin{theorem} \label{thm:cat-eq-log-isoc-log-nabla}
		Under Hypothesis \ref{hypo:ss-pair}, there is an equivalence between the category $I^{\log}_{\oc}(Y_\eta, Y/\Bf)$ of overconvergent log isocrystals over $(Y_\eta, Y)^{\log} / \Bf$ and the category of overconvergent log $\nabla$-modules over $]Y_\eta[_{\Yf}/\Ed_K$ with respect to the log structure induced by $Z$.
	\end{theorem}
	
	\begin{proof}
		Since $\Yf'(k) \to \Yf(k)$ is etale in an open neighborhood of $X$, by strong fibration theorem \ref{thm:strong-fibration}, we have an isomorphism $]X[^{\log}_{\Yf(k)} = ]X[_{\Yf'(k)} \simeq ]X[_{\Yf(k)}.$
		Given an overconvergent log isocrystal $E$, $E$ restricts to an overconvergent isocrystal on $(X, Y)$ by this isomorphism and $\epsilon$ induces the data of a overconvergent log $\nabla$-module by \ref{lem:int-log-con}
		Conversely, given a convergent log $\nabla$-module $\mathcal{E}$, we define the $\mathcal{O}_{]Y[_{\Yf'(1)}}$-linear isomorphism $\epsilon: p_1^*\mathcal{E} \rightarrow p_0^*\mathcal{E}$ by
		$$
		1 \otimes v \mapsto \sum_{i,i'}\left(\prod_{j=1}^{n+m} \frac{(u_j-1)^{i_j}}{i_j!} \right) \otimes \left( \prod_{j=1}^n \prod_{l=0}^{i_j-1} \left( t_j \frac{\partial}{\partial t_j}-l \right) \right)v
		$$
		where $u_j = t_j^{(1)}/t_j^{(0)}$. Thanks to Lemma \ref{lem:conv-mul-seq}, this sum uniformly converges on the subspace of $]Y_\eta[^{\log}_{\Yf(1)} = ]Y_\eta[_{\Yf'(1)}$ given by $\max |u_j-1|_x \leq \lambda$ for any $\lambda \in (0,1) \cap \Gamma^*$. Hence it is defined on all of  $]Y_\eta[_{\Yf(1)}$. This $\mathcal{E}$ satisfies $\Delta^*\epsilon = \mathrm{id}$ and the cocycle condition $p_{13}^*\epsilon = (p_{12}^*\epsilon)(p_{23}^*\epsilon)$ by direct calculation or deduce it by observing that it holds after passing to $]X[^{\log}_{\Yf(2)} = ]X[_{\Yf'(2)}$ as remarked in the proof of \cite[Theorem 6.4.1]{kedlaya2007semistable}.
	\end{proof}
	
	\begin{remark}
		The above equivalence of categories is compatible with restrictions to open subschemes and can be Zariski locally glued to the global case.
	\end{remark}
	
	\begin{definition}
		An overconvergent log isocrystal $E$ over $(Y_\eta, Y)/\Bf$ is said to have nilpotent residues if the corresponding log $\nabla$-module $\E$ does. More precisely, if $(Y, Z)$ is a strictly semistable pair, then an overconvergent log isocrystal $E$ on $(Y_\eta, Z_\eta) \hookrightarrow (Y, Z)$ is said to have nilpotent residues if there is an open covering $Y = \bigcup_i U_i$ such that $(U_i, U_i \cap Z)$ satisfies Hypothesis \ref{hypo:ss-pair} and on each of $U_i$, the corresponding log $\nabla$-module $\E_i$ does.
	\end{definition}
	
	\subsection{Semistable reduction theorem for $\Ed_K$-valued overconvergent isocrystals}
	Combining Lemma \ref{lem:ext-log-nabla} and Theorem \ref{thm:cat-eq-log-isoc-log-nabla}, we obtain the following theorem.
	
	\begin{theorem} \label{thm:log-ext-unip-monodromy}
		Let $(Y, Z)$ be a strictly semistable pair, $X = Y-Z$ and $E$ an $\Ed_K$-valued overconvergent isocrystal on $(X, Y)$. Then $E$ has unipotent monodromy along $Z_\eta$ if and only if $E$ extends to a locally free overconvergent log isocrystal on $(Y_\eta^\sharp, Y^\sharp) / \Bf^\sharp$ with nilpotent residues.
		Moreover, the restriction functor $I^{\log}_\oc(Y_\eta, Y/\Bf) \to \Isoc(X, Y / \Ed_K)$ is fully faithful.
	\end{theorem}
	
	Combining Theorem \ref{thm:des-unip}, Theorem \ref{thm:log-ext-unip-monodromy} and \cite[Theorem 6.4.5]{kedlaya2007semistable}, we can finally show that log extendability is equivalent after taking the completion:
	
	\begin{theorem} \label{thm:log-ext-compl}
		Let $(Y, Z)$ be a strictly semistable pair, $X=Y-Z$ and $E$ an $\Ed_K$-valued isocrystal on $(X, Y)$. Then $E$ extends to a locally free overconvergent log isocrystal on $(Y_\eta^\sharp, Y^\sharp) / \Bf^\sharp$ with nilpotent residues if and only if $E$ extends to a locally free convergent log isocrystal on $Y_\eta^\sharp / \Spf \O_{\E_K}$ with nilpotent residues.
	\end{theorem}
	Now we can show the first main theorem of this paper.
	
	\begin{theorem}\label{thm:main-edk}
		Let $X$ be a $k((t))$-variety, $Y$ a $k[[t]]$-variety, $X \subseteq Y$ an open immersion and $E$ an $\Ed_K$-valued overconvergent $F$-isocrystal on $X$. Then after base changnig $k[[t]]$ to $k^{1/q^n}[[t^{1/q^n}]]$, there is a finite extension $k[[t]] \hookrightarrow k'[[t']]$, a $k'[[t']]$-scheme $Y'$, a projective alteration $f : Y' \to Y$ over $k[[t]]$ such that for $Z' = (f^{-1}(Y_\eta-X))_\red \cup Y'_{s'}$, $(Y', Z')$ is a strictly semistable pair over $k'[[t']]$ and that $f^*E$ extends to an overconvergent log isocrystal over $({Y'}_{\eta'}^\sharp, {Y'}^\sharp) / {\Bf'}^\sharp$.
		Here $\eta'$ is the closed point of $k'[[t']]$, $\Bf' = \Spf \V'[[t']]$ with $\V'$ the ring of integers of the field $K'$ which is the finite separable extension of $K$ with residue field $k'$.
	\end{theorem}
	\begin{proof}
		Put $Z = Y-X$. Applying Kedlaya's semistable reduction theorem \cite[Theorem 2.4.4]{kedlaya2011semistable} to the overconvergent isocrystal $\wh{E}$ on $(X, Y_\eta)$, after replacing $k((t))$ by $k^{1/q^n}((t^{1/q^n}))$ for some $n$, there is a projective alteration $f_1 : Y_1 \to Y_\eta$ such that for $Z_1 = f^{-1}(Z_\eta)$, $(Y_1, Z_1)$ is a smooth pair and $f^*(\wh{E})$ extends to a locally free convergent log isocrystal over $Y_1^\sharp = (Y_1, Z_1)$ with nilpotent residues. Choose a projective alteration $f_2 : Y_2 \to Y$ of $k[[t]]$-varieties such that $f_2 : f_2^{-1}(Y_\eta) \to Y_\eta$ factors into $f_2^{-1}(Y_\eta) \xrightarrow{g} Y_1 \xrightarrow{f_1} Y_\eta.$ Applying de Jong's alteration theorem \ref{thm:alteration} to $\ol{g^{-1}(Z_1)} \cup Y_{2, s} \subseteq Y_2$, we get a finite extension $k[[t]] \hookrightarrow k'[[t']]$, a projective alteration $h : Y_3 \to Y_2$ over $k[[t]]$ and an open immersion $Y_3 \hookrightarrow \ol{Y_3}$ into a projective $k'[[t']]$-variety such that $(\ol{Y_3}, (h^{-1}(\ol{g^{-1}(Z_1)}) \cup \ol{Y_3}_{, s'})_\red)$ is a strictly semistable pair over $k'[[t']]$. Thus, we get a projective alteration $f : Y' \to Y$ over $k[[t]]$ such that for $Z' = f^{-1}(Z)$, $(Y', Z')$ is a strictly semistable pair over $k'[[t']]$ and that $f_\eta^*(\wh{E})$ extends to a locally free convergent log isocrystal over ${Y'}_{\eta'}^\sharp = (Y'_{\eta'}, Z'_{\eta'})$.
		By Theorem \ref{thm:log-ext-compl}, $f^*(E)$ extends to a locally free overconvergent log isocrystal over $({Y'_{\eta'}}^\sharp, {Y'}^\sharp) / {\Bf'}^\sharp.$
	\end{proof}
	
	\section{Semistable reduction for $K$-valued overconvergent isocrystals} \label{sec:ssr-rel-K}
	In this section, we will prove the semistable reduction theorem for $K$-valued overconvergent isocrystals.
	
	By the method of applying Kedlaya's semistable reduction theorem using descent of unipotence, we can only make monodromy along divisors dominating $k[[t]]$ being unipotent, and this results in the fact that one can only extends to an overconvergent log isocrystal over $(Y_\eta, Y)^{\log}/\V[[t]].$ In order to extend to a convergent log isocrystal over $Y/\V$, it is also necessary to consider monodromy along divisors contained in $Y_s$ and the derivation by $t$.
	
	\subsection{Log $\nabla$-modules relative to $K$}
	In the situation of Hypothesis \ref{hypo:ss-pair}, let $\Zf_i = V(t_i)$ for $i = 1, \dots, m+n$ and $Z_i = \Zf_i \times_{\V[[t]]} k[[t]].$ $Z_1, \dots, Z_m$ are the irreducible components of $Y_s$ and $Z_{m+1}, \dots, Z_{m+n}$ are those of $Z$ which dominates $k[[t]]$.
	
	\begin{lemma} \label{lem:rel-ann}
		There is an isomorphism $]Z_i[_{\Yf} \, \simeq \, ]Z_i[_{\Zf_i} \times_K A^1_K [0, 1)$ whose projection to $A^1$ is given by $z_i$, for $i = 1, \dots, m+n$.
	\end{lemma}
	\begin{proof}
		Note that $\Zf_i$ is formally smooth over $\V$. Thus, similarly as in \cite[Lemma 4.3.3]{kedlaya2007semistable}, we have $\Gamma(]Z_i[_{\Zf_i}, \O^+)[[t_i]] \simeq \Gamma(]Z_i[_{\Yf}, \O^+).$ This yields the desired isomorphism.
	\end{proof}
	
	\begin{definition} \label{def:almost-smooth}
		Let $f: \Xs \to \Ys$ be a morphism of rigid analytic spaces. We say that $f$ is almost smooth if locally for $\Xs$, $f$ factors through a smooth morphism $\Xs \to \Ys \times \Spa \V[[t]]_K \langle t_1, \dots, t_r \rangle /(t - t_1 \cdots t_r)$ followed by the projection.
		
		We say that $x'_1, \dots, x'_n, t'_1, \dots, t'_r \in \Gamma(X, \O_X)$ is a coordinate of $f$ if $t'_1 \cdots t'_r = t$ and the morphism $\Xs \to \Ys \times \Spa \V[[t]]_K \langle x_1, \dots, x_n, t_1, \dots, t_r \rangle /(t - t_1 \cdots t_r)$ induced by $x_i \mapsto x'_i, \, t_j \mapsto t'_j$ is etale.
		
		A rigid analytic space $X$ is said to be almost smooth if the structure morphism $X \to \Spa \V[[t]]_K$ is almost smooth.
	\end{definition}
	
	\begin{remark}
		Note that if $X$ is an almost smooth partial dagger space over $\V[[t]]_K$, then $X \times_{\Spa \V[[t]]_K} \Spd \Ed_K = \{ x \in X : |t|_{[x]} = 1 \}$ is a smooth partial dagger space over $\Ed_K$ and $\{ x \in X : |t|_{[x]} < 1 \}$ is a ``classical'' rigid analytic space smooth over $K$, which means that it is locally of finite type over $K$.
	\end{remark}
	
	For an almost smooth rigid analytic space $X$, we can define log $\nabla$-modules (with nilpotent residues) over $X/K$ in the same way. We can prove same results of \S\ref{sec:log-nabla-mod} with some modifications in the proof. We will only sketch those modifications in this subsection. We first start with a useful local freeness criterion.
	
	\begin{lemma} \label{lem:loc-free2}
		Let $X$ be a partial dagger space over $\V[[t]]_K$, $X' = X \times_{\Spa \V[[t]]_K} \Spd \Ed_K = \{ x \in X : |t|_{[x]} = 1 \}$ and $U = X - X'.$ Write $i : X' \hookrightarrow X$ for the closed immersion. If a coherent $\O_X$-module $E$ is locally free over $U$ and $i^{-1}\E$ is locally free over $X'$, then $\E$ is locally free over $X$.
	\end{lemma}
	\begin{proof}
		The assertion is local, so we assume that $X$ is an affinoid. Let $\Xs$ be an ambient space of $X$ and let $\Xs_\eta$ be a cofinal system of open neighborhoods of $X$ in $\Xs$. Then $E$ comes from some coherent $\O_{\Xf_{\eta_0}}$-module $\E$.
		$X'$ has a cofinal system of open neighborhoods $\Xs_\eta \cap \{ x \in \Xs : |t|_x \geq \epsilon \}$ with $\eta$ and $\epsilon \in (0, 1) \cap \Gamma^*$. By considering the (closed) locus of a basis of $E$ fails to generate $\E$ or to be linearly independent, one concludes that $\E$ is locally free over $\Xs_\eta \cap \{ x \in \Xs : |t|_x \geq \epsilon \}$ for some $\eta, \epsilon$. On the other hand, since $E$ is locally free over $\{ x \in X : |t|_x \leq \epsilon \} \subset U$, by the same argument, after replacing $\eta$, one concludes that $\E$ is locally free over $\Xs_\eta \cap \{ x \in \Xs : |t|_x \leq \epsilon \}$. These two cover $\Xs_\eta$, so we get the assertion.
	\end{proof}
	
	\begin{lemma} \label{lem:lnm-abel2}
		Let $X$ be an almost smooth rigid analytic space over $\V[[t]]_K$ with coordinate $t_1, \dots, t_n$ such that $t_1 \cdots t_r = t$ and zero loci are almost smooth and meet transversely. Then for any morphism $f : \E \to \F$ of $\LNM_{X/K}$, $\ker f$ and $\coker f$ belong to $\LNM_{X/K}$.
	\end{lemma}
	\begin{proof}
		Let $X'$ and $U$ be as in Lemma \ref{lem:loc-free2}. After ``forgetting the derivation by $t$'', we get a morphism of $\LNM_{X'/\Ed_K}$. The kernel and cokernel of $f$ are locally free after restricted to $X'$ by the corresponding Lemma \ref{lem:lnm-abel}. Since $U$ is the ``classical'' rigid analytic space smooth over $K$, the kernel and cokernel are locally free over $U$ thanks to \cite[Lemma 3.2.13]{kedlaya2007semistable}.
	\end{proof}
	\begin{proposition}
		For any almost smooth rigid analytic space $X$, $\LNM_{X/K}$ is an abelian category. If $X = V \times_K A^d_K(I)$ for some almost smooth rigid analytic space $V$ and an aligned subinterval $I$ of $[0, \infty)$, the category $\ULNM_{X/K}$ is an abelian tensor subcategory of $\LNM_{X/K}$.
	\end{proposition}
	\begin{proof}
		The first part is clear from Lemma \ref{lem:lnm-abel2}. For $\ULNM$ being an abelian tensor subcategory, we follow the proof of \cite[Proposition 3.2.20]{kedlaya2007semistable}. Using the same symbol as in {\it loc. cit.}, we have to show that $\E' \to \E/\pi_1^*\F$ and $\pi_1^*\F \to \E/\E'$ are zero maps. Let $X'$ and $U$ be as in Lemma \ref{lem:loc-free2}, this can be checked over $X'$ and $U$ separately; over $X'$ we can invoke Proposition \ref{prop:lnm-abel} and over $U$ we can invoke \cite[Proposition 3.2.20]{kedlaya2007semistable}.
	\end{proof}
	
	With these propositions in hand, we can show that for a quasi-open subinterval $I$ of $[0, \infty)$, the functor $\Uc_I$ is an equivalence of categories:
	
	\begin{theorem} \label{thm:unip-equiv2}
		Let $V$ be an almost smooth rigid analytic space and $I$ a quasi-open subinterval of $[0, \infty)$. Then the functor
		$$ \Uc_I : \ULNM_{V \times_K A^d_K[0, 0]} \to \ULNM_{V \times_K A^d_K(I)}$$
		is an equivalence of categories.
	\end{theorem}
	
	Now we can follow the entire argument of \cite[3.4]{kedlaya2007semistable} and \S\ref{subsec:descent} to prove the following theorem, which together with Theorem \ref{thm:des-unip} implies that a locally free log $\nabla$-module $\E$ relative to $K$ being constant/unipotent can be checked after ``forgetting the derivation by $t$'':
	
	\begin{theorem} \label{thm:des-unip2}
		Let $V = \Spa A$ be an affinoid almost smooth rigid analytic space with coordinate $t_1, \dots, t_n$ such that $t_1 \cdots t_r = t$, let $V' = V \times_{\Spa \V[[t]]_K} \Spd \Ed_K = \{ x \in V : |t|_{[x] = 1} \}$ and let $B = A \langle t^{-1} \rangle$. Suppose that:
		\begin{itemize}
			\item the spectral norm $| \cdot |_{\sup}$ of the affinoid algebra $B$ over $\E_K$ is a valuation;
			\item for $Z_i$ the zero locus of $t_i$, $Z_i$ locally admits an open neighborhood of the form $Z_i \times_K A^1_K[0, \eta)$ in $V$.
		\end{itemize}
		Let $L$ be a field containing $B$, which is complete with respect to a valuation restricting to the spectral norm of $B$. Then for any quasi-open aligned subinterval $I \subseteq [0, 1)$, the following hold:
		\begin{itemize}
			\item For any $\E \in \LNM_{V \times_K A^d_K(I)/K}$, $\E$ is  constant/unipotent if and only if the induced object $\F \in \LNM_{A^d_L(I)/L}$ is constant/unipotent.
			\item For any $\E \in \LNM_{V' \times_K A^d_K(I)/K}$, $\E$ is  constant/unipotent if and only if the induced object $\F \in \LNM_{A^d_L(I)/L}$ is constant/unipotent.
		\end{itemize}
		Remark that $V' \times_K A^d_K(I) \simeq V' \times_{\Ed_K} A^d_{\Ed_K}(I)$, an object $\E \in \LNM_{V' \times_K A^d_K(I)/K}$ induces an object $\E^\circ \in \LNM_{V' \times_{\Ed_K} A^d_{\Ed_K}(I) / \Ed_K}$. Then $\E$ is constant/unipotent if and only if $\E^\circ$ is constant/unipotent.
	\end{theorem}
	
	\subsection{Equivalence of categories revisited}
	\begin{lemma} \label{lem:conv-mul-seq2}
		Under Hypothesis \ref{hypo:ss-pair}, let $x_i = t_i, \, y_j = t_{m+j}$ and suppose that $\Yf$ is affine and let $\E$ be a $\nabla$-module over $]Y[_{\Yf}$ relative to $K$ which restricts to an overconvergent isocrystal over $]X[_\Yf$. For $\eta \in (0, 1) \cap \Gamma^*$, let
		$$\Ys_{\lambda} = \{ x \in \Yf^\ad : |t|_x \geq \lambda \}.$$
		Choose a log $\nabla$-module $\E_{\lambda_0}$ on $\Ys_{\lambda_0}$ which gives rise to $\E$. Let $v \in \Gamma(\Ys_{\lambda}, \E_{\lambda_0})$ for some $\lambda \in [\lambda, 1) \cap \Gamma^*$. Then for any $\eta \in (0, 1)$, there exists $\lambda \in [\lambda', 1) \cap \Gamma^*$ such that the multi-sequence in $(\nu_1, \nu_2, \dots, \nu_n)$
		$$ \frac{1}{\nu_1! \cdots \nu_n!}
		\left( \prod_{i=1}^{m} \prod_{l=0}^{\nu_j-1} \left( \frac{\partial}{\partial t_j} - l \right) \right)
		\left( \prod_{j=m+1}^{n} \frac{\partial^{\nu_j}}{\partial t_j^{\nu_j}} \right)v $$
		is $\eta$-convergent on $\Ys_{\lambda'}.$
	\end{lemma}
	\begin{proof}
		Similar to Lemma \ref{lem:conv-mul-seq}.
	\end{proof}
	
	Under Hypothesis \ref{hypo:ss-pair}, let $\Yf(k)$ and $\Pf(k)$ be $(k+1)$-fold fibered product over $\V$, let $p_i: \Pf(k) \to \Pf$ be the $i$-th projection and let $M_{\Pf(k)}$ be the log structure $p_0^*M_{\Pf} \oplus \cdots \oplus p_k^*M_{\Pf}$. Then $M_{\Pf(k)}$ is fine and therefore $(\Pf(k), M_{\Pf(k)})$ is the fibered product of $(\Pf, M_{\Pf})$ in the category $\pLFS/\V$.

	Let $S(k) = \V[[t^{(0)}]] \wh{\otimes}_\V \cdots \wh{\otimes}_\V \V[[t^{(k)}]]$ and let $t_i^{(l)} = p_l* t_i$ be the section of $\Pf(k)$ pulled back by the $l$-th projection $\Pf(k) \to \Pf$. Then we have
	$$\Pf(k) = \Spf S(k)\langle t_i^{(0)}, \dots, t_i^{(k)} \rangle_{i = 1, \dots, n}/(t^{(l)} - t_1^{(l)} \dots t_r^{(l)})_{l = 0, \dots, k}.$$
	Let
	$$\Pf'(k) = \Spf S(k)\langle t_i^{(0)}, t_i^{(1)}/t_i^{(0)}, \dots, t_i^{(k)}/t_i^{(k-1)}\rangle_{i=1, \dots, m} \langle t_i^{(0)}, \dots, t_i^{(k)} \rangle_{i = m+1, \dots, n} / (t^{(l)} - t_1^{(l)} \dots t_r^{(l)})_{l = 0, \dots, k}.$$
	
	The divisor $V(t_1^{(0)} \cdots t_m^{(0)})$ of $\Pf'(k)$ defines a log structure $M_{\Pf'(k)}$ on $\Pf'(k)$. There exists canonical morphisms of log formal schemes $i'(k) : (P, M_P) \hookrightarrow (\Pf'(k), M_{\Pf'(k)})$ and $f(k) : (\Pf'(k), M_{\Pf'(k)}) \to (\Pf(k), M_{\Pf(k)})$ whose composition is the closed immersion $i(k) : (P, M_P) \hookrightarrow (\Pf(k), M_{\Pf(k)}).$ By the same argument as in \S\ref{subsec:equiv-cat}, $i'(k)$ is an exact closed immersion and $f(k)$ is formally log etale. We can use $\Pf'(k)$ to calculate the log tubular neighborhood $]Y[^{\log}_{\Yf(k)}$ and similarly as in \S\ref{subsec:equiv-cat}, we get the following theorem:
	
	\begin{theorem} \label{thm:cat-eq-log-isoc-log-nabla2}
		Under Hypothesis \ref{hypo:ss-pair}, there is an equivalence between the category $I^{\log}_{\mathrm{conv}}(Y/\V)$ of convergent log isocrystals over $Y^\sharp / \V$ (resp. the category $I^{\log}_{\oc}(Y_\eta, Y/\V)$ of overconvergent log isocrystals over $(Y_\eta^\sharp, Y^\sharp)/\V$) and the category of overconvergent log $\nabla$-modules over $]Y[_{\Yf}/K$ (resp. the category of overconvergent log $\nabla$-modules over $]Y_\eta[_{\Yf}/K$) with respect to the log structure induced by $Z$.
	\end{theorem}
	
	The above equivalence glue to give the global case as before, and we make the following similar definition:
	
	\begin{definition}
		A convergent log isocrystal $E$ over $Y/\V$ is said to have nilpotent residues if the corresponding log $\nabla$-module $\E$ does. More precisely, if $(Y, Z)$ is a strictly semistable pair, then a convergent log isocrystal $E$ over $Y/\V$ is said to have nilpotent residues if there is an open covering $Y = \bigcup_i U_i$ such that $(U_i, U_i \cap Z)$ satisfies Hypothesis \ref{hypo:ss-pair} and on each of $U_i$, the corresponding log $\nabla$-module $\E_i$ does.
	\end{definition}
	
	\subsection{Monodromy along horizontal and vertical divisors}
	When an overconvergent isocrystal has a derivation by $t$ and has constant/unipotent monodromy along $Z_\eta$, one can enhance the extension lemma \ref{lem:ext-lemma} as follows.
	
	\begin{lemma} \label{lem:ext2}
		Under Hypothesis \ref{hypo:ss-pair}, let $E$ be a $K$-valued overconvergent isocrystal over $(X, Y)$. Then $E$ viewed as an $\Ed_K$-valued overconvergent isocrystal has constant/unipotent monodromy along $Z_\eta$ (see Definition \ref{def:monodromy-restricted}) if and only if $E$ extends to a locally free log $\nabla$-module with nilpotent residues over $]Y_{\eta}[_{\Yf}$ relative to $K$. Moreover, the extension is unique up to isomorphism if it exists.
	\end{lemma}
	\begin{proof}
		This proposition is a consequence of Theorem \ref{thm:unip-equiv2} and \ref{thm:des-unip2}.
	\end{proof}
	
	\begin{definition}
		Let $Y$ be an affine smooth $k[[t]]$-variety, let $Z = Y_s$, $X = Y-Z$ and let $\Yf$ be a lift of $Y$ to a smooth formal $\V[[t]]$-scheme. For a $K$-valued overconvergent isocrystal $E$ over $(X, Y)$, let $\E$ be the realization of $E$ over the smooth frame $(X, Y, \Yf)$. Then $\E$ induces a convergent $\nabla$-module over $V = ]Y_s[_{\Yf_s} \times_K A^1_K(\eta, 1)$ for some $\eta \in (0, 1) \cap \Gamma^*$. We say that $\E$ has constant/unipotent monodromy along $Y_s$ if $\E$ is constant/unipotent over $V$.
		
		In general, for a strictly semistable pair $(Y, Z)$, $X = Y-Z$ and a $K$-valued overconvergent isocrystal $E$ over $(X, Y)$, we say that $E$ has constant/unipotent monodromy along the closed fiber $Y_s$ if for any open subset $U \subset Y$ smooth over $k[[t]]$ such that $U \cap Z = U_s$, $E|(U_\eta, U)$ has constant/unipotent monodromy along $U_s$.
	\end{definition}

	\begin{lemma} \label{lem:ext3}
		Under Hypothesis \ref{hypo:ss-pair}, let $E$ be a $K$-valued overconvergent isocrystal over $(X, Y)$. Then $E$ has constant/unipotent monodromy along $Z_\eta$ and $Y_s$ if and only if $E$ extends to a locally free log $\nabla$-module with nilpotent residues over $]Y[_\Yf$ relative to $K$. Moreover, the extension is unique up to isomorphism if it exists.
	\end{lemma}
	\begin{proof}
		Let us prove ``only if'' part. We may start from a locally free log $\nabla$-module relative to $K$ with nilpotent residues over some open neighborhood $V = \{ x \in ]Y[_{\Yf} : |t_1|_x, \dots, |t_r|_x \geq \epsilon \}$ for some $\epsilon \in (0, 1) \cap \Gamma^*$, thanks to Lemma \ref{lem:ext2}. The proof is similar to Lemma \ref{lem:ext-lemma}. For $i = 1, \dots, r$, let $\Zf_i = V(t_i) \subseteq \Yf$, let $Z_i$ be the special fiber and choose an affine open subset $U_i$ of $Y$ which is smooth over $k[[t]]$ and $U_i \cap Z = U_i \cap Z_i = U_{i, s}$. Let $\Uc_i$ be the open subset of $\Yf$ whose special fiber is $U_i$. Then by assumption $\E$ induces a constant/unipotent $\nabla$-module over $]U_{i, s}[_{\Uc_i \cap \Zf_i} \times_K A_K^1(\eta, 1)$ for some $\eta \in (0, 1) \cap \Gamma^*$. By Theorem \ref{thm:des-unip2}, this implies that $\E$ is constant/unipotent over $]Z_i[_{\Zf_i} \times_K A^1_K(\eta, 1)$ and hence extends uniquely to $]Z_i[_{\Zf_i} \times_K A^1_K[0, 1).$ These extensions glue together with the one over $V$ to be a log $\nabla$-module over $]Y[_\Yf$.
		
		The converse direction and uniqueness assertion follow from Lemma \ref{lem:ext-lemma} and \cite[Lemma 5.1.1]{kedlaya2007semistable}.
	\end{proof}
	
	\subsection{Semistable reduction theorem for $K$-valued overconvergent isocrystals}
	\begin{lemma} \label{lem:ext-monodromy}
		Let $X$ be a $k((t))$-variety, $X \hookrightarrow Y$ an open immersion into a $k[[t]]$-variety and $Z = X-Y$. For a $K$-valued overconvergent isocrystal $E$ on $(X, Y)$, suppose that there exist finite extensions $k[[t]] \hookrightarrow k'[[t']]$ and $k[[t]] \hookrightarrow k''[[t'']]$, strictly semistable pairs $(Y', Z')$ over $k'[[t']]$ and $(Y'', Z'')$ over $k''[[t'']]$, alterations $f' : Y' \to Y$ and $f'' : Y'' \to Y$ over $k[[t]]$ with the following properties:
		\begin{itemize}
			\item For $X' = Y' - Z'$ and $X'' = Y'' - Z''$, we have $X' \subseteq f'^{-1}(X)$ and $X'' \subseteq f''^{-1}(X)$; hence they induce morphisms of pairs $f' : (X', Y') \to (X, Y)$ and $f'' : (X'', Y'') \to (X, Y)$.
			\item The $K'$-valued overconvergent isocrystal $f'^*(E)$ over $(X', Y')$ has unipotent monodromy along $Z'_{\eta'}$.
			\item The $K''$-valued overconvergent isocrystal $f''^*(E)$ over $(X'', Y'')$ has unipotent monodromy along $Y''_{s''}.$
		\end{itemize}
		Then there exist a finite extension $k[[t]] \hookrightarrow k_1[[t_1]]$, a strictly semistable pair $(Y_1, Z_1)$ over $k_1[[t_1]]$ and an alteration $f_1 : Y_1 \to Y$ over $k[[t]]$ such that for $X_1 = Y_1 - Z_1$, we have $X_1 \subseteq f^{-1}(X)$ and hence $f_1$ induces a morphism of pairs $f_1 : (X_1, Y_1) \to (X, Y)$ and that $K_1$-valued overconvergent isocrystal $f_1^*(E)$ on $(X_1, Y_1)$ has unipotent monodromy along $Z_{1, \eta_1}$ and along $Y_{1, s_1}.$
	\end{lemma}
	\begin{proof}
		Choose a finite extension $k_1[[t_1]]$ of $k[[t]]$ which contains both $k'[[t']]$ and $k''[[t'']].$ Put $X'_1 = X' \otimes_{k'((t'))} k_1((t_1))$, $Y'_1 = Y' \otimes_{k'[[t']]} k_1[[t_1]]$ and similarly for double primes $X''_1$, $Y''_1$. Let $X_1 = X'_1 \times_{k_1((t_1))} X''_1$ and $Y_1 = Y'_1 \times_{k_1[[t_1]]} Y''_1$. Applying de Jong's alteration to the pair $Z_1 := Y_1 - X_1 \subseteq Y_1,$ we get a finite extension $k_1[[t_1]] \hookrightarrow k_2[[t_2]]$, an alteration $f_2 : Y_2 \to Y_1$ over $k_1[[t_1]]$ such that for $Z_2 = f_2^{-1}(Z_1)_\red$, $(Y_2, Z_2)$ is a strictly semistable pair over $k_2[[t_2]].$ We have a commutative diagram of pairs
		\[
		\xymatrix{
			        & (X', Y')   \ar[ld] & (X'_1, Y'_1)   \ar[l] &            & \\
			 (X, Y) &                    &                       & (X_1, Y_1) \ar[lu] \ar[ld] & (X_2, Y_2) \ar[l] \\
			        & (X'', Y'') \ar[lu] & (X''_1, Y''_1). \ar[l] &            &
		}
		\]
		Since the pullback $E'$ of $E$ to $(X', Y')$ has unipotent monodromy along $Z'_{\eta'}$, the pullback $E_2$ of $E$ to $(X_2, Y_2)$, which coincides with the pullback of $E'$, has unipotent monodromy along $Z_{2, \eta_2}$. Similarly, $E_2$ must have unipotent monodromy along $Y_{2, s_2}.$
	\end{proof}
	
	Since we already have semistable reduction theorem \ref{thm:main-edk} regarding the monodromy along horizontal divisors, by this lemma we can now focus on the unipotence of monodromy along vertical divisors. For this, we follow the method of \cite{kedlaya2022monodromy}. We refer to some of its definitions.
	
	\begin{definition}
		Let $X$ be an analytic adic space. For a point $x \in X$, let $\mathcal{H}(x)$ be the completed residue field at $x$. Note that it is determined by the maximal generalization $[x]$ of $x$. For an adic morphism $f : X \to Y$ of analytic adic spaces and a point $y \in Y$, we define the fiber $X_y$ to be $X \times_Y \Spa (\mathcal{H}(x), \mathcal{H}(x)^\circ).$ Given an aligned subinterval $I$ of $(0, \infty)$ and a $\nabla$-module $\E$ over a relative annulus $X \times A^1_K(I) \to X$, it induces $\nabla$-modules at each fiber $A^1_{\mathcal{H}(x)}(I)$. 
		\begin{itemize}
			\item For $I = (\epsilon, 1)$ or $I = [\epsilon, 1)$, we say that $\E$ is solvable if every $\E_x$ is solvable, i.e., the intrinsic radius $\mathrm{IR}(\E_x, \rho)$ at $\rho \in I$ (cf. \cite[Definition 1.1.2]{kedlaya2022monodromy}) converges to 1 as $\rho \to 1^-$.
			\item We say that $\E$ is regular if every $\E_x$ is regular, i.e., $\mathrm{IR}(\E_x, \rho) = 1$ for every $\rho \in I$ (cf. \cite[Definition 1.1.4]{kedlaya2022monodromy}).
			\item For a quasi-compact $X$, the category $\mathcal{C}_X$ is the 2-limit of the categories of solvable $\nabla$-modules over $X \times A^1_K(\epsilon, 1)$ (cf. \cite[Definition 2.4.1]{kedlaya2022monodromy}).
			\item For an object $\E$ of $\mathcal{C}_X$, we say that $\E$ is regular if for every $x \in X$, after restricting $I$ to some $(\epsilon, 1)$ with $\epsilon$ depending on $x$, $\E_x$ is regular. This is equivalent to saying that $b(\E_x) = 0$ (cf. \cite[Definition 2.4.1]{kedlaya2022monodromy}).
		\end{itemize}
	\end{definition}
	
	\begin{definition}[{cf. \cite[Definition 1.4.5]{kedlaya2022monodromy}}]
		Suppose $k$ is perfect. An eligible cover of $W = A^1_K(\epsilon, 1)$ is a finite etale cover $V \to W$ which comes from some finite etale cover of $\R_K^{\mathrm{int}} = \R_K \cap \V[[t, t^{-1}]]$. Note that an eligible cover $W$ of $V$ itself is an annulus $A^1_{K'}(\epsilon', 1)$ with a different parameter $u$ for some finite unramified extension $K'/K$. Since $\R_K^{\mathrm{int}}$ is henselian with residue field $k((z))$, an eligible cover corresponds to a finite separable extension of $k((t))$.
		
		When $k$ is not perfect, $V \to W$ is an eligible cover if it there is some finite extension $L/k((t))$ such that after making some base change to finite extension $K'/K$ so that the components of $L \otimes_{k((t))} k'((t))$ are separable and that $V \otimes_K K' \to W \otimes_K K'$ splits into covers corresponding to those separable extensions.
		
		An eligible cover of $X \times A^1_K(\epsilon, 1)$ is a finite etale cover $Y \to X \times A^1_K(\epsilon, 1)$ such that for every $x \in X$, $Y_x \to A^1_{\mathcal{H}(x)}(\epsilon, 1)$ is an eligible cover.
	\end{definition}
	
	\begin{theorem}[{\cite[Theorem 2.5.5]{kedlaya2022monodromy}}] \label{thm:eligible-regular}
		For a quasi-compact analytic adic space $X$ and $\E \in \mathcal{C}_X$, for any $x \in X$, there is an open neighborhood $U$ of $x$ and an eligible cover $V \to U \times A^1_K(\epsilon, 1)$ such that the pullback of $\E$ to $V$ is regular.
	\end{theorem}
	
	\begin{_empty} \label{empty:regularity}
		Suppose that $\Yf$ is a smooth formal $\V[[t]]$-scheme and let $\Zf = V(t) \subseteq \Yf$. For a $K$-valued overconvergent isocrystal $E$ over $(Y_\eta, Y)$, let $\E$ be the realization over the frame $(Y_\eta, Y, \Yf)$. Then $\E$ induces a $\nabla$-module over $]Z[_\Zf \times A^1_K(\epsilon, 1)$ for some $\epsilon \in (0, 1) \cap \Gamma^*.$ By overconvergence condition, it induces an object of $\mathcal{C}_{]Z[}$. Since $]Z[_\Zf$ is a smooth affinoid with one point boundary, we may argue the regularity of $E$ along $Z$ in this situation as in \cite[Definition 3.1.1]{kedlaya2022monodromy}.
		
		In general, for a strictly semistable pair $(Y, Z),$ $X = Y-Z$ and a $K$-valued overconvergent isocrystal $E$ over $(X, Y)$,  we say that $E$ has regular monodromy along the closed fiber if for any open subset $U$ of $Y$ smooth over $k[[t]]$ with $U \cap X = U_\eta$, the restriction of $E$ to $(U_\eta, U)$ has regular monodromy along the closed fiber. It is easy to see that the condition of having regular monodromy is local on $Y$.
		
		Furthermore, since regularity is checked at the Shilov boundary by \cite[Theorem 2.4.5]{kedlaya2022monodromy} and in this situation it consists of one point, which is the generic point of $]Z[$ using the notation above, the regularity of monodromy along a component of the closed fiber can be checked by only over those open subsets $U$ which contains the generic point.
	\end{_empty}

	\begin{theorem}[{(cf. \cite[Theorem 3.1.3]{kedlaya2022monodromy})}] \label{thm:regular-h}
		Let $X$ be a smooth $k((t))$-variety, $X \hookrightarrow Y$ an open immersion into a $k[[t]]$-variety and $E$ a $K$-valued overconvergent isocrystal on $(X, Y).$ Then there is an alteration $f : Y' \to Y$ such that for any alteration $g : Y'' \to Y'$ and a reduced closed subscheme $Z'' \subseteq Y''$ such that $(Y'', Z'')$ is a strictly semistable pair over some finite extension $k''[[t'']]$ of $k[[t]]$ and that $X'' := Y'' - Z'' \subseteq g^{-1}f^{-1}(X),$ the pullback $g^*f^*E$ to the pair $(X'', Y'')$ has regular monodromy along each component of the closed fiber of $Y''$.
	\end{theorem}
	\begin{proof}
		We first suppose that $Y$ is affine and $X = Y_\eta.$ By \cite[Theorem 2]{kedlaya2003more}, we have a finite etale morphism $f_\eta : Y_\eta \to \mathbb{A}^n_{k((t))}$. By multiplying $t$, we may assume that there is a morphism $f : Y \to \mathbb{A}^n_{k[[t]]}$ which induces $f_\eta$. We may choose a lift $\Yf \to \wh{\mathbb{A}}^n_{\V[[t]]}$ of $f$ such that $\Yf_\eta \to \wh{\mathbb{A}}^n_{\V[[t]]\langle t^{-1} \rangle}$ is finite etale. Choose a lift $f : Y \to \mathbb{A}^n_{k[[t]]}$ of $f_\eta$ and a lift $f : \Yf \to \wh{\mathbb{A}}^n_{\V[[t]]}$ to formal $\V[[t]]$-schemes. Let $A = \Gamma(\Yf, \O)$. The ring homomorphism $\V[[t]] \langle x_1, \dots, x_n \rangle \langle t^{-1} \rangle ^\dag \to A \langle t^{-1} \rangle ^\dag$ of partial dagger algebras induced by $f$ is finite etale. Indeed, etaleness can be checked after completion by \cite[Theorem 2.5]{monsky1968formal}; note that flatness is preserved after completion by faithful flatness, Corollary \ref{cor:pda}. Finiteness can be checked after modulo a uniformizer $\varpi$ using weak completeness. Therefore $f$ induces a finite etale morphism $]Y_\eta[_{\Yf} \to ]\mathbb{A}^n_{k((t))}[_{\wh{\mathbb{A}}^n_{\V[[t]]}}$. We can define finite etale pushforward
		$$ f_* : \Isoc(Y_\eta, Y, \Yf/K) \to \Isoc(\mathbb{A}^n_{k((t))}, \mathbb{A}^n_{k[[t]]}, \wh{\mathbb{A}}^n_{\V[[t]]}/K).$$
		For any $E \in \Isoc(Y_\eta, Y, \Yf/K)$, $E$ is a direct summand of $f^*f_*E$ by considering the trace map as usual. Thus we may assume $Y = \mathbb{A}^n_{k[[t]]}$ in the first place.
		
		$E \in \Isoc(\mathbb{A}^n_{k((t))}, \mathbb{A}^n_{k[[t]]}$ induces a $\nabla$-module over $]\mathbb{A}^n_{k}[_{\wh{\mathbb{A}}^n_{\V}} \times_K A_K^1(\eta, 1).$ By Theorem \ref{thm:eligible-regular}, there is a finite affinoid open cover $\{\Us_i\}$ of $]\mathbb{A}^n_{k}[_{\wh{\mathbb{A}}^n_{\V}} = \wh{\mathbb{A}}^{n, \ad}_{\V}$ and eligible covers $\Us'_i \to \Us_i \times A^1_K(\epsilon', 1)$ such that the pullbacks of $\E$ to $\Us'_i$ are regular. Choose an admissible blowup $\varphi : \Pf \to \wh{\mathbb{A}}^n_{\V}$ so that there exist open subsets $\Uf_i$ of $\Pf$ which induces $\Us_i.$ For the special fiber $U_i = \Spec R_i$ of $\Uf_i$, the eligible covers $\Us'_i$ corresponds to finite etale covers of $R_i((t))$. By Katz-Gabber construction, we can algebrize them to get finite etale covers of $R_i[t^{-1}]$, which induces finite etale covers $R'_i$ of $R_i \otimes_k k((t)).$ Let $U'_i = \Spec R'_i$ and choose a proper, dominant, generically finite morphism $\pi : Q \to P \otimes_k k[[t]]$ which dominates finite etale covers $U'_i \to U_i \otimes_k k((t))$. Then the composition $Q \to P \otimes_k k[[t]] \to \mathbb{A}^n_{k[[t]]}$ satisfies the assertion, since for any pullback along $g : Y'' \to Q$, the components of the closed fibers of $Y''$ factors some of the eligible covers above, hence the regularity.
		
		Now we consider the general case. By de Jong's alteration theorem \ref{thm:alteration}, we may assume in the first place that for $Z = (Y-X)_\red$, $(Y, Z)$ is a strictly semistable pair. For each component $Z_i$ of $Y_s$, choose an open subset $U_i$ of $Y$ such that $U_{i, \eta} = X \cap U_i.$ Then we have an alteration $f_i : U'_i \to U_i$ by what we have proved. Since the regularity of monodromy along closed fiber is local at the generic point of each component as in \ref{empty:regularity}, choosing an alteration $f : Y' \to Y$ which dominates each $f_i$, we win.
	\end{proof}
	
	\begin{theorem} \label{thm:abs-ssr}
		Let $X$ be a smooth $k((t))$-variety, $X \hookrightarrow Y$ an open immersion into a $k[[t]]$-variety and $E$ a $K$-valued overconvergent $F$-isocrystal over $(X, Y)$. Then after base changing $k[[t]]$ to some $k^{1/q^n}[[t^{1/q^n}]]$, there is a finite extension $k[[t]] \hookrightarrow k'[[t']]$, an alteration $f: Y' \to Y$ over $k[[t]]$, an open immersion $X' \hookrightarrow Y'$ into a projective $k'[[t']]$-variety such that for $Z' = f^{-1}(Y-X)_\red$, $(Y', Z')$ is a strictly semistable pair and $f^* E$ extends to a locally free convergent log isocrystal over $Y'^\sharp/\V'$ with nilpotent residues.
	\end{theorem}
	\begin{proof}
		By Lemma \ref{lem:ext-monodromy} and Theorem \ref{thm:main-edk}, we are reduced to prove that after alteration $E$ has unipotent monodromy along the closed fiber.
		
		By Theorem \ref{thm:regular-h} and de Jong's alteration theorem, we may assume in the first place that $(Y, Z)$ is a projective strictly semistable pair with $X = Y-Z$ and $E$ has regular monodromy along the closed fiber. Since regular connections over an open annulus with a Frobenius structure becomes unipotent after base changing $k$ to $k^{1/q^n}$ and pulling back by a tamely ramified cover as mentioned in \cite[Theorem 1.5.5]{kedlaya2022monodromy}, the exponents along each component of the closed fiber belong to $\mathbb{Z}_{(p)}$, the localization of $\mathbb{Z}$ at the prime ideal $(p)$, modulo $\mathbb{Z}$.
		
		It suffices to show the following: given a positive integer $N$ coprime to $p$, for the vertical divisors $Z_1, \dots, Z_r$, there is a finite extension $k[[t]] \hookrightarrow k'[[t']]$, an alteration $f : Y' \to Y$ over $k[[t]]$ such that for $Z' = f^{-1}(Z)_\red$, $(Y', Z')$ is a strictly semistable pair over $k'[[t']]$ and that for vertical divisors $Z'_1, \dots, Z'_{r'}$ of $Y'$ and $f^*Z_i = \sum_j c_{ij} Z'_j$, the coefficients $c_{ij}$ are divisable by $N$. Indeed if locally $Z_i$ is cut out by $t_i$ and $Z'_j$ by $t'_j$, then $t_i$ is mapped to $\prod_j t_j^{c_{ij}}$ up to multiplying units and the pullback of $E$ has exponents $0$ modulo $\mathbb{Z}$ along each $Z'_j$, which implies unipotence.
		
		For this, choose some uniformizers $\pi_i \in \O_{Y, Z_i}$ and extend the function filed $K(Y)$ to $L = K(Y)(\pi_1^{1/N}, \dots, \pi_r^{1/N}).$ Since the integral closure of $\O_{Y, Z_i}$ in $K(Y)(\pi_i^{1/N})$ is $\O_{Y, Z_i}[\pi_i^{1/N}]$, for the normalization $\tilde{Y}$ of $Y$ in $L$, since $\tilde{Y} \to Y$ is finite and dominant, it has the property mentioned above all but being a strictly semistable pair. Finally apply de Jong's alteration to $\tilde{Y}$, we win. 
	\end{proof}
	
	\section{Application to rigid cohomology} \label{sec:app}
	It is well known that the semistable reduction theorem is essential for almost all of deep results in $p$-adic cohomology and we will give a small application to the theory of rigid cohomology. Lazda-P\'{a}l defined compactly supported rigid cohomology of an overconvergent isocrystal in \cite{lazda2016rigid}, but the finiteness is only proved under the condition of isocrystal being ``compactifiable'', which asserts that the isocrystal over $X$ is the restriction of one over a compactification $\overline{X}$. This condition is closely related to having constant monodromy, as we have seen in \S\ref{sec:monodromy-and-extension} and it is a very strong condition.
	
	If one tries to prove finiteness of compactly supported rigid cohomology with coefficient by imitating the proof of classical case \cite{kedlaya2006finiteness}, one encounters a $p$-adic functional analytic problem of Hahn-Banach type, which requires extensions of a continuous $\Ed_K$-linear map from a closed subspace of an LF-space. It seems difficult to prove finiteness in this fashion. We will instead apply the semistable reduction theorem to proving finiteness in this section.
	
	\subsection{Analytic cohomology and compactly supported analytic cohomology}
	\begin{_empty}
	Let $(Y, Z)$ be a strictly semistable pair, $X = Y-Z$ and let $(Y, M_Y) \hookrightarrow (\Pf, M_\Pf)$ be a closed immersion into a $p$-adic log formal scheme log smooth over $(\Bf, M_\Bf)$. For an overconvergent log isocrystal $E$ over $(Y_\eta, Y)/\Bf$, let $\E$ be the induced log $\nabla$-module over $]Y_\eta[^{\log}_\Pf$ relative to $\Ed_K$. Choose an open neighborhood $W$ of $]Y_\eta[^{\log}_\Pf$ and a log $\nabla$-module $\E_W$ over $W/\V[[t]]_K$ which induces $\E$. Let $a_W : \, ]X[^{\log}_\Pf \hookrightarrow W$ be the closed immersion. We define the analytic cohomology and compactly supported analytic cohomology by
	$$H^*_{\an}(Y_\eta, Y, \Pf/\Ed_K; E) = H^*(]Y_\eta[^{\log}_{\Pf}, \, \E \otimes \Omega^{*, \log}_{]Y_\eta[^{\log}_\Pf/\Ed_K} ) = H^*(W, \, j^\dag_{Y_\eta} \E_W \otimes \Omega^{*, \log}_{]Y_\eta[^{\log}_\Pf/\Ed_K} )$$
	and
	$$H^*_{\an, c}(Y_\eta, Y, \Pf/\Ed_K; E) = H^*(]X[^{\log}_{\Pf}, \, \mathbb{R}a_W^! j^\dag_{Y_\eta} \E_W \otimes \Omega^{*, \log}_{W/\V[[t]]_K}) = H^*(W, \, a_{W*}\mathbb{R}a_W^! j^\dag_{Y_\eta} \E_W \otimes \Omega^{*, \log}_{W/\V[[t]]_K}).$$
	If we let $i_W : W \cap ]Z[^{\log}_{\Pf} \hookrightarrow W$ be the complement of $a_W$, then there is a distinguished triangle
	$$ a_{W*} \mathbb{R} a_W^! \to \id \to \mathbb{R}i_{W*} i_W^{-1} \to$$
	and $\mathbb{R} a_W^! : D^+(W) \to D^+(]X[)$ is characterized by this property. $a_{W*}\mathbb{R}a_W^!$ is often denoted by $\mathbb{R}\underline{\Gamma}_{]X[^{\log}_\Pf}$. The second one is independent of the choice of $W$.
	
	In general, for a strictly semistable pair $(Y, Z)$, and an overconvergent log isocrystal $E$ over $(Y_\eta, Y)/\Bf$, we choose a Zariski hypercovering $Y_{\bullet} \to Y$ and closed immersions $(Y_\bullet, M_{Y_\bullet}) \hookrightarrow (\Pf_\bullet, M_{\Pf_\bullet})$ into log formal schemes formally log smooth over $(\Bf, M_\Bf)$. We will call it a good embedding system of $Y$ henceforth. Let $\E_\bullet$ be the induced log $\nabla$-module over $]Y_{\bullet, \eta}[^{\log}_{\Pf_\bullet}$ relative to $\Ed_K$ and choose open neighborhoods $W_\bullet$ of $]Y_{\bullet, \eta}[^{\log}_{\Pf_\bullet}$ which $\E_\bullet$ extends over to $\E_{W_\bullet}$ and form a simplicial rigid analytic spaces. Then we define
	$$H_{\an}^*(Y_\eta, Y/\Ed_K; \E) = H^*(]Y_{\bullet, \eta}[^{\log}_{\Pf_\bullet}, \E_\bullet \otimes \Omega^{*, \log}_{]Y_{\bullet, \eta}[^{\log}_{\Pf_\bullet}/\Ed_K})$$
	and
	$$H^*_{\an, c}(Y_\eta, Y/\Ed_K; E) = H^*(]X_\bullet[^{\log}_{\Pf_\bullet}, \, \mathbb{R}a_{W_\bullet}^! j^\dag_{Y_{\bullet, \eta}} \E_{W_\bullet} \otimes \Omega^{*, \log}_{W_\bullet/\V[[t]]_K}).$$
	The second one is again independent of the choice of $W_\bullet$.
	
	Remark that these are independence of the choice of $\Pf$ or $\Pf_\bullet$. For this, we have to prove log Poincar\'{e} lemma.
	\end{_empty}
	
	\begin{lemma} \label{lem:log-strong-fibration}
		Let
		\[
		\xymatrix{
			(Y, M_Y) \ar@{^(->}[r]^{i'} \ar@{^(->}[rd]^{i} & (\Pf', \M_{\Pf'}) \ar[d]^{f} \\
			& (\Pf, M_\Pf)
		}
		\]
		be a commutative diagrams with $i$, $i'$ closed immersions into $p$-adic formal log schemes formally log smooth over $(\Bf, M_\Bf)$ and $f$ formally log smooth of relative dimension $r$. Then Zariski locally over $\Pf$, $]Y[^{\log}_{\Pf'} \simeq ]Y[^{\log}_{\Pf} \times_K A^r_K[0, 1)$ whose projection $p : ]Y[^{\log}_{\Pf'} \to ]Y[^{\log}_{\Pf}$ is the morphism induced by $f$.
	\end{lemma}
	\begin{proof}
		Similar to \cite[Lemma 2.3.1]{shiho2007relativeI}.
	\end{proof}
	
	\begin{proposition}[Log Poincar\'{e} lemma]
		In the situation of Lemma \ref{lem:log-strong-fibration}, let $W$ be an open neighborhood of $]Y_\eta[^{\log}_{\Yf}$ and $\E_W$ a log $\nabla$-module over $W/\V[[t]]_K$. Then
		$$\mathbb{R}p_*(j^\dag_{Y_\eta} p^*\E_W \otimes \Omega^{*, \log}_{W \times A^r[0, 1)/\V[[t]]_K}) \simeq j^\dag_{Y_\eta} \E_W \otimes \Omega^{*, \log}_{W/\V[[t]]_K}$$
		and
		$$\mathbb{R}p_* (\mathbb{R}\underline{\Gamma}_{]X[^{\log}_{\Pf'}} j^\dag_{Y_\eta} p^*\E_W \otimes \Omega^{*, \log}_{W \times A^r[0, 1)/\V[[t]]_K}) \simeq \mathbb{R}\underline{\Gamma}_{]X[^{\log}_{\Pf}} j^\dag_{Y_\eta} \E_W \otimes \Omega^{*, \log}_{W/\V[[t]]_K}.$$
	\end{proposition}
	\begin{proof}
		This proposition is just the log version of \cite[Proposition 2.4.8]{lazda2016rigid} and \cite[Lemma 4.3.7]{lazda2016rigid} and the proof is similar.
	\end{proof}
	
	\begin{corollary}
		Let $(Y_\bullet, M_{Y_\bullet}) \hookrightarrow (\Pf_\bullet, M_{\Pf_\bullet})$ and $(Y_\bullet, M_{Y_\bullet}) \hookrightarrow (\Pf'_\bullet, M_{\Pf'_\bullet})$ be two good embedding systems. Then for any overconvergent log isocrystal $E$ over $(Y_\eta, Y)/\Bf$, we have an isomorphism
		$$H_{\an}^*(Y_{\bullet, \eta}, Y_\bullet, \Pf_\bullet/\Ed_K; E) \simeq H_{\an}^*(Y_{\bullet, \eta}, Y_\bullet, \Pf'_\bullet/\Ed_K; E)$$
		and
		$$H_{\an, c}^*(Y_{\bullet, \eta}, Y_\bullet, \Pf_\bullet/\Ed_K; E) \simeq H_{\an, c}^*(Y_{\bullet, \eta}, Y_\bullet, \Pf'_\bullet/\Ed_K; E).$$
	\end{corollary}
	
	Thus it makes sense to write $H^*_{\an}(Y_\eta, Y/\Ed_K; E)$ or $H^*_{\an, c}(Y_\eta, Y/\Ed_K; E)$.
	
	We will prove that if $Y$ is projective and $E$ is locally free with nilpotent residues, then the analytic cohomology is isomorphic to the rigid cohomology and the compactly supported cohomology is isomorphic to compactly supported rigid cohomology and is of finite dimension.
	
	Before proceeding, let us note that we can make similar definition for the convergent log isocrystals $\wh{E}$ over $Y_\eta/\E_K.$ We will prove finiteness by showing the comparison of $\Ed_K$-valued ones and $\E_K$-valued ones.
	
	\subsection{Finiteness of analytic and compactly supported analytic cohomology}
	Suppose that $(Y, Z)$ is a projective strictly semistable pair, $X = Y-Z$ and $E$ is a locally free overconvergent log isocrystal over $(Y_\eta, Y)$ with nilpotent residues. Recall that we have the canonical restriction functor $j^\dag : I^{\log}_{\oc}(Y_\eta, Y/\Bf) \to \Isoc(X, Y/\Bf).$
	
	\begin{theorem} \label{thm:an-rig-comparison}
		$H^*_{\an}(Y_\eta, Y/\Ed_K; E) \simeq H^*_{\rig}(X/\Ed_K; j^\dag E).$
	\end{theorem}
	\begin{proof}
		Choose a good embedding system $(Y_\bullet, M_{Y_\bullet}) \hookrightarrow (\Yf_\bullet, M_{\Yf_\bullet})$ such that each $Y_\bullet$ and $\Yf_\bullet$ satisfies Hypothesis \ref{hypo:ss-pair}. If we let $\E_\bullet$ to be the induced log-$\nabla$ module over $]Y_{\bullet, \eta}[_{\Yf_\bullet}$, we have
		$$H^*_{\an}(Y_\eta, Y/\Ed_K; E) \simeq H^*(]Y_{\bullet, \eta}[_{\Yf_\bullet}, \E_\bullet \otimes \Omega^{*, \log}_{]Y_{\bullet, \eta}[_{\Yf_\bullet}/\Ed_K}).$$
		
		On the other hand, choose a closed immersion $Y \hookrightarrow \Pf$ into a smooth projective formal $\V[[t]]$-scheme (which is possible since $Y$ is projective) and choose a Zariski hypercover $\Pf_\bullet \to \Pf$ which induces $Y_\bullet \to Y.$ If we let $\E'$ be the $\nabla$-module over $]X[_{\Pf}$ induced by $j^\dag E$ and let $\E'_\bullet$ be the induced $\nabla$-module over $]X_\bullet[_{\Pf_\bullet}$, then we have 
		$$H^*_{\rig}(X/\Ed_K; j^\dag E) \simeq H^*(]X[_{\Pf}, \E' \otimes \Omega^{*}_{]X[_\Pf/\Ed_K}) \simeq H^*(]X_\bullet[_{\Pf_\bullet}, \E'_\bullet \otimes \Omega^{*}_{]X_\bullet[_{\Pf_\bullet}/\Ed_K}).$$
		
		Since $\Yf$ and $\Pf$ are both smooth around $X$, we have $H^*(]X_\bullet[_{\Pf_\bullet}, \E'_\bullet \otimes \Omega^{*}_{]X_\bullet[_{\Pf_\bullet}/\Ed_K}) \simeq H^*(]X_{\bullet}[_{\Yf_\bullet}, \E_\bullet \otimes \Omega^{*}_{]X_{\bullet}[_{\Yf_\bullet}/\Ed_K})$ and by considering the $E_1$-terms of the associated spectral sequences, we are reduced to prove that if $(Y, Z)$ and $\Yf$ satisfies Hypothesis \ref{hypo:ss-pair}, there is a canonical isomorphism $H^*(]Y_\eta[_{\Yf}, \E \otimes \Omega^{*, \log}_{]Y_\eta[_{\Yf}/\Ed_K}) \simeq H^*(]X[_{\Yf}, \E \otimes \Omega^{*}_{]X[_{\Yf}/\Ed_K}).$ The proof is similar to that of \cite{shiho2002crystallineII}.
		
		Choose an open neighborhood $W = \{ x \in \Yf_K : |t_1|_x, \dots, |t_r|_x \geq \epsilon \}$ over which a locally free log $\nabla$-module $\E_W$ with nilpotent residues giving rise to $\E$ is defined. Let $V^\eta = \{ x \in W : |t_{r+1}|_{[x]}, \dots, |t_m|_{[x]} > \eta \}$ for $\eta \in (\epsilon, 1) \cap \Gamma^*$. Since we have $H^*(]Y_\eta[_{\Yf}, \E \otimes \Omega^{*, \log}_{]Y_\eta[_{\Yf}/\Ed_K}) \simeq H^*(W, j^\dag_{Y_\eta} \E_W \otimes \Omega^{*, \log}_{W/\V[[t]]_K})$ and $H^*(]X[_{\Yf}, \E \otimes \Omega^{*}_{]X[_{\Yf}/\Ed_K}) \simeq \varinjlim_{\eta \to 1} H^*(V^\eta, j^\dag_{Y_\eta} \E_{V^\eta} \otimes \Omega^{*}_{V^\eta/\V[[t]]_K})$, it suffices to prove that there is a canonical isomorphism $H^*(]Y_\eta[_{\Yf}, \E \otimes \Omega^{*, \log}_{]Y_\eta[_{\Yf}/\Ed_K}) \simeq H^*(V^\eta, j^\dag_{Y_\eta} \E_{V^\eta} \otimes \Omega^{*}_{V^\eta/\V[[t]]_K})$.
		
		Let $k = m-r$. For $J \subseteq \{1, \dots, k\}$, put
		$$W^{\eta}_J = \{ x \in W : |t_{r+j}|_{[x]} < 1 \text{ for } j \in J \text{ and } |t_{r+j}|_{[x]} > \eta \text{ for } j \in \{1, \dots, k\}-J\},$$
		$$U^{\eta}_J = \{ x \in W : |t_{r+j}|_{[x]} = 0 \text{ for } j \in J \text{ and } |t_{r+j}|_{[x]} > \eta \text{ for } j \in \{1, \dots, k\}-J\},$$
		and let $V^\eta_J = W^\eta_J \cap V^\eta.$ There is an isomorphism $W^\eta_J \simeq U^\eta_J \times_K A^{|J|}_K[0, 1)$ whose projection to the polydisc is given by $t_{r+j}$ for $j \in J$, which maps $V^\eta_J$ precisely to $U^\eta_J \times_K A^{|J|}_K(\eta, 1)$.
		Given $J_0, \dots, J_p \subseteq \{ 1, \dots, k \}, $ we have
		$$W^{\eta}_{J_0} \cap \cdots \cap W^{\eta}_{J_p} = \{ x \in W : |t_{r+j}|_{[x]} < 1 \text{ for } j \in J_0 \cup \cdots \cup J_p \text{ and } |t_{r+j}|_{[x]} > \eta \text{ for } j \in \{1, \dots, k\}-(J_0 \cap \cdots \cap J_p)\}.$$
		Put $U'^{\eta}_{J_0, \dots, J_p} = \{ x \in U^\eta_{J_0 \cup \cdots \cup J_p} : |t_{r+j}|_{[x]} > \eta \text{ for } j \in \{ 1, \dots, k \} - (J_0 \cap \cdots \cap J_p)\}.$ Then we have a commutative diagram
		\[
		\xymatrix{
			 & W^\eta_{J_0 \cup \cdots \cup J_p} \ar[rr]^{\sim} &  & U^\eta_{J_0 \cup \dots \cup J_p} \times A^{|J_0 \cup \cdots \cup J_p|}[0, 1) \\
			W^\eta_{J_0} \cap \cdots \cap W^\eta_{J_p} \ar@{^(->}[ru] \ar[rr]^{\sim}&  & U'^\eta_{J_0, \dots, J_p} \times A^{|J_0 \cup \cdots \cup J_p|}[0, 1) \ar@{^(->}[ru] &  \\
			 & V^\eta_{J_0 \cup \cdots \cup J_p} \ar@{^(-->}[uu] \ar@{-->}[rr]^{\sim}&  & U^\eta_{J_0 \cup \dots \cup J_p} \times A^{|J_0 \cup \cdots \cup J_p|}(\eta, 1) \ar@{^(->}[uu]\\
			V^\eta_{J_0} \cap \cdots \cap V^\eta_{J_p} \ar@{^(->}[uu] \ar@{^(->}[ru] \ar[rr]^{\sim} &  & U'^\eta_{J_0, \dots, J_p} \times A^{|J_0 \cup \cdots \cup J_p|}(\eta, 1) \ar@{^(->}[uu] \ar@{^(->}[ru] &  \\
		}
		\]
		
		By the spectral sequences
		$$E_1^{p, q} = \bigoplus_{J_0, \dots, J_p} H^q(W^\eta_{J_0} \cap \cdots \cap W^\eta_{J_p}, \Box) \Rightarrow H^{p+q}(W, \Box)$$
		and
		$$'E_1^{p, q} = \bigoplus_{J_0, \dots, J_p} H^q(V^\eta_{J_0} \cap \cdots \cap V^\eta_{J_p}, \Box) \Rightarrow H^{p+q}(V^\eta, \Box)$$
		associated to open covers $W = \cup_J W^\eta_J$ and $V^\eta = \cup_J V^\eta_J$, it suffices to compare $E_1$-terms.
		
		Let $U = U'^\eta_{J_0, \dots, J_p}, \, s = |J_0 \cup \cdots \cup J_p|$. The case where $s = 0$ is obvious since $W^\eta_\emptyset = V^\eta_\emptyset$. We will assume $s > 0$ and prove that
		$$H^*(U \times A^s[0, 1), j^\dag_{Y_\eta} \E \otimes \Omega^{*, \log}_{U \times A^s[0, 1)/\V[[t]]_K}) \simeq H^*(U \times A^s(\eta, 1), j^\dag_{Y_\eta} \E \otimes \Omega^{*}_{U \times A^s(\eta, 1)/\V[[t]]_K})$$
		is an isomorphism. For this, consider the Katz-Oda spectral sequences
		\begin{align*}
			E_2^{p, q} & = H^p(\Gamma(U, j^\dag_{Y_\eta}\Omega^{*}_{U/\V[[t]]_K}) \otimes_{j^\dag_{Y_\eta} \O(U)} H^q(U \times A^s[0, 1), j^\dag_{Y_\eta} \E \otimes \Omega^{*, \log}_{U \times A^s[0, 1) / U})) \\
			& \Rightarrow H^{p+q}(U \times A^s[0, 1), j^\dag_{Y_\eta} \E \otimes \Omega^{*, \log}_{U \times A^s[0, 1)/\V[[t]]_K})
		\end{align*}
		and
		\begin{align*}
			E_2^{p, q} & =  H^p(\Gamma(U, j^\dag_{Y_\eta}\Omega^{*}_{U/\V[[t]]_K}) \otimes_{j^\dag_{Y_\eta} \O(U)} H^q(U \times A^s(\eta, 1), j^\dag_{Y_\eta} \E \otimes \Omega^{*}_{U \times A^s(\eta, 1) / U})) \\
			& \Rightarrow H^{p+q}(U \times A^s(\eta, 1), j^\dag_{Y_\eta} \E \otimes \Omega^{*}_{U \times A^s(\eta, 1)/\V[[t]]_K})
		\end{align*}
		Note that higher cohomologies of coherent $j^\dag_{Y_\eta} \O$-module over $U$, $U \times A^s[0, 1)$ and $U \times A^s[0, 1)$ vanish since $U$ is ``quasi-Stein'' in a suitable sense. See Propositoin \ref{prop:vanish-jdag} for details. We are therefore reduced to prove that $H^q(U \times A^s[0, 1), j^\dag_{Y_\eta} \E \otimes \Omega^{*, \log}_{U \times A^s[0, 1) / U})$ and $H^q(U \times A^s(\eta, 1), j^\dag_{Y_\eta} \E \otimes \Omega^{*}_{U \times A^s(\eta, 1) / U})$ are canonically isomorphic.
		Since $\E$ is locally free with nilpotent residues, by Proposition \ref{prop:unip-ext}, $j^\dag\E$ over $U \times A^s[0, 1)$ is unipotent. Therefore by devissage, we are reduced to the case where $\E$ is constant. Then using local freeness, we are further reduced to the case $\E = \O$. Then easy calculation shows that $H^q(U \times A^s[0, 1), j^\dag_{Y_\eta} \Omega^{*, \log}_{U \times A^s[0, 1) / U})$ and $H^q(U \times A^s(\eta, 1), j^\dag_{Y_\eta} \Omega^{*, \log}_{U \times A^s(\eta, 1) / U})$ are isomorphic.
	\end{proof}
	
	\begin{lemma}
		$H^*_{\an, c}(Y_\eta, Y/\Ed_K; E) \simeq H^*_{\rig, c}(X/\Ed_K; j^\dag E)$ and $H^*_{\an, c}(Y_\eta/\E_K; \wh{E}) \simeq H^*_{\rig, c}(X/\E_K; j^\dag \wh{E})$
	\end{lemma}
	\begin{proof}
		We consider the first one. By the same argument of Theorem \ref{thm:an-rig-comparison}, we are reduced to show that, using the same symbols, there is a canonical isomorphism $H^*(]X[_\Yf, \mathbb{R}a_W^! j^\dag_{Y_\eta} \E_W \otimes \Omega^{*, \log}_{W/\V[[t]]_K}) \simeq H^*(]X[_\Yf, \mathbb{R}a_V^! j^\dag_{Y_\eta} \E_V \otimes \Omega^{*, \log}_{V/\V[[t]]_K})$
		where $a_W : ]X[_\Yf \hookrightarrow W$ and $a_V : ]X[_\Yf \hookrightarrow V$ are closed immersions. This is obvious since the complexes are isomorphic.
	\end{proof}
	
	\begin{theorem} \label{thm:anc-comparison}
		There is a canonical isomorphism $H^*_{\an, c}(Y_\eta, Y/\Ed_K; E) \otimes_{\Ed_K} \E_K \simeq H^*_{\an, c}(Y_\eta/\E_K; \wh{E}).$ In particular, $H^*_{\an, c}(Y_\eta, Y/\Ed_K; E)$ is of finite dimension over $\Ed_K$.
	\end{theorem}
	\begin{proof}
		Choose a good embedding system $(Y_{\bullet}, M_{Y_{\bullet}}) \hookrightarrow (\Yf_{\bullet}, M_{\Yf_\bullet})$ and open neighborhoods $W_{\bullet}$ of $]Y_{\bullet}[_{\Yf_{\bullet}}$ over which $E$ is defined such that $W_{\bullet}$ form a simplicial rigid analytic variety. Then we have
		$$R\Gamma_{\an,c}(Y_{\bullet, \eta}, Y/\Ed_K, E) = R\Gamma(W_{\bullet}, R\Gamma_{]X_\bullet[_{\Yf_\bullet}} j^\dag_{Y_{\bullet, \eta}}\E_{W_\bullet} \otimes \Omega^{*, \log}_{W_\bullet/\V[[t]]_K}),$$
		$$R\Gamma_{\an,c}(Y_{\eta}/\E_K, \wh{E}) = R\Gamma(V_{\bullet}, R\Gamma_{]X_\bullet[_{\Yf_{\bullet, \eta}}} \E_{V_\bullet} \otimes \Omega^{*, \log}_{V_\bullet/\E_K}),$$
		where $V_\bullet = ]Y_{\bullet, \eta}[_{\Yf_{\bullet, \eta}}$.
		
		There is a canonical commutative diagram with columns distinguished triangles:
		\[
		\xymatrix{
			R\Gamma_{\an, c}(W_{\bullet}, R\underline{\Gamma}_{]X_\bullet[_{\Yf_\bullet}} j_{Y_{\bullet, \eta}}^{\dag}\E_{W_\bullet} \otimes \Omega_{W/\V[[t]]_K}^{*, \log}) \otimes_{\Ed_K} \E_K \ar[r] \ar[d] & R\Gamma_{\an, c}(V_{\bullet}, R\underline{\Gamma}_{]X_\bullet[_{\Yf_{\bullet, \eta}}} \E_{V_{\bullet}} \otimes \Omega_{V_\bullet/\E_K}^{\bullet, \log}) \ar[d] \\
			R\Gamma(W_{\bullet}, j_{Y_{\bullet, \eta}}^{\dag}\E_{W_\bullet} \otimes \Omega_{W/\V[[t]]_K}^{*, \log}) \otimes_{\Ed_K} \E_K  \ar[r] \ar[d] & R\Gamma(V_{\bullet}, \E_{V_{\bullet}} \otimes \Omega_{V_\bullet/\E_K}^{\bullet, \log}) \ar[d] \\
			R\Gamma(W_{\bullet} \cap ]Z_\bullet[_{\Yf_\bullet}, j_{Y_{\bullet, \eta}}^{\dag}\E_{W_\bullet} \otimes \Omega_{W/\V[[t]]_K}^{*, \log}) \otimes_{\Ed_K} \E_K \ar[r] \ar[d] & R\Gamma(]Z_{\bullet, \eta}[_{\Yf_{\bullet, \eta}}, \E_{V_{\bullet}} \otimes \Omega_{V_\bullet/\E_K}^{\bullet, \log}) \ar[d] \\
			& 
		}
		\]
		
		The middle complexes are, after taking $\varprojlim_{W_\bullet}$ for $W_\bullet$ running over open neighborhoods of $]Y_{\bullet, \eta}[_{\Yf_{\bullet}}$, quasi-isomorphic to $R\Gamma_{\an}(Y_\eta, Y/\Ed_K; E) \otimes_{\Ed_K} \E_K$ and $R\Gamma_{\an}(Y_\eta/\E_K; \wh{E})$, hence the middle horizontal map is a quasi-isomorphism by the comparison theorem of analytic cohomologies and rigid cohomologies, namely Theorem \ref{thm:an-rig-comparison} and \cite[Theorem 2.4.4]{shiho2007relativeI} and the comparison of $\Ed_K$ and $\E_K$-valued rigid cohomologies \cite[Theorem 3.74]{lazda2016rigid}.
		
		Thus we are reduced to show that the bottom horizontal morphism is a quasi-isomorphism after taking $\varprojlim_{W_\bullet}$, which we will omit to lighten notation. Let $Z_1, \dots, Z_k$ be the irreducible components of $Z_f$, that is, the irreducible components of $Z$ which dominates $\Spec k[[t]]$. Then we have an open cover $]Y_{\bullet, \eta}[_{\Yf_\bullet} \cap ]Z_\bullet[_{\Yf_{\bullet}} = (]Y_{\bullet, \eta}[_{\Yf_\bullet} \cap ]Z_{1, \bullet}[_{\Yf_{\bullet}}) \cup \dots \cup (]Y_{\bullet, \eta}[_{\Yf_\bullet} \cap ]Z_{m, \bullet}[_{\Yf_{\bullet}})$ and $V_{\bullet} \cap ]Z[_{\Yf_{\bullet}} = (V_{\bullet} \cap ]Z_{1, \bullet, \eta}[_{\Yf_{\bullet, \eta}}) \cup \dots \cup (V_{\bullet} \cap ]Z_{m, \bullet, \eta}[_{\Yf_{\bullet, \eta}})$, therefore the associated spectral sequences 
		\begin{align*}
			E_1^{p,q} & = \bigoplus_{i_0 < \dots < i_p} H^q(W_{\bullet} \cap ]Z_{i_0 \dots i_p, \bullet}[_{\Yf_{\bullet}}, j_{Y_{\bullet, \eta}}^{\dag}\E_{W_{\bullet}} \otimes \Omega_{W_{\bullet}/\V[[t]]_K}^{*, \log}) \\
			& \Rightarrow H^{p+q}(W_{\bullet} \cap ]Z_\bullet[_{\Yf_{\bullet}}, j_{Y_{\bullet, \eta}}^{\dag}\E_{W_{\bullet}} \otimes \Omega_{W_{\bullet}/\V[[t]]_K}^{*, \log})
		\end{align*}
		\begin{align*}
			\wh{E}_1^{p,q} & = \bigoplus_{i_0 < \dots < i_p} H^q(V_{\bullet} \cap ]Z_{i_0 \dots i_p, \bullet, \eta}[_{\Yf_{\bullet, \eta}}, \E_{V_{\bullet}} \otimes \Omega_{V_{\bullet}/\V[[t]]_K}^{*, \log}) \\
			& \Rightarrow H^{p+q}(]Z_{\bullet, \eta}[_{\Yf_{\bullet, \eta}}, \E_{V_{\bullet}} \otimes \Omega_{V_{\bullet}/\V[[t]]_K}^{*, \log})
		\end{align*}
		with $Z_{i_0 \dots i_p} = Z_{i_0} \cap \cdots \cap Z_{i_p}$. Thus we are reduced to prove that $E_1^{p, q} \otimes_{\Ed_K} \E_K \to \wh{E}_1^{p, q}.$
		
		Fix $1 \leq i_0 < \cdots < i_p \leq m.$ Then $]Y_{\bullet, \eta}[_{\Yf_\bullet} \cap ]Z_{i_0 \dots i_p, \bullet}[_{\Yf_{\bullet}} \, \simeq  \, U_\bullet \times_{\Ed_K} A^{p+1}_{\Ed_K}[0, 1)$ and $V_\bullet \cap ]Z_{i_0 \dots i_p, \bullet}[_{\Yf_{\bullet}} \, \simeq  \, \wh{U}_{\bullet} \times_{\E_K} A^{p+1}_{\E_K}[0, 1)$ for $U_\bullet$ a partial dagger space. Consider Katz-Oda spectral sequences
		\begin{align*}
			'E_2^{p,q} & = H^p(U_{\bullet}, \Omega_{U_{\bullet}/\Ed_K}^{*, \log} \otimes_{\O(U_{\bullet})} H^q(U_{\bullet} \times A_{\Ed_K}^{p+1}[0, 1), \E_\bullet \otimes \Omega_{U_{\bullet}\times A^{p+1}_{\Ed_K}[0, 1)/U_\bullet}^{*, \log})) \\
			& \Rightarrow H^{p+q}(]Y_{\bullet, \eta}[_{\Yf_\bullet} \cap ]Z_{i_0 \dots i_p, \bullet}[_{\Yf_{\bullet}}, \E_\bullet \otimes \Omega^{\bullet, \log}_{]Y_{\bullet, \eta}[_{\Yf_\bullet}/\Ed_K})
		\end{align*}
		and
		\begin{align*}
			'\wh{E}_2^{p,q} & = H^p(\wh{U}_{\bullet}, \Omega_{\wh{U}_{\bullet}/\E_K}^{*, \log} \otimes_{\O(\wh{U}_{\bullet})} H^q(\wh{U}_{\bullet} \times A_{\E_K}^{p+1}[0, 1), \E_\bullet \otimes \Omega_{\wh{U}_{\bullet}\times A^{p+1}_{\E_K}[0, 1)/\wh{U}_\bullet}^{*, \log})) \\
			& \Rightarrow H^{p+q}(]Y_{\bullet, \eta}[_{\Yf_{\bullet, \eta}} \cap ]Z_{i_0 \dots i_p, \bullet}[_{\Yf_{\bullet, \eta}}, \E_\bullet \otimes \Omega^{\bullet, \log}_{]Y_{\bullet, \eta}[_{\Yf_{\bullet, \eta}}/\Ed_K})
		\end{align*}
		Therefore it suffices to prove that $'E_2^{p, q} \otimes_{\Ed_K} \E_K \simeq {'\wh{E}}_2^{p, q}$ is an isomorphism.
		
		Since $\E$ is locally free with nilpotent residues, by Proposition \ref{prop:unip-ext}, $\E_\bullet$ over $U_\bullet \times A^s[0, 1)$ is unipotent. Again by devissage, we get that $H^q(U_{\bullet} \times A_{\Ed_K}^{p+1}[0, 1), \E_\bullet \otimes \Omega_{U_{\bullet}\times A^{p+1}_{\Ed_K}[0, 1)/U_\bullet}^{*, \log})$ is a locally free  $\O(U_\bullet)$-module of finite rank, equipped with an overconvergent log connection with nilpotent residues and similarly for the $\E_K$-coefficient. These log connections glue to give a locally free overconvergent log isocrystals $F_q$ over $(Z_{i_0 \dots i_p, \eta}, Z_{i_0 \dots i_p})/\Bf$ and $\wh{F}_q$ over $Z_{i_0 \dots i_p, \eta}/\E_K$, whose log structures are induced by $Z_j$'s and $Z_{j, \eta}$'s for $j \not\in \{ i_0, \dots, i_p \}$. Since $Z_{i_0, \dots, i_p}$ is projective, we get $'E_2^{p, q} = H^p_{\an}(Z_{i_0 \dots i_p, \eta}, Z_{i_0 \dots i_p}/\Ed_K; F_q)$ and $'\wh{E}_2^{p, q} = H^p_{\an}(Z_{i_0 \dots i_p, \eta}/\E_K; \wh{F}_q)$. As discussed above, we get $'E_2^{p, q} \otimes_{\Ed_K} \xrightarrow{\sim} {'\wh{E}}_2^{p, q}$ and this finishes the proof.
	\end{proof}
	
	\subsection{Finiteness of compactly supported rigid cohomology}
	Now we can prove finiteness of compactly supported rigid cohomology of arbitrary overconvergent $F$-isocrystal over a separated $k((t))$-scheme of finite type.
	
	\begin{theorem}
		Suppose that $k$ is perfect. Let $X$ be a separated $k((t))$-scheme of finite type, $E$ an $\Ed_K$-valued overconvergent $F$-isocrystal. Then there is a canonical isomorphism $$H^*_{\rig, c}(X/\Ed_K; E) \otimes_{\Ed_K} \E_K \simeq H^*_{\rig, c}(X/\E_K; \wh{E}).$$
		In particular $H^*_{\rig, c}(X/\Ed_K; E)$ is of finite dimension over $\Ed_K$.
	\end{theorem}
	\begin{proof}
		Suppose by induction that the statement holds for $k((t))$-schemes of dimension $< \dim X$. We can find an nonempty open subset $U$ of $X$ and some $n \geq 1$ such that $U^{(n)} = (U \otimes_{k((t))} k((t^{1/p^n})))_\red$ is smooth over $k((t^{1/p^n}))$ (when $X$ is nonempty). Apply semistable reduction theorem to the pullback $E^{(n)}$ of $E$ to $U^{(n)}$, after replacing $n$ by a larger number if needed, we have a finite extension $k[[t^{1/p^n}]] \hookrightarrow k'[[t']],$ an alteration $g : U' \to U^{(n)}$ over $k((t^{1/p^n}))$ and an open immersion $U' \hookrightarrow Y'$ into a proper $k'[[t']]$-variety such that $(Y', Y'-U')$ is a strictly semistable pair and $g^*E^{(n)}$ extends to a locally free overconvergent log isocrystal over $(Y'_{\eta'}, Y')/\Bf'$ with nilpotent residues. By further applying de Jong's alteration theorem, we may assume that $Y'$ is projective over $k'[[t']]$. Choose a number $n_1 \geq n$ such that $k'((t'))$ is finite separable over $k((t^{1/p^{n_1}})).$ Then we have a nonempty open subset $V \subseteq U$ such that the induced map $g: g^{-1}(V) \to V^{(n_1)}$ is finite etale. Let $Z = U' - g^{-1}(V).$ By the long exact sequences
		$$ \to H^q_{\rig, c}(g^{-1}(V)/\E'^{\dag}_{K'}; g^*E^{(n)}) \otimes \E'_{K'} \to H^q_{\rig, c}(U'/\E'^{\dag}_{K'}; g^*E^{(n)}) \otimes \E'_{K'} \to H^q_{\rig, c}(Z'/\E'^{\dag}_{K'}; g^*E^{(n)}) \otimes \E'_{K'} \to $$
		and
		$$ \to H^q_{\rig, c}(g^{-1}(V)/\E'_{K'}; g_{\eta'}^*\wh{E}^{(n)}) \to H^q_{\rig, c}(U'/\E'_{K'}; g_{\eta'}^*\wh{E}^{(n)}) \to H^q_{\rig, c}(Z'/\E'_{K'}; g_{\eta'}^*\wh{E}^{(n)}) \to $$
		by Theorem \ref{thm:anc-comparison} and induction hypothesis, the middle and right terms are isomorphic. By five lemma, the left term is also isomorphic after base extension. Since $g: g^{-1}(V) \to V^{(n_1)}$ is finite etale, if we let $V' = V^{(n_1)} \otimes_{k((t^{1/p^{n_1}}))} k'((t')),$ then the induced morphism $g': g^{-1}(V) \to V'$ is also finite etale. Using finite etale pushforward, $H^q_{\rig, c}(g^{-1}(V)/\E'^{\dag}_{K'}; g^*E^{(n)}) \simeq H^q_{\rig, c}(V' /\E'^{\dag}_{K'}; g'_*g'^*E')$ where $E'$ is the pullback of $E$ to $V'$. Since the adjoint map $E' \rightarrow g'_*g'^*E'$ is split injective and is compatible with $\wh{E}' \rightarrow g'_*g'^* \wh{E}'$, the canonical morphism
		$H^q_{\rig, c}(V'/\E'^{\dag}_{K'}; E') \otimes_{\E'^{\dag}_{K'}} \E'_{K'} \to H^q_{\rig, c}(V'/\E_{K'}; \wh{E}')$ is an isomorphism. Now we have a commutative diagram
		\[
		\xymatrix{
			H^q_{\rig, c}(V/\Ed_K; E) \otimes_{\Ed_K} \E'_{K'} \ar[r]^(0.475){\sim} \ar[d] & H^q_{\rig, c}(V'/\E'^{\dag}_{K'}; E') \otimes_{\E'^{\dag}_{K'}} \E'_{K'} \ar[d]^{\wr} \\
			H^q_{\rig, c}(V/\E_K; \wh{E}) \otimes_{\E_K} \E'_{K'} \ar[r]^(0.55){\sim} & H^q_{\rig, c}(V'/\E'_{K'}; \wh{E}').
		}
		\]
		The horizontal maps are isomorphisms by the similar proof of \cite[Proposition 1.8]{berthelot1997finitude} (see also the appendix). Therefore the left vertical map is an isomorphism; hence $H^q_{\rig, c}(V/\Ed_K; E) \otimes_{\Ed_K} \E_K \to H^q_{\rig, c}(V/\E_K; \wh{E})$ is an isomorphism.
		
		We have shown that there is an nonempty open subset $V$ such that $H^q_{\rig, c}(V/\Ed_K; E) \otimes_{\Ed_K} \E_K \to H^q_{\rig, c}(V/\E_K; \wh{E})$ is an isomorphism. We can now appeal to noetherian induction on $X$ and using the long exact sequence to finish the proof.
	\end{proof}
	
	\begin{corollary}
		For a smooth $k((t))$-variety $X$ and an $\Ed_K$-valued overconvergent $F$-isocrystal over $X$, the Poincar\'{e} pairing
		$$H^{2d-i}_{\rig, c}(X/\Ed_K; E) \otimes_{\Ed_K} H^{i}_{\rig}(X/\Ed_K; E^\vee) \to \Ed_K$$
		is perfect.
	\end{corollary}
	
	\appendix
	\section{Vanishing of higher cohomology of coherent sheaves over quasi-Stein rigid analytic spaces}
	\setcounter{equation}{0}
	In this section we prove the following theorem, which was used in Theorem \ref{thm:an-rig-comparison}, stating that higher cohomology of coherent $j^\dag \O$-modules vanishes for adic spaces of quasi-Stein type in some suitable sense. The proof is an analogue of that of \cite[\S2.5]{kedlaya2016relative}.
	
	\begin{appdefinition}[{\cite[Definition 2.6.2]{kedlaya2016relative}}]
		An adic space $X$ is quasi-Stein if it is the union of ascending chain $U_0 \subseteq U_1 \subseteq \cdots$ of at most countable length such that the restriction map $\O(U_{i+1}) \to \O(U_i)$ has dense image.
	\end{appdefinition}
	
	Recall that by our assumption, adic spaces are adic and locally of finite type over $\V[[t]]_K$, in particular $p$ is always a topologically nilpotent unit and rings are strongly noetherian (hence sheafy) Banach rings over $K$. The consequence is that coherence implies pesudocoherence and the statements are slightly easier.
	
	\begin{apptheorem}[{\cite[Theorem 2.6.5]{kedlaya2016relative}}]
		Let $X$ be a quasi-Stein space and $\F$ a coherent $\O_X$-module. Then $H^j(X, \F) = 0$ for all $j > 0$.
	\end{apptheorem}
	
	The point of this theorem is \cite[Lemma 2.6.3.(b)]{kedlaya2016relative}, the vanishing of $R^1 \varprojlim_i M_i$, using the symbol of {\it loc. cit.} Thus by the similar argument, we are reduced to prove the following lemma:
	
	\begin{applemma}
		Let $X = \cup_i U_i$ be a quasi-Stein space, $f \in \O(X)$, $U_i = \Spa A_i$ and $M_i$ a finitely generated $A_i\langle f^{-1} \rangle^\dag$-module. Suppose that maps $\varphi_i : M_{i+1} \otimes_{A_{i+1}\langle f^{-1} \rangle^\dag} A_i\langle f^{-1} \rangle^\dag \to M_i$ are surjective for all $i$. Then $R^1 \varinjlim_i M_i = 0$.
	\end{applemma}
	\begin{proof}
		For $\eta \in (1, \infty) \cap \Gamma^*$, put $A_i^\eta = A_i \langle \eta_i / f \rangle$. For each $i = 0, 1, \dots$, successively choose $\cdots > \eta_1 > \eta_0 > 1$ and a finitely generated $A_i^{\eta_i} $-module $M_i^{\eta_i}$ such that $M_i = M_i^{\eta_i} \otimes_{A_i^{\eta_i}} A_i\langle f^{-1} \rangle^\dag$ and that $\varphi_i$ lifts to a map $M_{i+1}^{\eta_{i+1}} \otimes_{A_{i+1}^{\eta_{i+1}}} A_i^{\eta_i} \to M_i^{\eta_i}$, which by abuse of notation we also denote by $\varphi_i$.
		For $\cdots > \lambda_1 > \lambda_0 > 1$ with $\lambda_i < \eta_i$, define $N^\lambda = M_0^{\lambda_0} \times M_1^{\lambda_1} \times \cdots$. Also put $N = M_0 \times M_1 \times \cdots.$
		The maps $\varphi_i : M^{\lambda_{i+1}}_{i+1} \to M^{\lambda_{i}}_i$ define $\varphi^\lambda : N^\lambda \to N^\lambda$ and $\varphi_i : M_{i+1} \to M_i$ define $\phi : N \to N$. Then $R^1 \varinjlim M_i$ is identified with the cokernel of $1 - \varphi$. In order to prove that this vanishes, it suffices to prove that the cokernel of $1 - \varphi^\lambda$ vanishes for any $\lambda$. Now we are in the situation of \cite[Lemma 2.6.3]{kedlaya2016relative} and the proof is done.
	\end{proof}
	
	\begin{appproposition} \label{prop:vanish-jdag}
		Let $\Xf = \Spf A$ be an affine formal $\V[[t]]$-scheme, $f_1, \dots, f_r, g \in A$, $U = D(\bar{f}) \subseteq X$, $\Vs = \{ x \in \Xf_K : |f_i|_{[x]} > \eta_i \}$ for some $\eta_i \in (0, 1) \cap \Gamma^*.$ Then for a coherent $j^\dag_U \O$-module $\E$ over $\Ws = \Vs \times A^s (\eta, 1)$ for some $\eta \in (0, 1) \cap \Gamma^*$ or over $\Ws = \Vs \times A^s[0, 1)$, we have $H^j(\Ws, \, \E) = 0$ for $j > 0.$
	\end{appproposition}
	\begin{proof}
		$\Ws$ is quasi-Stein in both cases. The vanishing can be proved in the similar way as in \cite[Theorem 2.6.5]{kedlaya2016relative}.
	\end{proof}
	
	\section{Extension of base fields in rigid cohomology}
	Suppose $k$ is perfect. Let $K'$ be a finite extension of $K$ with residue field $k'$. For a finite extension $k'((t'))$ over $k((t))$, let $n \leq 0$ be the integer such that $k'((t'))$ is separable over $k((t^{1 / p^n}))$. Let $\Ed_K$ be the bounded Robba ring with parameter in $t$  over $K$, $\E^{(n)\dag}_K$ in $t^{1/p^n}$ over $K$ and $\E'^{\dag}_{K'}$ in $t'$ over $K'$. The separable extension $k'((t'))/k((t^{1 / p^n}))$ lifts to an extension $\E'^{\dag}_{K'} / \E^{(n)\dag}_K.$
	
	\begin{apptheorem}
		Let $X$ be a separated $k((t))$-scheme, $E$ and $\Ed_K$-valued overconvergent isocrystal over $X$, $X' = X \otimes_{k((t))} k'((t'))$, $f : X' \to X$ the projection and $E' = f^*E$ the $\E'^{\dag}_{K'}$-valued overconvergent isocrystal over $X'$. Then we have canonical isomorphisms
		$$ H^*_{\rig}(X/\Ed_K;\, E) \otimes \E'^{\dag}_{K'} \simeq H^*_{\rig}(X'/\E'^{\dag}_{K'};\, E')$$
		and
		$$ H^*_{\rig, c}(X/\Ed_K;\, E) \otimes \E'^{\dag}_{K'} \simeq H^*_{\rig, c}(X'/\E'^{\dag}_{K'};\, E')$$
	\end{apptheorem}
	\begin{proof}
		Using long exact sequence and induction on dimension, we may assume that $X$ is affine. Choose an open immersion $X \hookrightarrow Y$ into a proper $k[[t]]$-scheme and a closed immersion $Y \hookrightarrow \Pf$ into a smooth formal  $\V[[t]]$-scheme. By base change, we get an open immersion $X' \hookrightarrow Y'$ into a proper $k'[[t']]$-scheme and a closed immersion $Y' \hookrightarrow \Pf'$ into a smooth formal $\V'[[t']]$-scheme.
		
		Let $u : \Pf' \to \Pf$ be the projection. Then $u_K^{-1}(]Y[_\Pf) = \ ]Y'[_{\Pf'}$ and $u_K^{-1}(]X[_\Pf) = \  ]X'[_{\Pf'}.$ We first prove the isomorphism for $H^*_\rig.$ It suffices to prove that the canonical map
		$$\E'^{\dag}_{K'} \otimes_{\Ed_K} E \otimes \Omega^\cdot_{]X[_\Pf/\Ed_K} \to Ru_{K*} (E' \otimes \Omega^\cdot_{]X'[_{\Pf'}/\E'^{\dag}_{K'}})$$
		is an isomorphism. This is local on $\Pf$ and we may assume that $\Pf$ is affine. Again using the long exact sequence, we may furthermore assume that there is some $f \in \Gamma(\Pf, \O)$ and $X = Y \cap D(t\bar{f}).$
		Put $[X]_{\Pf, \lambda} = \{ x \in ]Y[_\Pf : |t|_{x}, |f|_{x} \geq \lambda \}$ and similarly define $[X']_{\Pf', \lambda}$ for $\lambda \in (0, 1) \cap \Gamma^*.$ Choose some $\nabla$-module $\E$ over $[X]_{\Pf, \lambda_0}$ which gives rise to $E$. Then for an affinoid open subset $\Us \subset \Pf_K$, we have
		\begin{align*}
			R\Gamma(\Us, \E'^{\dag}_{K'} \otimes_{\Ed_K} E \otimes \Omega^\cdot_{]X[_\Pf/\Ed_K}) & = \E'^{\dag}_{K'} \otimes_{\Ed_K} \varinjlim_{\lambda\to 1} R\Gamma(\Us \cap [X]_{\Pf, \lambda}, \E \otimes \Omega^\cdot_{[X]_{\Pf, \lambda_0}/\V[[t]]_K})
		\end{align*}
		Since $u_K^{-1}([X]_{\Pf, \lambda}) = [X']_{\Pf', \lambda'}$ for some $\lambda' \in (0, 1) \cap \Gamma^*$, and when $\lambda$ tends to $1$, $\lambda'$ also tends to $1$. Therefore
		\begin{align*}
			R\Gamma(\Us, Ru_{K*} (E' \otimes \Omega^\cdot_{]X'[_{\Pf'}/\E'^{\dag}_{K'}})) & = \varinjlim_{\lambda\to 1} R\Gamma(u_K^{-1}(\Us \cap [X]_{\Pf, \lambda}), u_K^* \E \otimes \Omega^\cdot_{[X']_{\Pf', \lambda'}/\V'[[t']]_{K'}})\\
			& = \varinjlim_{\lambda\to 1} \V'[[t']_{K'} \otimes_{\V[[t]]_K} R\Gamma(\Us \cap [X]_{\Pf, \lambda}, \E \otimes \Omega^\cdot_{[X]_{\Pf, \lambda_0}/\V[[t]]_{K}}) \\
			& = \V'[[t']]_{K'} \otimes_{\V[[t]]_K} \varinjlim_{\lambda\to 1} R\Gamma(\Us \cap [X]_{\Pf, \lambda}, \E \otimes \Omega^\cdot_{[X]_{\Pf, \lambda_0}/\V[[t]]_{K}}).
		\end{align*}
		Since $\V'[[t']]_{K'} \otimes_{\V[[t]]_K} \Ed_K = \E'^{\dag}_{K'},$ two complexes are isomorphic.
		
		We now prove the isomorphism for $H^*_{\rig, c}.$ For the definition of compactly supported rigid cohomology over $\Ed_K$, see \cite[4.3.1]{lazda2016rigid}. Let $Z = Y - X$, $i : \, ]Z[_\Pf \cap [X]_{\Pf, \lambda_0} \hookrightarrow [X]_{\Pf, \lambda_0}$ the open immersion. For a complex $K$ of sheaves over $[X]_{\Pf, \lambda_0}$, $R\underline{\Gamma}_{]X[_\Pf} K$ is defined to be the mapping fiber of $K \to Ri_*i^{-1} K.$ Define $Z'$, $i'$ and $R\underline{\Gamma}_{]X'[_{\Pf'}}$ similarly. It suffices to prove that the canonical map
		$$\E'^{\dag}_{K'} \otimes_{\Ed_K} R\underline{\Gamma}_{]X[_\Pf} j^\dag_{Y_\eta}(\E \otimes \Omega^\cdot_{[X]_{\Pf, \lambda_0}/\V[[t]]_K}) \to Ru_{K*} R\underline{\Gamma}_{]X'[_{\Pf'}} j^\dag_{Y'_{\eta'}} (u_K^*\E \otimes \Omega^\cdot_{]Y'[_{\Pf'}/\V'[[t']]_{K'}})$$
		is an isomorphism. By considering the spectral sequences, it suffices to prove that for a coherent $\O_{[X]_{\Pf, \lambda_0}}$-module $\F$, the canonical map
		$$\E'^{\dag}_{K'} \otimes_{\Ed_K} R\underline{\Gamma}_{]X[_\Pf} j^\dag_{Y_\eta}\F \to Ru_{K*} R\underline{\Gamma}_{]X'[_{\Pf'}} j^\dag_{Y'_{\eta'}} u_K^*\F$$
		is an isomorphism.
		Let $[Y_\eta]_{\Pf, \lambda} = {x \in \Pf_K : |t|_x \geq \lambda}$ for $\lambda \in (0, 1) \cap \Gamma^*$ and define $[Y'_{\eta'}]_{\Pf', \lambda'}$ similarly. Note that $u_K^{-1}(]Z[_{\Pf}) = ]Z'[_{\Pf'}$ and that $u_K^{-1}([Y_\eta]_{\Pf, \lambda}) = [Y'_{\eta'}]_{\Pf', \lambda'}$ for some $\lambda' \in (0, 1) \cap \Gamma^*$ and $\lambda' \to 1$ as $\lambda \to 1$. For an affinoid open subset $\Us \subseteq [X]_{\Pf, \lambda_0},$
		$$
			R\Gamma(\Us, \E'^{\dag}_{K'} \otimes_{\Ed_K} R\underline{\Gamma}_{]X[_\Pf} \F) 
			= \E'^{\dag}_{K'} \otimes_{\Ed_K} \mathrm{Cone} [ R\Gamma(\Us, j^\dag_{Y_\eta}\F) \to R\Gamma(\Us \cap ]Z[_{\Pf}, j^\dag_{Y_\eta}\F)][-1]
		$$
		and
		$$
			R\Gamma(\Us, Ru_{K*} R\underline{\Gamma}_{]X'[_{\Pf'}} j^\dag_{Y'_{\eta'}} u_K^*\F)
			= \mathrm{Cone} [ R\Gamma(u_K^{-1}(\Us), j^\dag_{Y'_{\eta'}}u_K^*\F) \to R\Gamma(u_K^{-1}(\Us \cap ]Z[_{\Pf}), j^\dag_{Y'_{\eta'}}u_K^*\F)][-1].
		$$
		By Proposition \ref{prop:vanish-jdag}, the higher cohomology of the four complexes inside the mapping cone all vanish.
		$\E'^{\dag}_{K'} \otimes_{\Ed_K} \Gamma(\Us, j^\dag_{Y_{\eta}}\F) = \Gamma(u_K^{-1}(\Us), j^\dag_{Y'_{\eta'}}u_K^*\F).$
		$\Us \cap ]Z[_\Pf$ is covered by $[Z]_{\Pf, \lambda} = \{ x \in \Us : |f|_x \leq \lambda \}$ and since $\E'^{\dag}_{K'} \otimes_{\Ed_K} \Gamma(\Us \cap [Z]_{\Pf, \lambda}, j^\dag_{Y_{\eta}}\F) = \Gamma(u_K^{-1}(\Us \cap [Z]_{\Pf, \lambda}), j^\dag_{Y'_{\eta'}}u_K^*\F)$, using the spectral sequence associated to this open cover we conclude that it is an isomorphism.
	\end{proof}
	
	\bibliographystyle{plain}
	\bibliography{main}
\end{document}